\documentclass[11pt,reqno]{amsart}
\usepackage{amsmath,amsfonts,amssymb}
\usepackage[cp1251]{inputenc}
\usepackage[english]{babel}
\usepackage{mathrsfs}
\newcommand{\esssup}{\mathop{\mathrm{ess}\:\mathrm{sup}}\limits}

\def\C{\mathbb C}

\def\x{\mathbf x}

\def\D{\mathbf D}

\def\1{\mathbf 1}

\def\1{\bold 1}

\def\eps{\varepsilon}

\def\le{\leqslant}

\def\ge{\geqslant}

\theoremstyle{theorem}
\newtheorem{theorem}{Theorem}[section]
\newtheorem{proposition}[theorem]{Proposition}
\newtheorem{lemma}[theorem]{Lemma}
\newtheorem{condition}[theorem]{Condition}
\newtheorem{remark}[theorem]{Remark}
\newtheorem{corollary}[theorem]{Corollary}
\newtheorem{example}[theorem]{Example}
\numberwithin{equation}{section}

\theoremstyle{plain}
\newtoks\thehProclaim
\newtheorem*{Proclaim}{\the\thehProclaim}

\begin{document}

\title[Two-parametric error estimates]{Two-parametric error estimates in~homogenization of~second order elliptic systems in~$\mathbb{R}^d$ including lower order terms}

\author{Yu.~M.~Meshkova and T.~A.~Suslina}

\thanks{Supported by RFBR (grant no.~14-01-00760) and SPbSU (project no.~11.38.263.2014).
The first author is supported by the Chebyshev Laboratory  (Department of Mathematics and Mechanics, St.~Petersburg State University)  under RF Government grant 11.G34.31.0026 and JSC \glqq Gazprom Neft\grqq.}


\address{Chebyshev Laboratory, St. Petersburg State University, 14th Line, 29b, Saint Petersburg, 199178, Russia}
\email{juliavmeshke@yandex.ru}

 \address{St.~Petersburg State University, 7/9 Universitetskaya nab., St.~Petersburg, 199034, Russia}

\email{t.suslina@spbu.ru,\quad suslina@list.ru}


\begin{abstract}
In $L_2({\mathbb R}^d;{\mathbb C}^n)$, we consider a selfadjoint operator ${\mathcal B}_\varepsilon$,
\hbox{$0< \varepsilon \le 1$}, given by the differential expression
$b({\mathbf D})^* g({\mathbf x}/\varepsilon)b({\mathbf D}) + \sum_{j=1}^d (a_j({\mathbf x}/\varepsilon) D_j
+D_j a_j({\mathbf x}/\varepsilon)^*) + Q({\mathbf x}/\varepsilon)$, where $b({\mathbf D}) = \sum_{l=1}^d b_l D_l$
is the first order differential operator, and $g, a_j, Q$ are matrix-valued functions in ${\mathbb R}^d$
periodic with respect to some lattice $\Gamma$.
It is assumed that $g$ is bounded and positive definite, while $a_j$ and $Q$ are, in general, unbounded.
 We study the generalized resolvent $({\mathcal B}_\varepsilon - \zeta Q_0({\mathbf x}/\varepsilon))^{-1}$, where
$Q_0$ is a $\Gamma$-periodic, bounded and positive definite matrix-valued function, and $\zeta$ is a complex-valued parameter.
Approximations for the generalized resolvent in the $(L_2 \to L_2)$- and $(L_2 \to H^1)$-norms
with two-parametric error estimates (with respect to the parameters $\varepsilon$ and $\zeta$) are obtained.
\end{abstract}

\maketitle

\tableofcontents

\section*{Introduction}
\setcounter{section}{0}
\setcounter{equation}{0}

The paper concerns homogenization theory of periodic differential operators (DO's). A broad literature is devoted to
homogenization problems. First, we mention the books \cite{BeLP,BaPa,ZhKO}.

\subsection{Statement of the problem} We study matrix elliptic DO's
$\mathcal{B}_\varepsilon$, \hbox{$0<\varepsilon \leqslant 1$}, acting in the space $L_2(\mathbb{R}^d;\mathbb{C}^n)$.
Let $\Gamma$ be a lattice in  $\mathbb{R}^d$, and let $\Omega$ be the cell of  $\Gamma$.
For $\Gamma$-periodic functions in $\mathbb{R}^d$, we use the notation $\psi ^\varepsilon (\mathbf{x}):=\psi (\varepsilon ^{-1}\mathbf{x})$
and $\overline{\psi}:=\vert \Omega \vert ^{-1}\int _\Omega \psi (\mathbf{x})\,d\mathbf{x}$.

The principal part $\mathcal{A}_\varepsilon $ of the operator $\mathcal{B}_\varepsilon$
is given in a factorized form
\begin{equation}
\label{A_eps in Introduction}
\mathcal{A}_\varepsilon = b(\mathbf{D})^*g^\varepsilon (\mathbf{x})b(\mathbf{D}),
\end{equation}
where $b(\mathbf{D})$ is a matrix homogeneous first order DO and $g(\mathbf{x})$ is a bounded and positive definite
$\Gamma$-periodic matrix-valued function in $\mathbb{R}^d$.
(The precise assumptions on $b(\mathbf{D})$ and $g(\mathbf{x})$ are given below in
Subsection~\ref{Subsection Operator A}.)
The simplest example of the operator $\mathcal{A}_\varepsilon$ is the acoustics operator
$-\mathrm{div}\,g^\varepsilon (\mathbf{x})\nabla$; the operator of elasticity theory also can be represented in the required form.
The homogenization problem for the operator $\mathcal{A}_\varepsilon$ was studied in a~series of papers
\cite{BSu,BSu05,BSu06} and also in \cite{Su14,Su15} in detail.
In the present paper we consider a more general selfadjoint DO $\mathcal{B}_\varepsilon$ including lower order terms:
\begin{equation}
\label{B_eps from Introduction}
\mathcal{B}_\varepsilon =b(\mathbf{D})^*g^\varepsilon (\mathbf{x})b(\mathbf{D})+\sum _{j=1}^d \left( a_j^\varepsilon (\mathbf{x})D_j+D_ja_j^\varepsilon (\mathbf{x})^*\right) +Q^\varepsilon (\mathbf{x}).
\end{equation}
Here $a_j(\mathbf{x})$ are $\Gamma$-periodic matrix-valued functions, in general, unbounded.
In general,  the coefficient $Q(\mathbf{x})$ is a distribution generated by a periodic matrix-valued measure.
(The precise assumptions on the coefficients are given below in Subsections \ref{Subsection Operators Y}, \ref{Subsection form q}.)
The precise definition of the operator $\mathcal{B}_\varepsilon$ is given in terms of the corresponding quadratic form.

The coefficients of  the differential expression \eqref{B_eps from Introduction} oscillate rapidly as \hbox{$\varepsilon \rightarrow 0$}.
{\it A typical homogenization problem for the operator $\mathcal{B}_\varepsilon$ is to approximate the resolvent
$(\mathcal{B}_\varepsilon -\zeta I)^{-1}$ or the generalized resolvent \hbox{$(\mathcal{B}_\varepsilon -\zeta Q_0^\varepsilon )^{-1}$} for small $\varepsilon$.} Here $Q_0(\mathbf{x})$ is a positive definite and bounded $\Gamma$-periodic matrix-valued function.

\subsection{A survey of the results on the operator error estimates} The homogenization problem
for the operator $\mathcal{A}_\varepsilon$ was studied in a series of papers \cite{BSu,BSu05,BSu06} by M.~Sh.~Birman and T.~A.~Suslina.
In \cite{BSu}, it was proved that
\begin{equation}
\label{A_eps appr in L2 Introduction}
\Vert (\mathcal{A}_\varepsilon +I)^{-1}-(\mathcal{A}^0 +I)^{-1}\Vert _{L_2(\mathbb{R}^d)\rightarrow L_2(\mathbb{R}^d)}\leqslant C\varepsilon .
\end{equation}
Here $\mathcal{A}^0=b(\mathbf{D})^*g^0b(\mathbf{D})$ is the {\it effective operator} with a constant {\it effective matrix }$g^0$.
The definition of the effective matrix (see Subsection~\ref{Subsection Effective matrix} below) is well known
in homogenization theory. Next, in \cite{BSu06} approximation of the resolvent
$(\mathcal{A}_\varepsilon +I)^{-1}$ in the norm of operators acting from $L_2(\mathbb{R}^d;\mathbb{C}^n)$ to the Sobolev space  $H^1(\mathbb{R}^d;\mathbb{C}^n)$ was obtained:
\begin{equation}
\label{A_eps appr L_2->H^1 Introduction}
\Vert (\mathcal{A}_\varepsilon +I)^{-1}-(\mathcal{A}^0 +I)^{-1}-\varepsilon \mathcal{K}(\varepsilon )\Vert _{L_2(\mathbb{R}^d)\rightarrow H^1(\mathbb{R}^d)}\leqslant C\varepsilon .
\end{equation}
Here $\mathcal{K}(\varepsilon)$ is a {\it corrector}. The corrector has zero order with respect to $\varepsilon$,
but it  involves  rapidly oscillating factors. Therefore,
$\Vert \mathcal{K}(\varepsilon )\Vert _{L_2\rightarrow H^1}=O(\varepsilon ^{-1})$.
Estimates of the form \eqref{A_eps appr in L2 Introduction} and \eqref{A_eps appr L_2->H^1 Introduction}
are called {\it operator error estimates}; they are order-sharp, and the constants in estimates are controlled explicitly
in terms of the problem data. The method of \cite{BSu,BSu05,BSu06} is based on the scaling transformation, the Floquet-Bloch theory,
and the analytic perturbation theory.

Later the method of \cite{BSu,BSu05,BSu06} was developed by T.~A.~Suslina \cite{Su08,SuAA,SuAA14}
for the operator \eqref{B_eps from Introduction}. In \cite{SuAA}, the following analogs of estimates
\eqref{A_eps appr in L2 Introduction} and \eqref{A_eps appr L_2->H^1 Introduction} were obtained:
\begin{align}
\label{B_eps approx in L_2 by Su Introduction}
\Vert & (\mathcal{B}_\varepsilon +\lambda Q_0^\varepsilon )^{-1}-(\mathcal{B}^0+\lambda \overline{Q_0})^{-1}\Vert _{L_2(\mathbb{R}^d)\rightarrow L_2(\mathbb{R}^d)}\leqslant C\varepsilon ,\\
\label{B_eps approx L_2->H^1 by Su Introduction}
\Vert & (\mathcal{B}_\varepsilon +\lambda Q_0^\varepsilon )^{-1}-(\mathcal{B}^0+\lambda \overline{Q_0})^{-1}-\varepsilon K(\varepsilon )\Vert _{L_2(\mathbb{R}^d)\rightarrow H^1(\mathbb{R}^d)}\leqslant C\varepsilon .
\end{align}
Here the real-valued parameter $\lambda$ is such that the operator
 $\mathcal{B}_\varepsilon +\lambda Q_0^\varepsilon$ is positive definite; $\mathcal{B}^0$ is the corresponding effective operator
with constant coefficients. Estimates \eqref{B_eps approx in L_2 by Su Introduction} and \eqref{B_eps approx L_2->H^1 by Su Introduction}
are order-sharp, and the constants in estimates are controlled explicitly in terms of the problem data
and $\vert \lambda \vert $ (however, in the cited papers
the optimal dependence of the constants in estimates on the spectral parameter $\lambda$ was not searched out).

A different approach to operator error estimates in homogenization theory was suggested by V.~V.~Zhikov.
In  \cite{Zh1,Zh2} and \cite{ZhPas}, estimates of the form  \eqref{A_eps appr in L2 Introduction}
and \eqref{A_eps appr L_2->H^1 Introduction} for the acoustics operator and the operator of elasticity theory were obtained.
The method was based on analysis of the first order approximation to the solution and introduction of an additional parameter.
 Besides the problems in $\mathbb{R}^d$, in \cite{Zh1,Zh2,ZhPas} homogenization problems for elliptic equations
in a bounded domain $\mathcal{O}\subset \mathbb{R}^d$ with the Dirichlet or Neumann boundary conditions
were studied.

Also, operator error estimates for the Dirichlet and Neumann problems
for a second order elliptic equation in a bounded domain (without lower order terms)
have been studied by different methods in the papers \cite{Gr1, Gr2}, \cite{KeLiS},
\cite{PSu}, \cite{Su13,Su_SIAM,Su14,Su15}; see a detailed survey in the introduction to \cite{Su_SIAM,Su15}.

In the presence of the lower order terms,  homogenization problem for the operator of the form
\eqref{B_eps from Introduction} in ${\mathbb R}^d$ was studied in the paper \cite{Bo} by D.~I.~Borisov.
In \cite{Bo}, expression for the effective operator $\mathcal{B}^0$ was found and error estimates of the form
\eqref{B_eps approx in L_2 by Su Introduction} and \eqref{B_eps approx L_2->H^1 by Su Introduction} were obtained.
Moreover, it was assumed that the coefficients of the operator depend
not only on the rapid variable, but also on the slow variable.
However, in \cite{Bo} it was assumed that the coefficients are sufficiently smooth.

Up to now, we have discussed the results about approximation of the resolvent at a fixed regular point.
Approximations of the operator \hbox{$(\mathcal{A}_\varepsilon -\zeta I)^{-1}$}
with error estimates depending on $\varepsilon$ and $\zeta =\vert \zeta \vert e^{i\phi}\in \mathbb{C}\setminus\mathbb{R}_+$
were obtained in the recent papers \cite{Su14,Su15}. It was proved that
\begin{align}
\label{A_eps appr 2 parameters in L2 by Su Introduction}
\Vert &(\mathcal{A}_\varepsilon -\zeta I)^{-1}-(\mathcal{A}^0-\zeta I)^{-1}\Vert _{L_2(\mathbb{R}^d)\rightarrow L_2(\mathbb{R}^d)}\leqslant C(\phi)\vert \zeta \vert ^{-1/2}\varepsilon ,\\
\label{A_eps appr 2 parameters in L2->H1 by Su Introduction}
\Vert &(\mathcal{A}_\varepsilon -\zeta I)^{-1}-(\mathcal{A}^0 -\zeta I)^{-1}-\varepsilon \mathcal{K}(\varepsilon ;\zeta )\Vert _{L_2(\mathbb{R}^d)\rightarrow H^1(\mathbb{R}^d)}\leqslant C(\phi)(1+\vert \zeta \vert ^{-1/2})\varepsilon.
\end{align}
The dependence of the constants in estimates
\eqref{A_eps appr 2 parameters in L2 by Su Introduction} and \eqref{A_eps appr 2 parameters in L2->H1 by Su Introduction}
on the angle $\phi$ was traced. Estimates   \eqref{A_eps appr 2 parameters in L2 by Su Introduction}
and \eqref{A_eps appr 2 parameters in L2->H1 by Su Introduction} are two-parametric (with respect to $\varepsilon$ and $\vert \zeta \vert$);
they are uniform with respect to $\phi$ in any sector $\phi \in [\phi _0 ,2\pi -\phi _0]$ with arbitrarily small $\phi _0>0$.
Also, in \cite{Su14,Su15}  the operators  $\mathcal{A}_{D,\varepsilon}$ and $\mathcal{A}_{N,\varepsilon}$
given by the expression \eqref{A_eps in Introduction} in a bounded domain with the Dirichlet or Neumann boundary conditions
were studied. Approximations of the resolvents of these operators with two-parametric error estimates
were found.

Investigation  of the two-parametric estimates of the form
\eqref{A_eps appr 2 parameters in L2 by Su Introduction} and \eqref{A_eps appr 2 parameters in L2->H1 by Su Introduction}
was stimulated by the study of homogenization for the parabolic problems.
This study is based on the following representation for the operator exponential:
\begin{equation*}
e^{-\mathcal{A}_{\dag ,\varepsilon}t}=-(2\pi i)^{-1}\int _\gamma e^{-\zeta t}(\mathcal{A}_{\dag ,\varepsilon}-\zeta I)^{-1}\,d\zeta ,\quad \dag =D,N,
\end{equation*}
where $\gamma \subset \mathbb{C}$ is a contour on the complex plane enclosing the spectrum of the operator
$\mathcal{A}_{\dag ,\varepsilon}$ (in the positive direction). See details in  \cite{MSu1,MSu2}.

\subsection{Main results} Before we formulate the results, it is convenient to turn to the
nonnegative operator $B_\varepsilon =\mathcal{B}_\varepsilon +cQ_0^\varepsilon $, choosing an appropriate
constant $c$. Then the corresponding effective operator is $B^0=\mathcal{B}^0+c\overline{Q_0}$. {\it Our goal} is to approximate
the generalized resolvent $(B_\varepsilon -\zeta Q_0^\varepsilon )^{-1}$ with error estimates depending on $\varepsilon$
and the spectral parameter $\zeta$.

Our main results are the following estimates:
\begin{align}
\label{main result L2 Introduction}
\Vert &(B_\varepsilon -\zeta Q_0^\varepsilon )^{-1}-(B^0-\zeta \overline{Q_0})^{-1}\Vert _{L_2(\mathbb{R}^d)\rightarrow L_2(\mathbb{R}^d)}\leqslant C(\phi)\varepsilon \vert \zeta \vert ^{-1/2},\\
\label{main result L2->H1 Introduction}
\begin{split}
\Vert  &(B_\varepsilon -\zeta Q_0^\varepsilon )^{-1}-(B^0-\zeta \overline{Q_0})^{-1}-\varepsilon K(\varepsilon ;\zeta )\Vert _{L_2(\mathbb{R}^d)\rightarrow H^1(\mathbb{R}^d)}
\leqslant C(\phi) \varepsilon ,
\end{split}
\end{align}
for $\zeta \in \mathbb{C}\setminus\mathbb{R}_+$, $\vert \zeta \vert \geqslant 1$, and $0<\varepsilon\leqslant 1$.
The dependence of the constants in estimates on the angle $\phi=\mathrm{arg}\,\zeta$ is traced.
The two-parametric estimates \eqref{main result L2 Introduction} and \eqref{main result L2->H1 Introduction}
are uniform with respect to $\phi$ in any domain \hbox{$\lbrace \zeta =\vert \zeta\vert e^{i\phi} \in \mathbb{C} : \vert \zeta \vert \geqslant 1 , \phi _0 \leqslant\phi \leqslant 2\pi -\phi _0\rbrace$}
with arbitrarily small $\phi _0>0$.

In the general case, the corrector in  \eqref{main result L2->H1 Introduction} contains a smoothing operator.
We distinguish the cases where a simpler corrector can be used.

Also, we find approximation in the $(L_2\rightarrow L_2)$-norm for
the operator $g^\varepsilon b(\mathbf{D})(B_\varepsilon -\zeta Q_0^\varepsilon )^{-1}$ corresponding to the flux.

Estimates \eqref{main result L2 Introduction} and \eqref{main result L2->H1 Introduction}
are generalizations of estimates \eqref{A_eps appr 2 parameters in L2 by Su Introduction} and
\eqref{A_eps appr 2 parameters in L2->H1 by Su Introduction} from \cite{Su14,Su15}
for the case of the operator $B_\varepsilon$ including the lower order terms.
However, there is a difference:
estimates \eqref{A_eps appr 2 parameters in L2 by Su Introduction} and \eqref{A_eps appr 2 parameters in L2->H1 by Su Introduction}
are valid for all $\zeta \in \mathbb{C}\setminus\mathbb{R} _+$, while
estimates  \eqref{main result L2 Introduction} and \eqref{main result L2->H1 Introduction}
are proved under the additional assumption that $\vert \zeta \vert \geqslant 1$.
This is related to the presence of the lower order terms.

For completeness, we find approximation of the operator \hbox{$(B_\varepsilon -\zeta Q_0^\varepsilon )^{-1}$}
in a wider domain of the parameter $\zeta$; the corresponding error estimates have a different behavior
with respect to $\zeta$. (See Section \ref{Section another approximation} below.)

\subsection{Method}
The method is based on application of the known results at a fixed point, the scaling transformation, and
appropriate identities for generalized resolvents.
We consider  the auxiliary operator family $B_\varepsilon (\vartheta)$ depending on the additional parameter
$0 < \vartheta\leqslant 1$.  The operator $B_\varepsilon (\vartheta)$ is given by
$B_\varepsilon (\vartheta )=b (\mathbf{D})^*g^\varepsilon b(\mathbf{D})+\vartheta\sum _{j=1}^d \left(a_j^\varepsilon D_j +D_j(a_j^\varepsilon)^*
\right)+\vartheta ^2 ({Q}^\varepsilon + c Q_0^\varepsilon ).$
Estimates \eqref{B_eps approx in L_2 by Su Introduction} and \eqref{B_eps approx L_2->H^1 by Su Introduction}
at a fixed point $\lambda$ are valid for  $(B_\varepsilon (\vartheta )+\lambda Q_0^\varepsilon)^{-1}$ with the common constants for all $0< \vartheta\leqslant 1$.

Let us discuss the proof of estimate \eqref{main result L2 Introduction}.
Using \eqref{B_eps approx in L_2 by Su Introduction} for  $B_\varepsilon (\vartheta )$, by the scaling transformation,
we obtain
\begin{equation}
\label{B(eps;theta) approx L2 Introduction}
\begin{split}
\Vert & (B(\varepsilon ;\vartheta)+\lambda \varepsilon ^2 Q_0)^{-1}-(B^0(\varepsilon ;\vartheta )+\lambda\varepsilon ^2 \overline{Q_0})^{-1}\Vert _{L_2(\mathbb{R}^d)\rightarrow L_2(\mathbb{R}^d)}\\
&\leqslant C\varepsilon ^{-1},\quad 0 < \vartheta\leqslant 1, \quad 0<\varepsilon\leqslant \vartheta ^{-1}.
\end{split}
\end{equation}
(For $1<\varepsilon\leqslant\vartheta ^{-1}$ we use rather rough estimates.) Here
\begin{equation*}
\begin{split}
B(\varepsilon ;\vartheta )&=b (\mathbf{D})^*g(\mathbf{x})b(\mathbf{D})+\vartheta\varepsilon\sum _{j=1}^d \left(a_j(\mathbf{x})D_j +D_ja_j(\mathbf{x})^*\right)\\
&+\vartheta ^2\varepsilon ^2 Q(\mathbf{x})+\vartheta ^2\varepsilon ^2 c\, Q_0(\mathbf{x}),
\end{split}
\end{equation*}
and the operator $B^0(\varepsilon ;\vartheta )$ is obtained from the effective operator $B^0$ in a~similar way.
By an appropriate identity for the generalized resolvents, we carry  estimate
\eqref{B(eps;theta) approx L2 Introduction} over to the point $\widehat{\zeta}=e^{i\phi}$, $\phi \in (0,2\pi)$:
\begin{equation*}
\Vert (B(\varepsilon ;\vartheta )-\widehat{\zeta}\varepsilon ^2 Q_0)^{-1}-(B^0(\varepsilon ;\vartheta )-\widehat{\zeta}\varepsilon ^2\overline{Q_0})^{-1}\Vert _{L_2(\mathbb{R}^d)\rightarrow L_2(\mathbb{R}^d)}\leqslant C(\phi)\varepsilon ^{-1}.
\end{equation*}
In this inequality, we put $\varepsilon =\widetilde{\varepsilon}\vert \zeta\vert ^{1/2}$ and $\vartheta =\vert\zeta\vert ^{-1/2}$, where \hbox{$0<\widetilde{\varepsilon}\leqslant 1$}.
The restriction $\vert\zeta\vert\geqslant 1$ ensures that \hbox{$0< \vartheta\leqslant 1$}.
By the inverse scaling transformation, renaming $\widetilde{\varepsilon}=:\varepsilon$, we arrive at \eqref{main result L2 Introduction}.
Estimate \eqref{main result L2->H1 Introduction} is checked by the same method.

The trick based on the scaling transformation and the appropriate resolvent identities
was applied before in \cite{Su14,Su15}.
In the present paper, we use one more trick, namely, introduction of the additional parameter $\vartheta$;
this is related to the presence of the lower order terms.

The authors plan to apply the results of the present paper
to homogenization of elliptic systems in a bounded domain $\mathcal{O}\subset\mathbb{R}^d$;
a separate paper \cite{MSu_prep} is in preparation.

\subsection{Plan of the paper} The paper consists of nine sections.
In Section\ref{Section Preliminaries}, the class of operators is introduced, the effective operator is described, and
the results of the paper \cite{SuAA} are formulated. Section \ref{Section Lemmas} contains the auxiliary material.
In Section \ref{Section Steklov smoothing}, from  \eqref{B_eps approx L_2->H^1 by Su Introduction} we deduce a similar estimate, but with
a different smoothing operator in the corrector (we turn to the Steklov smoothing, since it is more convenient
for further investigation of the problems in a bounded domain). Main results are formulated in Section \ref{Section main results}.
The proof of estimate \eqref{main result L2 Introduction} is given in Section  \ref{Section proof of main result L2}.
Section \ref{Section Proof of main result in H1} contains the proof of estimate \eqref{main result L2->H1 Introduction}
and the proof of approximation for the ,,flux''\,$g^\varepsilon b(\mathbf{D})(B_\varepsilon -\zeta Q_0^\varepsilon )^{-1}$.
In Section \ref{section_special}, we distinguish the cases where the smoothing operator in the corrector can be removed,
and discuss some special cases. Approximations of the generalized resolvent for a wider domain of $\zeta$
are obtained in Section~\ref{Section another approximation}.
In Section~\ref{Section Applications}, some applications of the general results are discussed. The scalar elliptic operator
of the form
$$
{\mathcal B}_\varepsilon = ({\mathbf D}- {\mathbf A}^\varepsilon({\mathbf x}))^* g^\varepsilon({\mathbf x})
({\mathbf D}- {\mathbf A}^\varepsilon({\mathbf x}))
+ \varepsilon^{-1}  v^\varepsilon({\mathbf x}) + {\mathcal V}^\varepsilon({\mathbf x})
$$
is considered; it can be treated as the periodic Schr\"odinger operator with rapidly oscillating metric $g^\varepsilon$,
magnetic potential ${\mathbf A}^\varepsilon$, and electric potential
$\varepsilon^{-1}  v^\varepsilon + {\mathcal V}^\varepsilon$ containing the singular first term.
Also, the periodic Schr\"odinger operator involving a strongly singular potential $\varepsilon^{-2}  \check{v}^\varepsilon$ is studied.

\subsection{Notation} Let $\mathfrak{H}$ and $\mathfrak{H}_*$ be complex separable Hilbert spaces. The symbols
$(\cdot ,\cdot )_\mathfrak{H}$ and $\Vert \cdot \Vert _{\mathfrak{H}}$ stand for the inner product and the norm in $\mathfrak{H}$;
the symbol $\Vert \cdot\Vert _{\mathfrak{H}\rightarrow \mathfrak{H}_*}$ denotes the norm of a continuous linear operator acting
from $\mathfrak{H}$ to $\mathfrak{H}_*$.

The symbols $\langle\cdot ,\cdot\rangle$ and $\vert\cdot\vert$ denote the inner product and the norm in $\mathbb{C}^n$;
$\mathbf{1}_n$ is the identity $(n\times n)$-matrix.
For $z\in\mathbb{C}$, we write  $z^*$ for the complex conjugate number. (This nonstandard
notation is employed because $\overline{g}$ denotes the mean value of a periodic function $g({\mathbf x})$.)
If $a$ is an $(m\times n)$-matrix, then $\vert a\vert$ denotes the norm
of the corresponding operator from $\mathbb{C}^n$ to $\mathbb{C}^m$.
We use the notation $\mathbf{x}=(x_1,\dots,x_d)\in \mathbb{R}^d$, $iD_j=\partial _j=\partial /\partial x_j$, $j=1,\dots ,d$, $\mathbf{D}=-i\nabla=(D_1,\dots ,D_d)$. The $L_p$-classes of  $\mathbb{C}^n$-valued functions in a domain $\mathcal{O}\subset\mathbb{R}^d$ are denoted
by $L_p(\mathcal{O};\mathbb{C}^n)$, $1\leqslant p\leqslant \infty$. The Sobolev classes of $\mathbb{C}^n$-valued functions
in a domain $\mathcal{O}$ are denoted by $H^s(\mathcal{O};\mathbb{C}^n)$.
If $n=1$, we write simply $L_p(\mathcal{O})$, $H^s(\mathcal{O})$, but sometimes we use such abbreviated
notation also for spaces of vector-valued or matrix-valued functions.

We denote $\mathbb{R}_+=[0,\infty)$.
Various constants in estimates are denoted by the symbols
$c$, $\mathfrak c$, $C$, $\mathfrak{C}$ (possibly, with indices and marks).

\section{The class of operators. Approximation of the generalized resolvent $(B_\varepsilon + \lambda_0 Q_0^\varepsilon)^{-1}$}
\label{Section Preliminaries}
\setcounter{section}{1}
\setcounter{equation}{0}
In this section, we introduce the class of operators under consideration, describe the effective operator, and
formulate the results of the paper \cite{SuAA}.

\subsection{Lattices in $\mathbb{R}^d$}\label{Subsection Lattices} Let $\Gamma\subset \mathbb{R}^d$ be a lattice generated by
a basis $\mathbf{a}_1,\dots ,\mathbf{a}_d$:
\begin{equation*}
\Gamma=\left\lbrace \mathbf{a}\in \mathbb{R}^d \,:\,  \mathbf{a}=\sum \limits _{j=1}^d \nu _j \mathbf{a}_j,\; \nu _j\in\mathbb{Z}\right\rbrace .
\end{equation*}
Let $\Omega$ be the elementary cell of the lattice $\Gamma$:
\begin{equation*}
\Omega = \left \lbrace \mathbf{x}\in\mathbb{R}^d \,:\, \mathbf{x}=\sum \limits _{j=1}^d \tau _j\mathbf{a}_j,\; -\frac{1}{2}<\tau _j <\frac{1}{2}\right \rbrace .
\end{equation*}
We denote $\vert \Omega \vert= \textnormal{meas}\,\Omega$ and $2 r_1 = \textnormal{diam}\,\Omega$.

The basis $\mathbf{b}_1,\dots ,\mathbf{b}_d\in \mathbb{R}^d$ dual to the basis $\mathbf{a}_1,\dots ,\mathbf{a}_d$
is defined by the relations $\langle \mathbf{b}_j,\mathbf{a}_i\rangle =2\pi \delta _{ji}$, where $\delta _{ji}$
is the Kronecker delta. The lattice
\begin{equation*}
\widetilde{\Gamma}=\left\lbrace \mathbf{b}\in\mathbb{R}^d\,:\,\mathbf{b}=\sum\limits _{j=1}^d \mu _j\mathbf{b}_j,\;
\mu _j\in \mathbb{Z}\right\rbrace
\end{equation*}
generated by the dual basis is called the lattice dual to $\Gamma$.
We denote by $\widetilde{\Omega}$  the central Brillouin zone of the lattice $\widetilde{\Gamma}$:
\begin{equation*}
\widetilde{\Omega}=\left \lbrace \mathbf{k}\in\mathbb{R}^d \,:\,\vert \mathbf{k}\vert <\vert \mathbf{k}-\mathbf{b}\vert ,\;0\neq \mathbf{b}\in \widetilde{\Gamma}\right\rbrace.
\end{equation*}
Note that $\widetilde{\Omega}$ is a fundamental domain of $\widetilde{\Gamma}$.
Let $r_0$ be the radius of the ball inscribed in $\mathrm{clos}\,\widetilde{\Omega}$, i.~e.,
$2r_0=\min_{0\ne {\mathbf b} \in \widetilde{\Gamma}} |{\mathbf b}|$.

For $\Gamma$-periodic measurable matrix-valued functions, we systematically use the following notation:
$$
\begin{aligned}
&f^\varepsilon (\mathbf{x}):=f(\mathbf{x}/\varepsilon),\ \varepsilon >0;
\cr
&\overline{f}:= |\Omega|^{-1} \int_\Omega f({\mathbf x})\,d{\mathbf x}, \quad
\underline{f}:= \biggl(|\Omega|^{-1} \int_\Omega f({\mathbf x})^{-1}\,d{\mathbf x}\biggr)^{-1}.
\end{aligned}
$$
Here, in the definition of $\overline{f}$ it is assumed that $f \in L_{1,\text{loc}}({\mathbb R}^d)$,
and in the definition of $\underline{f}$ is is assumed that the matrix $f$ is square and nondegenerate, and $f^{-1} \in L_{1,\text{loc}}({\mathbb R}^d)$. Let $[f^\varepsilon ]$ denote the operator of multiplication by the matrix-valued function $f^\varepsilon (\mathbf{x})$.

By $\widetilde{H}^1(\Omega)$ we denote the subspace of all functions in $H^1(\Omega)$
whose $\Gamma$-periodic extension to $\mathbb{R}^d$ belongs to $H^1_{\textnormal{loc}}(\mathbb{R}^d)$.

\subsection{The smoothing operator $\Pi _\varepsilon$}
Let $\Pi _\varepsilon ^{(k)}$, $\varepsilon >0$, be the  pseudodifferential operator with the symbol $\chi _{\widetilde{\Omega}/\varepsilon}(\boldsymbol{\xi})$ acting in $L_2(\mathbb{R}^d;\mathbb{C}^k)$ (where $k \in {\mathbb N}$), i.~e.,
\begin{equation}
\label{Pi_eps}
\begin{split}
(\Pi _\varepsilon ^{(k)} \mathbf{u})(\mathbf{x})=(2\pi )^{-d/2}\int _{\widetilde{\Omega}/\varepsilon}e^{i\langle \mathbf{x},\boldsymbol{\xi}\rangle }\widehat{\mathbf{u}}(\boldsymbol{\xi})\,d\boldsymbol{\xi},\quad \mathbf{u}\in L_2(\mathbb{R}^d;\mathbb{C}^k).
\end{split}
\end{equation}
Here $\widehat{\mathbf{u}}$ is the Fourier-image of the function $\mathbf{u}$.
Obviously, for  $\mathbf{u}\in H^s(\mathbb{R}^d;\mathbb{C}^k)$
we have $\Pi_\varepsilon ^{(k)}\mathbf{D}^\alpha \mathbf{u} =\mathbf{D}^\alpha \Pi _\varepsilon ^{(k)}\mathbf{u}$ for any
multiindex $\alpha$ such that $\vert \alpha \vert \leqslant s$.
In what follows, we drop the index $k$ in the notation $\Pi _\varepsilon ^{(k)}$ and write simply $\Pi_\varepsilon$.

Below we need the following properties of the operator $\Pi _\varepsilon$ proved in \cite[Proposition~1.4]{PSu} and \cite[Subsection~10.2]{BSu06}.

\begin{proposition}
\label{Proposition Pi_eps -I}
For any function $\mathbf{u}\in H^1(\mathbb{R}^d;\mathbb{C}^k)$ and $\varepsilon >0$ we have
\begin{equation*}
\Vert \Pi _\varepsilon \mathbf{u}-\mathbf{u}\Vert _{L_2(\mathbb{R}^d)}\leqslant \varepsilon r_0^{-1}\Vert \mathbf{D}\mathbf{u}\Vert _{L_2(\mathbb{R}^d)}.
\end{equation*}
\end{proposition}

\begin{proposition}
\label{Proposition Pi_eps L2->L2}
Let $f$ be a $\Gamma$-periodic function in  $\mathbb{R}^d$ such that $f\in L_2(\Omega)$.
Then the operator $[f^\varepsilon ]\Pi_\varepsilon$ is continuous in $L_2(\mathbb{R}^d;\mathbb{C}^k)$,
and
\begin{equation*}
\Vert [f^\varepsilon ]\Pi _\varepsilon \Vert_{L_2(\mathbb{R}^d)\rightarrow L_2(\mathbb{R}^d)}\leqslant \vert \Omega \vert ^{-1/2}\Vert f\Vert _{L_2(\Omega)},\quad \varepsilon >0.
\end{equation*}
\end{proposition}

\subsection{The operator $\mathcal{A}$}
\label{Subsection Operator A} In  $L_2(\mathbb{R}^d;\mathbb{C}^n)$,
we consider the operator $\mathcal{A}$ given formally by the differential expression
\begin{equation*}
\mathcal{A} =b(\mathbf{D})^*g(\mathbf{x})b(\mathbf{D}),\quad \mathbf{x}\in\mathbb{R}^d.
\end{equation*}
Here $g$ is a $\Gamma$-periodic Hermitian matrix-valued function of size $m\times m$, in general, with complex entries.
It is assumed that
\begin{equation}
\label{g}
g(\mathbf{x})>0,\quad g,g^{-1}\in L_\infty (\mathbb{R}^d).
\end{equation}
Next, $b(\mathbf{D})$ is the first order DO with constant
coefficients given by
\begin{equation}
\label{b(D)=}
b(\mathbf{D})=\sum \limits _{l=1}^d b_lD_l.
\end{equation}
Here $b_l$, $l=1,\dots ,d$, are constant $(m\times n)$-matrices, in general, with complex entries.
Let $b(\boldsymbol{\xi})=\sum _{l=1}^d b_l\xi _l $ be the symbol
of the operator $b(\mathbf{D})$. We assume that $m\geqslant n$ and
\begin{equation*}
\textnormal{rank}\,b(\boldsymbol{\xi})=n,\quad 0\neq \boldsymbol{\xi}\in \mathbb{R}^d.
\end{equation*}
This condition is equivalent to the existence of positive constants $\alpha _0$ and $\alpha _1$ such that
\begin{equation}
\label{<b^*b<}
\alpha _0\mathbf{1}_n \leqslant b(\boldsymbol{\theta})^*b(\boldsymbol{\theta})
\leqslant \alpha _1\mathbf{1}_n,\quad
\boldsymbol{\theta}\in \mathbb{S}^{d-1},\quad
0<\alpha _0\leqslant \alpha _1<\infty.
\end{equation}
From  \eqref{<b^*b<} it follows that
\begin{equation}
\label{bl<=}
\vert b_l\vert \leqslant\alpha _1^{1/2},\quad l=1,\dots ,d.
\end{equation}

The precise definition is the following:  $\mathcal{A}$ is the selfadjoint operator in
$L_2(\mathbb{R}^d;\mathbb{C}^n)$ generated by the quadratic form
\begin{equation*}
\mathfrak{a} [\mathbf{u},\mathbf{u}]=\int _{\mathbb{R}^d}\langle g (\mathbf{x})b(\mathbf{D})\mathbf{u}(\mathbf{x}),b(\mathbf{D})\mathbf{u}(\mathbf{x})\rangle \,d\mathbf{x},\quad \mathbf{u}\in H^1(\mathbb{R}^d;\mathbb{C}^n).
\end{equation*}
The form $\mathfrak{a}$ is closed and nonnegative because of the estimates
\begin{equation}
\label{<a<}
\begin{split}
\alpha _0 \Vert g^{-1}\Vert _{L_\infty}^{-1}\int _{\mathbb{R}^d}\vert \mathbf{D}\mathbf{u}(\mathbf{x})\vert ^2\,d\mathbf{x}\leqslant \mathfrak{a} [\mathbf{u},\mathbf{u}]\leqslant \alpha _1 \Vert g\Vert _{L_\infty}\int _{\mathbb{R}^d}\vert \mathbf{D}\mathbf{u}(\mathbf{x})\vert ^2\,d\mathbf{x},\\ \mathbf{u}\in H^1(\mathbb{R}^d;\mathbb{C}^n),
\end{split}
\end{equation}
that follow from \eqref{g} and \eqref{<b^*b<}.

\subsection{The operators $\mathcal{Y}$ and $\mathcal{Y}_{2 }$}
\label{Subsection Operators Y}
We consider the closed operator $\mathcal{Y}$ acting from $L_2(\mathbb{R}^d;\mathbb{C}^n)$ to $L_2(\mathbb{R}^d;\mathbb{C}^{dn})$
and defined by
\begin{equation*}
\mathcal{Y}\mathbf{u}=\mathbf{D}\mathbf{u}=\textnormal{col}\,\lbrace D_1\mathbf{u},\dots ,D_d\mathbf{u}\rbrace ,\quad \mathbf{u}\in H^1(\mathbb{R}^d;\mathbb{C}^n).
\end{equation*}
The lower estimate \eqref{<a<} means that
\begin{equation}
\label{Y<=}
\Vert \mathcal{Y}\mathbf{u}\Vert ^2 _{L_2(\mathbb{R}^d)}\leqslant c_1^2\mathfrak{a}[\mathbf{u},\mathbf{u}],\quad c_1=\alpha _0^{-1/2}\Vert g^{-1}\Vert ^{1/2}_{L_\infty},\quad \mathbf{u}\in H^1(\mathbb{R}^d;\mathbb{C}^n).
\end{equation}

Let $a_j(\mathbf{x})$, $j=1,\dots ,d$,
be $\Gamma $-periodic $(n\times n)$-matrix-valued functions in $\mathbb{R}^d$ (in general, with complex entries) such that
\begin{equation}
\label{a_j}
a_j\in L_\rho (\Omega),\quad \rho =2\enskip\text{for}\enskip d=1,\quad
\rho >d\enskip\text{for}\enskip d\geqslant 2;\quad j=1,\dots ,d.
\end{equation}
We consider the operator $\mathcal{Y}_{2}:L_2(\mathbb{R}^d;\mathbb{C}^{n})\rightarrow L_2(\mathbb{R}^d;\mathbb{C}^{dn})$
that acts as multiplication by the $(dn\times n)$-matrix-valued function consisting of the matrices
$a_j (\mathbf{x})^*$, $j=1,\dots ,d$. In other words,
\begin{equation*}
\mathcal{Y}_{2}\mathbf{u}(\mathbf{x})=\textnormal{col}\,\lbrace a_1 (\mathbf{x})^*\mathbf{u}(\mathbf{x}),\dots ,a_d (\mathbf{x})^*\mathbf{u}(\mathbf{x})\rbrace ,\quad \mathbf{u}\in H^1(\mathbb{R}^d;\mathbb{C}^n).
\end{equation*}
Using the H\"older inequality and the Sobolev embedding theorem,
it is easy to check (see \cite[(5.11)--(5.14)]{SuAA}) that for any $\nu>0$ there exist constants $C_j(\nu)>0$ such that
\begin{equation*}
\begin{split}
\Vert a_j^*\mathbf{u}\Vert ^2 _{L_2(\mathbb{R}^d)}\leqslant \nu \Vert \mathbf{D}\mathbf{u}\Vert ^2_{L_2(\mathbb{R}^d)}+C_j(\nu )\Vert \mathbf{u}\Vert ^2 _{L_2(\mathbb{R}^d)},\\ \mathbf{u}\in H^1(\mathbb{R}^d;\mathbb{C}^n),\;j=1,\dots ,d.
\end{split}
\end{equation*}
Summing over $j$ and taking the lower estimate \eqref{<a<} into account, we see that for any $\nu >0$
there exists a constant $C(\nu)>0$ such that
\begin{equation}
\label{Y2<=}
\Vert \mathcal{Y}_2\mathbf{u}\Vert ^2 _{L_2(\mathbb{R}^d)}\leqslant \nu \mathfrak{a}[\mathbf{u},\mathbf{u}]+C(\nu)\Vert \mathbf{u}\Vert ^2_{L_2(\mathbb{R}^d)},\quad \mathbf{u}\in H^1(\mathbb{R}^d;\mathbb{C}^n).
\end{equation}
For $\nu$ fixed, the constant $C(\nu)$ depends only on $d$, $\rho$, $\alpha _0$,  $\Vert g ^{-1}\Vert _{L_\infty}$,
the norms $\Vert a_j\Vert _{L_\rho (\Omega)}$, $j=1,\dots ,d$,  and the parameters of the lattice $\Gamma$.

\subsection{The form $q$}
\label{Subsection form q}
Suppose that $d\mu(\mathbf{x})$ is a $\Gamma$-periodic Borel $\sigma$-finite measure
in  $\mathbb{R}^d$ with values in the class of Hermitian $(n\times n)$-matrices.
In other words, $d\mu (\mathbf{x})=\lbrace d\mu _{jl}(\mathbf{x})\rbrace$, $j,l=1,\dots ,n$, where
 $d\mu _{jl}(\mathbf{x})$ is a complex-valued $\Gamma$-periodic measure in $\mathbb{R}^d$, and $d\mu _{jl}=d\mu _{lj}^*$. Suppose that the measure $d\mu $ is such that the function $\vert u(\mathbf{x})\vert ^2$ is integrable with respect to each measure
$d\mu _{jl}$ for any $u\in H^1(\mathbb{R}^d)$.

In $L_2(\mathbb{R}^d;\mathbb{C}^n)$, we consider the quadratic form
\begin{equation}
\label{q=}
q[\mathbf{u},\mathbf{u}]=\int _{\mathbb{R}^d}\langle d\mu (\mathbf{x})\mathbf{u},\mathbf{u}\rangle =
\sum \limits _{j,l=1}^n \int _{\mathbb{R}^d}u_l (\mathbf{x})u_j(\mathbf{x})^*\,d\mu _{jl}(\mathbf{x}),\quad \mathbf{u}\in H^1(\mathbb{R}^d;\mathbb{C}^n).
\end{equation}

The measure $d\mu$ is subject to the following condition.

\begin{condition}
\label{Condition on measure mu}
There exist constants $c_0\geqslant 0$, $\widetilde{c}_2\geqslant 0$, $c_3\geqslant 0$, and \break$0\leqslant \widetilde{c}<\alpha _0\Vert g^{-1}\Vert ^{-1}_{L_\infty}$ such that
\begin{equation}
\label{condition mu}
\begin{split}
-\widetilde{c}\Vert \mathbf{D}\mathbf{u}\Vert ^2 _{L_2(\Omega)}&-c_0\Vert \mathbf{u}\Vert ^2 _{L_2(\Omega)}\leqslant \int  _{\Omega}\langle d\mu (\mathbf{x})\mathbf{u},\mathbf{u}\rangle \\
&\leqslant \widetilde{c}_2\Vert \mathbf{D}\mathbf{u}\Vert ^2_{L_2(\Omega)} +c_3\Vert \mathbf{u}\Vert ^2 _{L_2(\Omega)},\quad \mathbf{u}\in H^1(\Omega ;\mathbb{C}^n).
\end{split}
\end{equation}
\end{condition}

For  $\mathbf{u} \in H^1(\mathbb{R}^d;\mathbb{C}^n)$, writing inequalities of the form
\eqref{condition mu} over the shifted cells $\Omega+{\mathbf a}$, ${\mathbf a}\in \Gamma$, and summing up, we obtain
similar inequalities in ${\mathbb R}^d$.
Together with \eqref{<a<} this yields
\begin{equation}
\label{<q<}
\begin{split}
-&(1-\kappa )\mathfrak{a}[\mathbf{u},\mathbf{u}] -c_0\Vert \mathbf{u}\Vert ^2_{L_2(\mathbb{R}^d)}\\
&\leqslant q[\mathbf{u},\mathbf{u}]\leqslant c_2 \mathfrak{a}[\mathbf{u},\mathbf{u}]+c_3\Vert \mathbf{u}\Vert ^2_{L_2(\mathbb{R}^d)},\quad \mathbf{u}\in H^1(\mathbb{R}^d;\mathbb{C}^n),
\end{split}
\end{equation}
where
\begin{equation}
\label{c_2 kappa}
c_2=\widetilde{c}_2\alpha _0^{-1}\Vert g^{-1}\Vert _{L_\infty},\quad \kappa =1-\widetilde{c}\alpha _0^{-1}\Vert g^{-1}\Vert _{L_\infty},\quad 0<\kappa \leqslant 1.
\end{equation}

Examples of forms \eqref{q=} are given in \cite[Subsection 5.5]{SuAA}.
Here we give only the main example (see \cite[Example~5.3]{SuAA}).

\begin{example}
\label{Example where Q in L_s}
{\rm
Suppose that the measure $d\mu$ is absolutely continuous with respect to Lebesgue measure: $d\mu (\mathbf{x})=Q(\mathbf{x})\,d\mathbf{x}$,
where $Q(\mathbf{x})$ is a  $\Gamma$-periodic Hermitian $(n\times n)$-matrix-valued function in $\mathbf{x}\in\mathbb{R}^d$ such that
\begin{equation}
\label{Q in L_s example}
Q\in L_s(\Omega),\quad s=1\ \mbox{for}\ d=1, \quad s>d/2\ \mbox{for}\ d\geqslant 2.
\end{equation}
Then $q[\mathbf{u},\mathbf{u}]=(Q\mathbf{u},\mathbf{u})_{L_2(\mathbb{R}^d)}$, $\mathbf{u}\in H^1(\mathbb{R}^d;\mathbb{C}^n)$,
and for any $\nu >0$ there exists a constant $C_Q(\nu)> 0$ such that
\begin{equation*}
\int _\Omega \vert \langle Q(\mathbf{x})\mathbf{u},\mathbf{u}\rangle \vert \,d\mathbf{x}\leqslant \nu \int _\Omega \vert \mathbf{D}\mathbf{u}\vert ^2\,d\mathbf{x}
+C_Q(\nu)\int _\Omega \vert \mathbf{u}\vert ^2\,d\mathbf{x},\quad\mathbf{u}\in H^1(\Omega ;\mathbb{C}^n).
\end{equation*}
If $\nu$ is fixed, the constant $C_Q(\nu)$ is controlled in terms of $d$, $s$, $\Vert Q\Vert _{L_s(\Omega)}$,
and the parameters of the lattice $\Gamma$. Putting $\nu =2^{-1}\alpha _0\Vert g^{-1}\Vert _{L_\infty}^{-1}$,
we see that Condition \ref{Condition on measure mu} is satisfied with $\widetilde{c}=\nu$, $c_0=C_Q(\nu)$, $\widetilde{c}_2=1$, and $c_3=C_Q(1)$.
Then, by \eqref{c_2 kappa}, $c_2 = \alpha_0^{-1} \|g^{-1}\|_{L_\infty}$ and $\kappa = 1/2$.}

\end{example}

\subsection{The operator $\mathcal{B}(\varepsilon) $}\label{Subsection Operator B(eps)}
In $L_2(\mathbb{R}^d;\mathbb{C}^n)$, we consider the family of operators $\mathcal{B}(\varepsilon )$, $0<\varepsilon \leqslant 1$,
formally given by the differential expression
\begin{equation*}
\begin{split}
\mathcal{B}(\varepsilon)&=\mathcal{A}+\varepsilon (\mathcal{Y}_2^*\mathcal{Y}+\mathcal{Y}^*\mathcal{Y}_2)+\varepsilon ^2 Q\\
&=b(\mathbf{D})^*g(\mathbf{x})b(\mathbf{D})+\varepsilon \sum \limits _{j=1}^d \left( a_j(\mathbf{x})D_j +D_j a_j(\mathbf{x})^*\right)+\varepsilon ^2 Q(\mathbf{x}),
\end{split}
\end{equation*}
where $Q(\mathbf{x})$ can be interpreted as the generalized matrix-valued potential
generated by the measure $d\mu (\mathbf{x})$.
Precisely, $\mathcal{B}(\varepsilon)$ is the selfadjoint operator generated by the
quadratic form
\begin{equation}
\label{b(eps)}
\mathfrak{b}(\varepsilon)[\mathbf{u},\mathbf{u}]=\mathfrak{a}[\mathbf{u},\mathbf{u}]+2\varepsilon \textnormal{Re}\,(\mathcal{Y}\mathbf{u},\mathcal{Y}_2\mathbf{u})_{L_2(\mathbb{R}^d)}+\varepsilon ^2 q[\mathbf{u},\mathbf{u}],\quad \mathbf{u}\in H^1(\mathbb{R}^d;\mathbb{C}^n).
\end{equation}
Let us check that the form \eqref{b(eps)} is closed and lower semibounded. By \eqref{Y<=} and \eqref{Y2<=},
\begin{equation}
\label{Su_AA_5.24}
\begin{split}
&2\varepsilon \vert \textnormal{Re}\,(\mathcal{Y}\mathbf{u},\mathcal{Y}_2\mathbf{u})_{L_2(\mathbb{R}^d)}\vert \leqslant \frac{\kappa}{2}\mathfrak{a}[\mathbf{u},\mathbf{u}]+c_4\varepsilon ^2\Vert \mathbf{u}\Vert ^2_{L_2(\mathbb{R}^d)},\\
&\mathbf{u}\in H^1(\mathbb{R}^d;\mathbb{C}^n),\quad
c_4=4\kappa ^{-1}c_1^2C(\nu_0)\ \mbox{for}\ \nu_0 =\kappa ^2(16c_1^2)^{-1}.
\end{split}
\end{equation}
Now, the lower estimate \eqref{<q<}, relation \eqref{Su_AA_5.24}, and the restriction $0<\varepsilon\leqslant 1$
imply the lower estimate for the form \eqref{b(eps)}:
\begin{equation*}
\begin{split}
\mathfrak{b}(\varepsilon)[\mathbf{u},\mathbf{u}]\geqslant \frac{\kappa}{2}\mathfrak{a}[\mathbf{u},\mathbf{u}]-(c_0+c_4)\varepsilon ^2\Vert \mathbf{u}\Vert ^2_{L_2(\mathbb{R}^d)},\quad \mathbf{u}\in H^1(\mathbb{R}^d;\mathbb{C}^n).
\end{split}
\end{equation*}
Together with the lower estimate \eqref{<a<} this yields
\begin{equation}
\label{1.17a}
\begin{split}
&\mathfrak{b}(\varepsilon)[\mathbf{u},\mathbf{u}]\geqslant
c_* \Vert \mathbf{D}\mathbf{u} \Vert ^2_{L_2(\mathbb{R}^d)}
-(c_0+c_4)\varepsilon ^2\Vert \mathbf{u}\Vert ^2_{L_2(\mathbb{R}^d)},
\cr
&\mathbf{u}\in H^1(\mathbb{R}^d;\mathbb{C}^n),
\quad 0 < \varepsilon \leqslant 1;
\quad c_* := \frac{\kappa}{2} \alpha_0 \|g^{-1}\|^{-1}_{L_\infty}.
\end{split}
\end{equation}
Combining \eqref{Y<=}, \eqref{Y2<=} with $\nu =1$, and the upper estimate \eqref{<q<}, we arrive at
\begin{equation*}
\mathfrak{b}(\varepsilon)[\mathbf{u},\mathbf{u}]\leqslant (2+c_1^2+c_2)\mathfrak{a}[\mathbf{u},\mathbf{u}]+(C(1)+c_3)\varepsilon ^2\Vert \mathbf{u}\Vert ^2_{L_2(\mathbb{R}^d)},\quad \mathbf{u}\in H^1(\mathbb{R}^d;\mathbb{C}^n).
\end{equation*}
Together with the upper estimate \eqref{<a<} this leads to
\begin{equation}
\label{1.18a}
\begin{split}
&\mathfrak{b}(\varepsilon)[\mathbf{u},\mathbf{u}]\leqslant
C_* \Vert \mathbf{D}\mathbf{u} \Vert ^2_{L_2(\mathbb{R}^d)}
+(C(1)+c_3)\varepsilon ^2\Vert \mathbf{u}\Vert ^2_{L_2(\mathbb{R}^d)},
\cr
&\mathbf{u}\in H^1(\mathbb{R}^d;\mathbb{C}^n),
\quad 0 < \varepsilon \leqslant 1;
\quad C_* := (2+c_1^2+c_2) \alpha_1 \|g\|_{L_\infty}.
\end{split}
\end{equation}

\subsection{The operator $\mathcal{B}_\varepsilon$}\label{Subsection Operator B_eps}
Let $T_\varepsilon$, $\varepsilon >0$, be the unitary scaling transformation in $L_2(\mathbb{R}^d;\mathbb{C}^n)$
defined by
\begin{equation}
\label{T_eps}
(T_\varepsilon \mathbf{u})(\mathbf{x})=\varepsilon ^{d/2}\mathbf{u}(\varepsilon \mathbf{x}),\quad \mathbf{u}\in L_2(\mathbb{R}^d;\mathbb{C}^n),
\quad \varepsilon >0.
\end{equation}

Let $\mathcal{A}_\varepsilon$ be the selfadjoint operator in $L_2(\mathbb{R}^d;\mathbb{C}^n)$
corresponding to the quadratic form
\begin{equation*}
\mathfrak{a}_\varepsilon [\mathbf{u},\mathbf{u}]:=\varepsilon ^{-2}\mathfrak{a}[T_\varepsilon \mathbf{u},T_\varepsilon \mathbf{u}]=(g^\varepsilon b(\mathbf{D})\mathbf{u},b(\mathbf{D})\mathbf{u})_{L_2(\mathbb{R}^d)},\quad \mathbf{u}\in H^1(\mathbb{R}^d;\mathbb{C}^n).
\end{equation*}
Formally, we have $\mathcal{A}_\varepsilon =b(\mathbf{D})^*g^\varepsilon (\mathbf{x})b(\mathbf{D})$.

We introduce the operator $\mathcal{Y}_{2,\varepsilon}$ defined by
\begin{equation*}
(\mathcal{Y}_{2,\varepsilon}\mathbf{u})(\mathbf{x})=\textnormal{col}\,\lbrace a_1^\varepsilon (\mathbf{x})^*\mathbf{u}(\mathbf{x}),\dots ,a_d^\varepsilon (\mathbf{x})^*\mathbf{u}(\mathbf{x})\rbrace ,\quad \mathbf{u}\in H^1(\mathbb{R}^d;\mathbb{C}^n).
\end{equation*}

Let $d\mu$ be a measure defined in Subsection \ref{Subsection form q}.
We define the measure  $d\mu ^\varepsilon$ as follows.
For any Borel set $\Delta \subset\mathbb{R}^d$, we consider the set 
\hbox{$\varepsilon ^{-1}\Delta : =\lbrace \mathbf{x}\in \mathbb{R}^d : \varepsilon \mathbf{x}\in \Delta \rbrace $} and put
$\mu ^\varepsilon (\Delta ):=\varepsilon ^d \mu (\varepsilon ^{-1}\Delta)$.
Consider the quadratic form $q_\varepsilon$ defined by
\begin{equation*}
q_\varepsilon [\mathbf{u},\mathbf{u}]=\int _{\mathbb{R}^d}\langle d\mu ^\varepsilon (\mathbf{x})\mathbf{u},\mathbf{u}\rangle ,\quad \mathbf{u}\in H^1(\mathbb{R}^d;\mathbb{C}^n).
\end{equation*}

We put $\mathcal{B}_\varepsilon :=\varepsilon ^{-2} T_\varepsilon ^*\mathcal{B}(\varepsilon )T_\varepsilon $,
$0< \varepsilon \leqslant 1$. In other words, $\mathcal{B}_\varepsilon$ is the selfadjoint operator in $L_2(\mathbb{R}^d;\mathbb{C}^n)$
generated by the quadratic form
\begin{equation}
\label{b_eps}
\begin{split}
\mathfrak{b}_\varepsilon [\mathbf{u},\mathbf{u}]&:=\varepsilon ^{-2}\mathfrak{b}(\varepsilon)[T_\varepsilon\mathbf{u},T_\varepsilon \mathbf{u}]\\
&=
\mathfrak{a}_\varepsilon [\mathbf{u},\mathbf{u}]
+2\textnormal{Re}\,(\mathcal{Y}\mathbf{u},\mathcal{Y}_{2,\varepsilon}\mathbf{u})_{L_2(\mathbb{R}^d)}+q_\varepsilon [\mathbf{u},\mathbf{u}],\quad \mathbf{u}\in H^1(\mathbb{R}^d;\mathbb{C}^n).
\end{split}
\end{equation}
Relations  \eqref{b_eps},  \eqref{1.17a}, and \eqref{1.18a}  imply that
\begin{align}
\label{b_eps>=}
&\mathfrak{b}_\varepsilon [\mathbf{u},\mathbf{u}]\geqslant c_* \Vert \mathbf{D}\mathbf{u} \Vert ^2_{L_2(\mathbb{R}^d)}
-(c_0+c_4)\Vert \mathbf{u}\Vert ^2 _{L_2(\mathbb{R}^d)},\quad \mathbf{u}\in H^1(\mathbb{R}^d;\mathbb{C}^n),\\
&\mathfrak{b}_\varepsilon [\mathbf{u},\mathbf{u}]\leqslant
C_* \Vert \mathbf{D}\mathbf{u} \Vert ^2_{L_2(\mathbb{R}^d)}
+(C(1)+c_3)\Vert \mathbf{u}\Vert ^2_{L_2(\mathbb{R}^d)},\quad \mathbf{u}\in H^1(\mathbb{R}^d;\mathbb{C}^n),\nonumber
\end{align}
for $0< \varepsilon \leqslant 1$.
Thus, the form  $\mathfrak{b}_\varepsilon$ is closed and lower semibounded.
Formally, we can write
\begin{equation}
\label{B_eps}
\mathcal{B}_\varepsilon =b(\mathbf{D})^*g^\varepsilon (\mathbf{x})b(\mathbf{D})+\sum \limits _{j=1}^d\left( a_j^\varepsilon (\mathbf{x})D_j +D_ja_j^\varepsilon (\mathbf{x})^*\right) +Q^\varepsilon (\mathbf{x}).
\end{equation}
Here $Q^\varepsilon $ can be interpreted as the generalized matrix potential generated by the measure $d\mu ^\varepsilon $.
We see that the coefficients of the operator $\mathcal{B}_\varepsilon$ oscillate rapidly for small $\varepsilon$.

\subsection{The effective matrix}
\label{Subsection Effective matrix}
The effective operator for $\mathcal{A}_\varepsilon =b(\mathbf{D})^*g^\varepsilon (\mathbf{x})b(\mathbf{D})$
is given by $\mathcal{A}^0=b(\mathbf{D})^*g^0b(\mathbf{D})$. Here $g^0$ is a constant positive $(m\times m)$-matrix called the effective matrix. The matrix  $g^0$ is defined in terms of the solution of the auxiliary problem on $\Omega$.
Suppose that a $\Gamma$-periodic $(n\times m)$-matrix-valued function $\Lambda (\mathbf{x})$
is the solution of the problem
\begin{equation}
\label{Lambda problem}
b(\mathbf{D})^*g(\mathbf{x})(b(\mathbf{D})\Lambda (\mathbf{x})+\mathbf{1}_m)=0,\quad \int _{\Omega }\Lambda (\mathbf{x})\,d\mathbf{x}=0.
\end{equation}
(The equation is understood in the weak sense.) Then the effective matrix is given by
\begin{equation}
\label{g^0}
g^0=\vert \Omega \vert ^{-1}\int _{\Omega} \widetilde{g}(\mathbf{x})\,d\mathbf{x},
\end{equation}
where
\begin{equation}
\label{tilde g}
\widetilde{g}(\mathbf{x}):=g(\mathbf{x})(b(\mathbf{D})\Lambda (\mathbf{x})+\mathbf{1}_m).
\end{equation}
It is easily seen that $g^0$ is positive definite.

We need the following estimates for $\Lambda$ proved in \cite[(6.28) and Subsection~7.3]{BSu05}:
\begin{align}
\label{Lambda <=}
&\Vert \Lambda \Vert _{L_2(\Omega)}\leqslant \vert \Omega \vert ^{1/2} M_1,\quad M_1:= m^{1/2}(2r_0)^{-1}\alpha _0^{-1/2}\Vert g\Vert ^{1/2}_{L_\infty}\Vert g^{-1}\Vert ^{1/2}_{L_\infty},\\
\label{DLambda<=}
&\Vert \mathbf{D}\Lambda \Vert _{L_2(\Omega)}\leqslant \vert \Omega \vert ^{1/2} M_2,\quad M_2:= m^{1/2}\alpha _0^{-1/2}\Vert g\Vert ^{1/2}_{L_\infty}\Vert g^{-1}\Vert ^{1/2}_{L_\infty}.
\end{align}

The effective matrix has the following properties (see \cite[Chapter~3, Theorem~1.5]{BSu}).

\begin{proposition}
\label{Proposition Voigt-Reuss}
Let $g^0$ be the effective matrix \eqref{g^0}. Then
\begin{equation}
\label{Foigt-Reiss}
\underline{g}\leqslant g^0\leqslant \overline{g}.
\end{equation}
If $m=n$, then $g^0=\underline{g}$.
\end{proposition}

Estimates \eqref{Foigt-Reiss} are known in homogenization theory as the Voigt-Reuss bracketing.
From \eqref{Foigt-Reiss} it follows that
\begin{equation*}
\vert g^0\vert \leqslant \Vert g\Vert _{L_\infty},\quad \vert (g^0)^{-1}\vert \leqslant \Vert g^{-1}\Vert _{L_\infty}.
\end{equation*}

Now we distinguish the cases where one of the inequalities in \eqref{Foigt-Reiss} becomes an identity
(see \cite[Chapter 3, Propositions 1.6 and 1.7]{BSu}).

\begin{proposition}
\label{Proposition g^0=mean value g}
The identity $g^0=\overline{g}$ is equivalent to the relations
\begin{equation}
\label{overline-g}
b(\D)^* {\mathbf g}_k(\x) =0,\ \ k=1,\dots,m,
\end{equation}
where ${\mathbf g}_k(\x)$, $k=1,\dots,m,$ are the columns of the matrix $g(\x)$.
\end{proposition}

\begin{proposition}
\label{Proposition g^0 = underline g}
The identity $g^0 =\underline{g}$ is equivalent to the representations
\begin{equation}
\label{underline-g}
{\mathbf l}_k(\x) = {\mathbf l}_k^0 + b(\D) {\mathbf w}_k,\ \ {\mathbf l}_k^0\in \C^m,\ \
{\mathbf w}_k \in \widetilde{H}^1(\Omega;\C^n),\ \ k=1,\dots,m,
\end{equation}
where ${\mathbf l}_k(\x)$, $k=1,\dots,m,$ are the columns of the matrix $g(\x)^{-1}$.
\end{proposition}

\subsection{The effective operator $\mathcal{B}^0$}
\label{Subsection Effective operator B^0}
 The effective operator for $\mathcal{B}_\varepsilon$ was introduced in \cite{SuAA} (see also \cite{Bo}).

 Suppose that a $\Gamma$-periodic  $(n\times n)$-matrix-valued function
$\widetilde{\Lambda}(\mathbf{x})$ is the (weak) solution of the problem
\begin{equation}
\label{tildeLambda_problem}
b(\mathbf{D})^*g(\mathbf{x})b(\mathbf{D})\widetilde{\Lambda }(\mathbf{x})+\sum \limits _{j=1}^dD_ja_j(\mathbf{x})^*=0,\quad \int _{\Omega }\widetilde{\Lambda }(\mathbf{x})\,d\mathbf{x}=0.
\end{equation}
Define constant matrices $V$ and $W$ as follows:
\begin{align}
\label{V matr}
&V=\vert \Omega \vert ^{-1}\int _{\Omega}(b(\mathbf{D})\Lambda (\mathbf{x}))^*g(\mathbf{x})(b(\mathbf{D})\widetilde{\Lambda}(\mathbf{x}))\,d\mathbf{x},\\
\label{W}
&W=\vert \Omega \vert ^{-1}\int _{\Omega} (b(\mathbf{D})\widetilde{\Lambda}(\mathbf{x}))^*g(\mathbf{x})(b(\mathbf{D})\widetilde{\Lambda}(\mathbf{x}))\,d\mathbf{x}.
\end{align}
The effective operator for the operator \eqref{B_eps} is given by
\begin{equation}
\label{B^0}
\mathcal{B}^0 =b(\mathbf{D})^*g^0b(\mathbf{D})-b(\mathbf{D})^*V-V^*b(\mathbf{D})
+\sum \limits _{j=1}^d (\overline{a_j+a_j^*})D_j-W+\overline{Q}.
\end{equation}
The operator $\mathcal{B}^0$ is a second order elliptic operator with constant coefficients.

According to \cite[(5.7)]{SuAA14}, if $\lambda > c_0+c_4$, then
the symbol $L_\lambda (\boldsymbol{\xi})$ of the operator $\mathcal{B}^0+\lambda I$ satisfies
$L_\lambda (\boldsymbol{\xi})\geqslant c_\lambda (\vert \boldsymbol{\xi}\vert ^2+1)\mathbf{1}_n$, $\boldsymbol{\xi}\in \mathbb{R}^d$,
where ${c}_\lambda =\min \lbrace c_*; \lambda - c_0 - c_4 \rbrace$.
Putting $\lambda _*:=c_0+c_4+c_*$, we obtain the following estimate for the symbol
$L_*(\boldsymbol{\xi})$ of the operator $\mathcal{B}^0+\lambda _* I$:
\begin{equation*}
L_*(\boldsymbol{\xi})\geqslant c_* (\vert \boldsymbol{\xi}\vert ^2+1)\mathbf{1}_n,\quad \boldsymbol{\xi}\in\mathbb{R}^d.
\end{equation*}
Consequently, the quadratic form $\mathfrak{b}^0$ of the operator $\mathcal{B}^0$ satisfies
\begin{equation*}
\mathfrak{b}^0[\mathbf{u},\mathbf{u}]+\lambda _*\Vert \mathbf{u}\Vert ^2 _{L_2(\mathbb{R}^d)}\geqslant
c_* \left(\Vert \mathbf{D}\mathbf{u}\Vert ^2 _{L_2(\mathbb{R}^d)}+\Vert \mathbf{u}\Vert ^2 _{L_2(\mathbb{R}^d)}\right),\quad \mathbf{u}\in H^1(\mathbb{R}^d;\mathbb{C}^n).
\end{equation*}
Since $\lambda _* =c_0+c_4+c_*$, we deduce
\begin{equation}
\label{frak b^0 >=}
\mathfrak{b}^0[\mathbf{u},\mathbf{u}]\geqslant c_*
 \Vert \mathbf{D}\mathbf{u}\Vert ^2 _{L_2(\mathbb{R}^d)}-(c_0+c_4)\Vert \mathbf{u}\Vert ^2 _{L_2(\mathbb{R}^d)},\quad \mathbf{u}\in H^1(\mathbb{R}^d;\mathbb{C}^n).
\end{equation}

Below we need the following estimates for $\widetilde{\Lambda}$ proved in  \cite[(7.52) and (7.51)]{SuAA}:
\begin{align}
\label{tildeLambda <=}
&\Vert \widetilde{\Lambda}\Vert _{L_2(\Omega)}\leqslant (2r_0)^{-1}C_an^{1/2}\alpha _0^{-1}\Vert g^{-1}\Vert _{L_\infty},\\
\label{DtildeLambda<=}
&\Vert \mathbf{D} \widetilde{\Lambda}\Vert _{L_2(\Omega)}\leqslant C_a n^{1/2}\alpha _0^{-1}\Vert g^{-1}\Vert _{L_\infty},
\end{align}
where $C_a^2=\sum _{j=1}^d\int _\Omega \vert a_j(\mathbf{x})\vert ^2\,d\mathbf{x}$.

\subsection{The generalized resolvent}
\label{Subsection generalized resolvent}
Let $Q_0(\mathbf{x})$ be a $\Gamma$-periodic $(n\times n)$-matrix-valued function
such that
\begin{equation*}
Q_0(\mathbf{x}) >0,\quad Q_0,Q_0^{-1}\in L_\infty (\mathbb{R}^d).
\end{equation*}
We study the generalized resolvent of the operator ${\mathcal B}_\varepsilon$,
i.~e., the operator $({\mathcal B}_\varepsilon - z Q_0^\varepsilon)^{-1}$. We rely on the results of  \cite{SuAA},
where approximations for this resolvent at a fixed real point $z$ were obtained.
Before we formulate the results, it is convenient to turn from the operator ${\mathcal B}_\varepsilon$
to the nonnegative operator
\begin{equation}
\label{Beps}
B_\varepsilon := {\mathcal B}_\varepsilon + c_5 Q_0^\varepsilon,
\end{equation}
putting $c_5:= (c_0 + c_4) \|Q_0^{-1}\|_{L_\infty}$. From \eqref{b_eps>=} it follows that  $B_\varepsilon \geqslant 0$.
Observe that the operator $B_\varepsilon$ can be considered as the operator of the form
\eqref{B_eps} with the initial coefficients $g^\varepsilon$, $a_j^\varepsilon$, and the ,,new'' matrix-valued potential
$\check{Q}^\varepsilon = Q^\varepsilon +c_5Q_0^\varepsilon$. The corresponding form
$\int_\Omega \langle d\mu({\mathbf x}) {\mathbf u}, {\mathbf u}\rangle +
c_5 \int_\Omega \langle Q_0({\mathbf x}) {\mathbf u}, {\mathbf u}\rangle\, d{\mathbf x}$
satisfies Condition \ref{Condition on measure mu} with $\check{c}_0=0$ in place of $c_0$,
the constant $\check{c}_3 = c_3 +c_5\Vert Q_0\Vert _{L_\infty}$ in place of $c_3$, and the initial
$\widetilde{c}$ and $\widetilde{c}_2$.

Let us fix $\lambda_0$ as follows: 
\begin{equation}
\label{lambda0}
\lambda_0 = 2 \| Q_0^{-1}\|_{L_\infty} c_4,
\end{equation}
cf. \cite[(5.27)]{SuAA}.
The operator $B_\varepsilon + \lambda_0 Q_0^\eps$ is positive definite, whence the generalized resolvent
$(B_\varepsilon + \lambda_0 Q_0^\eps)^{-1}$ is a bounded operator in $L_2({\mathbb R}^d;{\mathbb C}^n)$.

\begin{remark}
\label{Remark on condition on lambda_0}
In \cite{SuAA}, the generalized resolvent  $(B_\varepsilon + \lambda Q_0^\eps)^{-1}$
was studied for $\lambda \ge \lambda_0$, while in  \cite{SuAA14} it was studied under the weaker assumption that $\lambda >0$.
For our goals, it suffices to formulate the results for $\lambda$ fixed{\rm ;}
for convenience of referring to \cite{SuAA}, we use condition \textnormal{\eqref{lambda0}}.
\end{remark}

For convenience of further references, the set of parameters
\begin{equation}
\label{problem data}
\begin{split}
&d,\,m,\,n,\,\rho ;\,\alpha _0,\, \alpha _1 ,\,\Vert g\Vert _{L_\infty},\, \Vert g^{-1}\Vert _{L_\infty},\, \Vert a_j\Vert _{L_\rho (\Omega)},\, j=1,\dots ,d,\\
&\widetilde{c},\, c_0,\, \widetilde{c}_2,\, c_3 \mbox{ from Condition \ref{Condition on measure mu}};\, \Vert Q_0\Vert _{L_\infty},\,\Vert Q_0^{-1}\Vert _{L_\infty}; \\
& \mbox{ the parameters of the lattice}\ \Gamma
\end{split}
\end{equation}
is called the ,,initial data''.
Note that the constants $c_1$, $C(1)$, $\kappa$, $c_2$, and $c_4$ are determined by the initial data.

The effective operator for the operator \eqref{Beps} is given by
\begin{equation}
\label{B0}
B^0 = \mathcal{B}^0+c_5 \overline{Q_0}.
\end{equation}
Relation \eqref{frak b^0 >=} implies the lower estimate for the quadratic form $b^0$ of the operator~$B^0$:
\begin{equation}
\label{b^0[u,u]>=}
b^0[\mathbf{u},\mathbf{u}]\geqslant c_*
\Vert \mathbf{D}\mathbf{u}\Vert ^2_{L_2(\mathbb{R}^d)},\quad \mathbf{u}\in H^1(\mathbb{R}^d;\mathbb{C}^n).
\end{equation}
This is equivalent to the following estimate for the symbol $L_0(\boldsymbol{\xi})$ of the operator $B^0$:
\begin{equation}
\label{L0}
L_0(\boldsymbol{\xi})\geqslant {c}_* \vert \boldsymbol{\xi}\vert ^2 \mathbf{1}_n,\quad \boldsymbol{\xi}\in \mathbb{R}^d.
\end{equation}
The operator $B^0 + \lambda_0 \overline{Q_0}$ is the second order DO with constant coefficients with the symbol
\begin{equation*}
\begin{split}
&L(\boldsymbol{\xi})=L_0(\boldsymbol{\xi}) + \lambda_0 \overline{Q_0}
\\
&= b(\boldsymbol{\xi})^*g^0 b(\boldsymbol{\xi})-b(\boldsymbol{\xi})^*V-V^*b(\boldsymbol{\xi})+\sum\limits _{j=1}^d (\overline{a_j+a_j^*})\xi _j+\overline{Q}-W+ (c_5 + \lambda_0) \overline{Q_0}.
\end{split}
\end{equation*}
From \eqref{L0} and \eqref{lambda0} it follows that
\begin{equation}
\label{L>=}
L(\boldsymbol{\xi})\geqslant \check{c}_* (\vert \boldsymbol{\xi}\vert ^2+1)\mathbf{1}_n,\quad \boldsymbol{\xi}\in \mathbb{R}^d;
\quad \check{c}_*=\min \lbrace c_*;2c_4\rbrace.
\end{equation}

\subsection{Approximation of the operator $(B_\varepsilon +\lambda_0 Q_0^\varepsilon)^{-1}$}
Applying Theorem~9.2 from \cite{SuAA} to the operator \eqref{Beps}, we obtain the following result.

\begin{theorem}[\cite{SuAA}]\label{Theorem Su L2} Suppose that the assumptions of Subsections~\textnormal{\ref{Subsection Operator A}--\ref{Subsection generalized resolvent}} are satisfied. Let $\lambda_0$ be given by~\eqref{lambda0}.
 Then for $0<\varepsilon \leqslant 1$ we have
\begin{equation*}
\Vert (B_\varepsilon +\lambda_0 Q_0^\varepsilon )^{-1}-(B^0+\lambda_0 \overline{Q_0})^{-1}\Vert _{L_2(\mathbb{R}^d)\rightarrow L_2(\mathbb{R}^d)}\leqslant C_1\varepsilon .
\end{equation*}
The constant $C_1$ is controlled in terms of the initial data~\eqref{problem data}.
\end{theorem}

In order to approximate the operator $(B_\varepsilon +\lambda_0 Q_0^\varepsilon )^{-1}$ in the
$(L_2\rightarrow H^1)$-operator norm, we introduce a corrector
\begin{equation}
\label{K(eps)}
K(\varepsilon )=\left([\Lambda ^\varepsilon] b(\mathbf{D})+[\widetilde{\Lambda}^\varepsilon] \right)\Pi _\varepsilon
(B^0+\lambda_0 \overline{Q_0})^{-1}.
\end{equation}
Here $\Pi_\varepsilon$ is given by \eqref{Pi_eps}.
The corrector \eqref{K(eps)} is a continuous mapping of $L_2(\mathbb{R}^d;\mathbb{C}^n)$ to $H^1(\mathbb{R}^d;\mathbb{C}^n)$.
This can be easily checked by using Proposition~\ref{Proposition Pi_eps L2->L2} and the relations $\Lambda ,\widetilde{\Lambda}\in \widetilde{H}^1(\Omega)$. Herewith,  $\Vert \varepsilon K(\varepsilon )\Vert _{L_2\rightarrow H^1}=O(1)$.

The following result was obtained in \cite[Theorem~9.7]{SuAA}.

\begin{theorem}[\cite{SuAA}]\label{Theorem Su L2->H1}
Suppose that the assumptions of Theorem~\textnormal{\ref{Theorem Su L2}} are satisfied.
Let $K(\varepsilon)$ be the operator given by \eqref{K(eps)}.
Then for $0<\varepsilon \leqslant 1$ we have
\begin{equation*}
\Vert (B_\varepsilon +\lambda_0 Q_0^\varepsilon )^{-1}-(B^0+\lambda_0 \overline{Q_0} )^{-1}-\varepsilon K(\varepsilon)\Vert _{L_2(\mathbb{R}^d)\rightarrow H^1(\mathbb{R}^d)}\leqslant C_2\varepsilon .
\end{equation*}
The constant $C_2$ is controlled in terms of the initial data \eqref{problem data}.
\end{theorem}

\begin{remark}
\label{Remark C1 C2 in Theorems by Su}
Below in Sections {\rm \ref{Section proof of main result L2}} and  {\rm \ref{Section Proof of main result in H1}}
we will apply Theorems~\textnormal{\ref{Theorem Su L2}} and~\textnormal{\ref{Theorem Su L2->H1}}
{\rm (}precisely, instead of Theorem \textnormal{\ref{Theorem Su L2->H1}}
we will use its analog, Theorem~\textnormal{\ref{Theorem S_eps z fixed})}
in order to approximate the generalized resolvent \hbox{$(B_\varepsilon(\vartheta) +\lambda_0 Q_0^\varepsilon )^{-1}$}
of the operator family~$B_\varepsilon(\vartheta)$ depending on the auxiliary parameter~$\vartheta \in (0,1]$.
The operator $B_\varepsilon(\vartheta)$ has the same principal part as $B_\varepsilon$ and the lower order
coefficients $\vartheta a_j^\varepsilon$ and $\vartheta^2 \check{Q}^\varepsilon$
{\rm (}see Subsection~{\rm \ref{Subsection operator B_eps (vartheta)})}.
Note that, for all $\vartheta \in (0,1]$,  the form $\vartheta^2 \left(\int_\Omega \langle d\mu({\mathbf x}) {\mathbf u}, {\mathbf u}\rangle + c_5
\int_\Omega \langle Q_0({\mathbf x}) {\mathbf u}, {\mathbf u}\rangle\, d{\mathbf x}\right)$
satisfies Condition~\textnormal{\ref{Condition on measure mu}} with the same constants:
$\check{c}_0=0$ in the role of $c_0$, $\check{c}_3 = c_3 +c_5\Vert Q_0\Vert _{L_\infty}$ in the role of $c_3$, and the previous
$\widetilde{c}$ and $\widetilde{c}_2$. Next, the constant $\nu_0 = \kappa^2 (16 c_1^2)^{-1}$
does not depend on~$\vartheta$. Multiplying \eqref{Y2<=} with $\nu=\nu_0$ by $\vartheta^2$,
we see that the inequality of the form~\eqref{Y2<=} {\rm (}in the case of the coefficients $\vartheta a_j${\rm )}
is valid with the constant $C(\nu_0)$ for all $\vartheta \in (0,1]$.
Hence, the constant $c_4$ {\rm (}see~{\rm \eqref{Su_AA_5.24})}, and then also $\lambda_0$
{\rm (}see~{\rm \eqref{lambda0})} can be taken independently of $\vartheta$.
Note also that the norms of the coefficients $\vartheta a_j$ in $L_\rho(\Omega)$
are majorated by the norms $\|a_j\|_{L_\rho(\Omega)}$ for all $\vartheta \in (0,1]$.
In \cite{SuAA}, the dependence of the constants $C_1$ and $C_2$ on the initial data was searched out
{\rm (}in particular, these constants increase, as the norms $\|a_j\|_{L_\rho(\Omega)}$ increase{\rm )}.
What was said allows us to choose the constants $C_1$ and $C_2$ in approximations for the operator
$(B_\varepsilon(\vartheta) +\lambda_0 Q_0^\varepsilon )^{-1}$ to be independent of the parameter $\vartheta$.
\end{remark}

\section{Auxiliary statements}
\label{Section Lemmas}
\setcounter{section}{2}
\setcounter{equation}{0}
\setcounter{theorem}{0}

\subsection{Properties of the matrix-valued function $\Lambda$}
We need the following result proved in~\cite[Lemma~2.3]{PSu}:

\begin{lemma}
\label{Lemma Lambda-varepsilon}
Suppose that $\Lambda $ is the $\Gamma$-periodic solution of problem \eqref{Lambda problem}.
Then for any $u\in C_0^\infty (\mathbb{R}^d)$ and $\varepsilon >0$ we have
\begin{equation*}
\int _{\mathbb{R}^d}\vert (\mathbf{D}\Lambda )^\varepsilon (\mathbf{x})\vert ^2\vert u(\mathbf{x})\vert ^2\,d\mathbf{x}
\leqslant\beta _1\Vert u\Vert ^2_{L_2(\mathbb{R}^d)}
+\beta _2 \varepsilon ^2\int _{\mathbb{R}^d}\vert \Lambda ^\varepsilon (\mathbf{x})\vert ^2 \vert \mathbf{D}u(\mathbf{x})\vert ^2\,d\mathbf{x}.
\end{equation*}
The constants $\beta _1$ and $\beta _2$ depend only on
$m$, $d$, $\alpha _0$, $\alpha _1$, $\Vert g\Vert _{L_\infty}$, and $\Vert g^{-1}\Vert _{L_\infty}$.
\end{lemma}

From Lemma \ref{Lemma Lambda-varepsilon} and the density of $C_0^\infty({\mathbb R}^d)$ in $H^1({\mathbb R}^d)$
we deduce the following statement.

\begin{corollary}
\label{corollary_Lambda}
Suppose that $\Lambda $ is the $\Gamma$-periodic solution of problem \eqref{Lambda problem}.
Assume also that $\Lambda \in L_\infty$.
Then for any $u\in H^1 (\mathbb{R}^d)$ and $\varepsilon >0$ we have
\begin{equation*}
\int _{\mathbb{R}^d}\vert (\mathbf{D}\Lambda )^\varepsilon (\mathbf{x})\vert ^2\vert u(\mathbf{x})\vert ^2\,d\mathbf{x}
\leqslant\beta _1\Vert u\Vert ^2_{L_2(\mathbb{R}^d)}
+\beta _2 \varepsilon ^2 \|\Lambda\|^2_{L_\infty} \int _{\mathbb{R}^d} \vert \mathbf{D}u(\mathbf{x})\vert ^2\,d\mathbf{x}.
\end{equation*}
\end{corollary}

\subsection{Properties of the matrix-valued function  $\widetilde{\Lambda}$}
The proof of the following statement is similar to the proof of Lemma~2.1 from \cite{PSu}.

\begin{lemma}
\label{Lemma on Lambda-tilda}
Suppose that $\widetilde{\Lambda}$ is the $\Gamma$-periodic solution of problem \eqref{tildeLambda_problem}.
Then for any $u\in C_0^\infty (\mathbb{R}^d)$ we have
\begin{equation}
\label{Lm2.1_tildeLambda}
\int _{\mathbb{R}^d}\vert \mathbf{D}\widetilde{\Lambda}(\mathbf{x})\vert ^2\vert u(\mathbf{x})\vert ^2\,d\mathbf{x}
\leqslant \widetilde{\beta} _1 \Vert u\Vert ^2_{H^1(\mathbb{R}^d)}+\widetilde{\beta} _2 \int _{\mathbb{R}^d}\vert \widetilde{\Lambda} (\mathbf{x})\vert ^2 \vert \mathbf{D}u(\mathbf{x})\vert ^2 \,d\mathbf{x}.
\end{equation}
The constants $\widetilde{\beta} _1$ and $\widetilde{\beta} _2$ are given below in \eqref{Lm2.1_proof_11}
and depend only on $n$, $d$, $\alpha _0$, $\alpha _1$, $\rho$, $\Vert g\Vert _{L_\infty}$, $\Vert g^{-1}\Vert _{L_\infty}$,
the norms $\Vert a_j\Vert _{L_\rho (\Omega)}$, $j=1,\dots ,d$, and the parameters of the lattice $\Gamma$.
 \end{lemma}

\begin{proof}
Let $\mathbf{w}_k(\mathbf{x})$, $k=1,\dots ,n$, be the columns of the matrix $\widetilde{\Lambda}(\mathbf{x})$.
By \eqref{tildeLambda_problem}, for any function $\boldsymbol{\eta}\in H^1(\mathbb{R}^d;\mathbb{C}^n)$
such that  $\boldsymbol{\eta}(\mathbf{x})=0$ for $\vert \mathbf{x}\vert >R$ (with some $R>0$) we have
\begin{equation}
\label{Lm2.1_proof_1}
\int _{\mathbb{R}^d}\biggl( \langle g(\mathbf{x})b(\mathbf{D})\mathbf{w}_k(\mathbf{x}),b(\mathbf{D})\boldsymbol{\eta}(\mathbf{x})\rangle
+\sum \limits _{j=1}^d \langle a_j(\mathbf{x})^*\mathbf{e}_k ,D_j \boldsymbol{\eta}(\mathbf{x})\rangle \biggr) \,d\mathbf{x}=0.
\end{equation}
Here $\mathbf{e}_k$, $k=1,\dots ,n$, is the standard orthonormal basis in $\mathbb{C}^n$.

Let $u\in C_0^\infty (\mathbb{R}^d)$. We put $\boldsymbol{\eta}(\mathbf{x}):=\mathbf{w}_k(\mathbf{x})\vert u(\mathbf{x})\vert ^2$.
Then, by \eqref{b(D)=},
\begin{align}
\label{Lm2.1_proof_2}
b(\mathbf{D})\boldsymbol{\eta}(\mathbf{x})=\left(b(\mathbf{D})\mathbf{w}_k(\mathbf{x})\right)\vert u(\mathbf{x})\vert ^2
+\sum \limits _{l=1}^d b_l \mathbf{w}_k(\mathbf{x})D_l\vert u(\mathbf{x})\vert ^2,\\
\label{Lm2.1_proof_3}
D_j\boldsymbol{\eta}(\mathbf{x})=\left(D_j\mathbf{w}_k(\mathbf{x})\right)\vert u(\mathbf{x})\vert ^2+\mathbf{w}_k(\mathbf{x})D_j\vert u(\mathbf{x})\vert ^2.
\end{align}
Substituting \eqref{Lm2.1_proof_2} and \eqref{Lm2.1_proof_3} in \eqref{Lm2.1_proof_1}, we obtain
\begin{equation}
\label{Lm2.1_proof_4}
\begin{split}
J:&=\int _{\mathbb{R}^d} \langle g(\mathbf{x})b(\mathbf{D})\mathbf{w}_k(\mathbf{x}),b(\mathbf{D})\mathbf{w}_k(\mathbf{x})\rangle \vert u(\mathbf{x})\vert ^2 \,d\mathbf{x}\\
&=-\sum \limits _{j=1}^d \int _{\mathbb{R}^d}\langle a_j(\mathbf{x})^*\mathbf{e}_ k, D_j \mathbf{w}_k(\mathbf{x})\rangle \vert u(\mathbf{x})\vert ^2\,d\mathbf{x}\\
&+\sum \limits _{l=1}^d \int _{\mathbb{R}^d} \langle g(\mathbf{x})b(\mathbf{D})\mathbf{w}_k(\mathbf{x}),b_l\mathbf{w}_k(\mathbf{x})\rangle (u^*D_l u+u D_l u^*)\,d\mathbf{x}\\
&+\sum \limits _{j=1}^d \int _{\mathbb{R}^d}\langle a_j(\mathbf{x})^*\mathbf{e}_k,\mathbf{w}_k(\mathbf{x})\rangle (u^*D_j u +uD_ju^*)\,d\mathbf{x}.
\end{split}
\end{equation}
Denote the consecutive summands in the right-hand side of \eqref{Lm2.1_proof_4} by $J_1$, $J_2$, and $J_3$. We have
\begin{equation*}
\begin{split}
\vert J_1\vert &\leqslant 4\alpha _0^{-1}\Vert g^{-1}\Vert _{L_\infty}\sum \limits _{j=1}^d \int _{\mathbb{R}^d}\vert a_j(\mathbf{x})^*\mathbf{e}_k\vert ^2\vert u(\mathbf{x})\vert ^2\,d\mathbf{x}\\
&+\frac{1}{4}\left(4\alpha _0^{-1}\Vert g^{-1}\Vert _{L_\infty}\right)^{-1}\int _{\mathbb{R}^d}\vert \mathbf{D}\mathbf{w}_k(\mathbf{x})\vert ^2 \vert u(\mathbf{x})\vert ^2\,d\mathbf{x}.
\end{split}
\end{equation*}
Using  condition \eqref{a_j} on the coefficients $a_j$,  we see that
\begin{equation}
\label{Lm2.1_proof_5}
\begin{split}
\int _{\mathbb{R}^d}\vert a_j(\mathbf{x})^*\mathbf{e}_k\vert ^2\vert u(\mathbf{x})\vert ^2\,d\mathbf{x}
&\leqslant \int _{\mathbb{R}^d}\vert a_j (\mathbf{x})\vert ^2\vert u(\mathbf{x})\vert ^2\,d\mathbf{x}
\\
&\leqslant C_{\Omega ,\rho}^2\Vert a_j\Vert ^2 _{L_\rho (\Omega)}\Vert u\Vert ^2 _{H^1(\mathbb{R}^d)},
\end{split}
\end{equation}
where $C_{\Omega ,\rho}$ is the norm of the embedding $H^1(\Omega)\subset L_{2\rho /(\rho -2)}(\Omega)$. Hence,
\begin{equation}
\label{Lm2.1_proof_6}
\begin{split}
\vert J_1\vert &\leqslant 4\alpha _0^{-1}\Vert g^{-1}\Vert _{L_\infty}C_{\Omega ,\rho}^2\sum \limits _{j=1}^d \Vert a_j\Vert ^2_{L_\rho (\Omega)}\Vert u\Vert ^2_{H^1(\mathbb{R}^d)}\\
&+\frac{1}{4}\left( 4\alpha _0^{-1}\Vert g^{-1}\Vert _{L_\infty}\right)^{-1}\int _{\mathbb{R}^d}\vert \mathbf{D}\mathbf{w}_k(\mathbf{x})\vert ^2\vert u(\mathbf{x})\vert ^2\,d\mathbf{x}.
\end{split}
\end{equation}
Next, by \eqref{bl<=} and \eqref{Lm2.1_proof_4}, we obtain
\begin{equation}
\label{Lm2.1_proof_7}
\begin{split}
\vert J_2\vert &\leqslant 2\int _{\mathbb{R}^d}\vert g(\mathbf{x})^{1/2}b(\mathbf{D})\mathbf{w}_k(\mathbf{x})\vert \vert u(\mathbf{x})\vert
\left(\sum \limits _{l=1}^d \vert g(\mathbf{x})^{1/2}b_l \mathbf{w}_k(\mathbf{x})\vert \vert D_l u(\mathbf{x})\vert \right)\,d\mathbf{x}\\
&\leqslant \frac{1}{2}J+2d\alpha _1\Vert g\Vert _{L_\infty}\int _{\mathbb{R}^d}\vert \mathbf{w}_k(\mathbf{x})\vert ^2\vert \mathbf{D}u(\mathbf{x})\vert ^2\,d\mathbf{x}.
\end{split}
\end{equation}
Finally, we estimate $J_3$:
\begin{equation*}
\begin{split}
\vert J_3\vert &\leqslant 2\sum \limits _{j=1}^d \int _{\mathbb{R}^d}\vert a_j(\mathbf{x})\vert \vert \mathbf{w}_k(\mathbf{x})\vert \vert u(\mathbf{x})\vert \vert D_ju(\mathbf{x})\vert \,d\mathbf{x}\\
&\leqslant \sum\limits _{j=1}^d \int _{\mathbb{R}^d}\vert a_j(\mathbf{x})\vert ^2 \vert u(\mathbf{x})\vert ^2\,d\mathbf{x}
+\int _{\mathbb{R}^d}\vert \mathbf{w}_k(\mathbf{x})\vert ^2 \vert \mathbf{D} u(\mathbf{x})\vert ^2\,d\mathbf{x}.
\end{split}
\end{equation*}
Taking  \eqref{Lm2.1_proof_5} into account, we obtain
\begin{equation}
\label{Lm2.1_proof_8}
\vert J_3\vert \leqslant C_{\Omega ,\rho}^2\sum \limits _{j=1}^d \Vert a_j\Vert ^2_{L_\rho (\Omega)}\Vert u\Vert ^2_{H^1(\mathbb{R}^d)}
+\int _{\mathbb{R}^d}\vert \mathbf{w}_k(\mathbf{x})\vert ^2\vert \mathbf{D}u(\mathbf{x})\vert ^2\,d\mathbf{x}.
\end{equation}
Combining \eqref{Lm2.1_proof_4} and \eqref{Lm2.1_proof_6}--\eqref{Lm2.1_proof_8}, we arrive at
\begin{equation}
\label{Lm2.1_proof_9}
\begin{split}
J&\leqslant 2\left(4\alpha _0^{-1}\Vert g^{-1}\Vert _{L_\infty}+1\right)C_{\Omega ,\rho}^2\sum \limits _{j=1}^d \Vert a_j\Vert ^2_{L_\rho (\Omega)}\Vert u\Vert ^2_{H^1(\mathbb{R}^d)}\\
&+\frac{1}{2}\left(4\alpha _0^{-1}\Vert g^{-1}\Vert _{L_\infty}\right)^{-1}\int _{\mathbb{R}^d}\vert \mathbf{D}\mathbf{w}_k(\mathbf{x})\vert ^2\vert u(\mathbf{x})\vert ^2\,d\mathbf{x}\\
&+2\left(2d\alpha _1\Vert g\Vert _{L_\infty}+1\right)\int _{\mathbb{R}^d}\vert \mathbf{w}_k(\mathbf{x})\vert ^2\vert \mathbf{D}u(\mathbf{x})\vert ^2\,d\mathbf{x}.
\end{split}
\end{equation}

Now we deduce the required estimate from \eqref{Lm2.1_proof_9}.
By \eqref{<b^*b<},
\begin{equation*}
\int _{\mathbb{R}^d}\vert \mathbf{D}(\mathbf{w}_k u)(\mathbf{x})\vert ^2\,d\mathbf{x}\leqslant \alpha _0^{-1}\int _{\mathbb{R}^d}\vert b(\mathbf{D})(\mathbf{w}_k u)(\mathbf{x})\vert ^2\,d\mathbf{x}.
\end{equation*}
According to \eqref{b(D)=}, $b(\mathbf{D})=\sum _{l=1}^d b_l D_l$, whence
\begin{equation*}
b(\mathbf{D})(\mathbf{w}_k u)=(b(\mathbf{D})\mathbf{w}_k)u+\sum \limits _{l=1}^d b_l\mathbf{w}_kD_l u.
\end{equation*}
Using \eqref{bl<=} and the expression for $J$ (see \eqref{Lm2.1_proof_4}), we have
\begin{equation}
\label{Lm2.1_proof_10}
\begin{split}
\int _{\mathbb{R}^d}&\vert \mathbf{D}(\mathbf{w}_k u)(\mathbf{x})\vert ^2\,d\mathbf{x}\leqslant 2\alpha _0^{-1}\int _{\mathbb{R}^d}\vert b(\mathbf{D})\mathbf{w}_k(\mathbf{x})\vert ^2 \vert u(\mathbf{x})\vert ^2\,d\mathbf{x}\\
&+2\alpha _0^{-1}\alpha _1 d\int _{\mathbb{R}^d}\vert \mathbf{w}_k(\mathbf{x})\vert ^2\vert \mathbf{D}u(\mathbf{x})\vert ^2\,d\mathbf{x}\\
&\leqslant 2\alpha _0^{-1}\Vert g^{-1}\Vert _{L_\infty}J+2\alpha _0^{-1}\alpha _1 d\int _{\mathbb{R}^d}\vert \mathbf{w}_k(\mathbf{x})\vert ^2\vert \mathbf{D}u(\mathbf{x})\vert ^2\,d\mathbf{x}.
\end{split}
\end{equation}
Obviously,
\begin{equation*}
\begin{split}
\int _{\mathbb{R}^d}\vert \mathbf{D}\mathbf{w}_k(\mathbf{x})\vert ^2\vert u(\mathbf{x})\vert ^2\,d\mathbf{x}
&\leqslant 2\int _{\mathbb{R}^d}\vert \mathbf{D}(\mathbf{w}_k u)(\mathbf{x})\vert ^2\,d\mathbf{x}\\
&+2\int _{\mathbb{R}^d}\vert \mathbf{w}_k(\mathbf{x})\vert ^2\vert \mathbf{D}u(\mathbf{x})\vert ^2\,d\mathbf{x}.
\end{split}
\end{equation*}
Combining this with \eqref{Lm2.1_proof_9} and \eqref{Lm2.1_proof_10}, we obtain
\begin{equation*}
\begin{split}
\int _{\mathbb{R}^d}&\vert \mathbf{D}\mathbf{w}_k(\mathbf{x})\vert ^2 \vert u(\mathbf{x})\vert ^2\,d\mathbf{x}\\
&\leqslant 16 \alpha _0^{-1}\Vert g^{-1}\Vert _{L_\infty}\left(4\alpha _0^{-1}\Vert g^{-1}\Vert _{L_\infty}+1\right)C_{\Omega ,\rho}^2\sum \limits _{j=1}^d \Vert a_j\Vert ^2 _{L_\rho (\Omega)}\Vert u\Vert ^2_{H^1(\mathbb{R}^d)}\\
&+\left(16\alpha _0^{-1}\Vert g^{-1}\Vert _{L_\infty}(2d\alpha _1\Vert g\Vert _{L_\infty}+1)+8\alpha _0^{-1}\alpha _1 d+4\right)\\
&\times\int _{\mathbb{R}^d}\vert \mathbf{w}_k(\mathbf{x})\vert ^2 \vert \mathbf{D}u(\mathbf{x})\vert ^2\,d\mathbf{x}.
\end{split}
\end{equation*}
Summing up over $k$, we arrive at \eqref{Lm2.1_tildeLambda} with
\begin{equation}
\label{Lm2.1_proof_11}
\begin{split}
&\widetilde{\beta} _1 =16 n\alpha _0^{-1}\Vert g^{-1}\Vert _{L_\infty}\left(4\alpha _0^{-1}\Vert g^{-1}\Vert _{L_\infty}+1\right)C_{\Omega ,\rho}^2\sum \limits _{j=1}^d \Vert a_j\Vert ^2 _{L_\rho (\Omega)},\\
&\widetilde{\beta} _2=16\alpha _0^{-1}\Vert g^{-1}\Vert _{L_\infty}(2d\alpha _1\Vert g\Vert _{L_\infty}+1)+8\alpha _0^{-1}\alpha _1 d+4.
\end{split}
\end{equation}
\end{proof}

By the scaling transformation, Lemma~\ref{Lemma on Lambda-tilda} implies the following result.

\begin{lemma}
\label{Lemma Lambda-tilda-varepsilon}
Under the assumptions of Lemma~\textnormal{\ref{Lemma on Lambda-tilda}}, for $0< \varepsilon \leqslant 1$ we have
\begin{equation*}
\begin{split}
\int _{\mathbb{R}^d}\vert (\mathbf{D}\widetilde{\Lambda})^\varepsilon(\mathbf{x})\vert ^2\vert u(\mathbf{x})\vert ^2\,d\mathbf{x}
\leqslant \widetilde{\beta} _1 \Vert u\Vert ^2_{H^1(\mathbb{R}^d)}
+\widetilde{\beta} _2 \varepsilon ^2\int _{\mathbb{R}^d}\vert \widetilde{\Lambda}^\varepsilon (\mathbf{x})\vert ^2 \vert \mathbf{D}u(\mathbf{x})\vert ^2 \,d\mathbf{x}.
\end{split}
\end{equation*}
\end{lemma}

Below in Section 7, we will need the following simple statement.

\begin{lemma}
\label{Lemma_Lambda_tilda3}
Let $f({\mathbf x})$ be a $\Gamma$-periodic function in  ${\mathbb R}^d$ such that
\begin{equation}
\label{condition_p}
f \in L_p(\Omega),\quad p=2\mbox{ for } d=1,\quad p>2 \mbox{ for } d=2,\quad p=d \mbox{ for } d\geqslant 3.
\end{equation}
Then for $0< \varepsilon \leqslant 1$ the operator $[f^\varepsilon]$ is a continuous mapping of $H^1({\mathbb R}^d)$ to
$L_2({\mathbb R}^d)$, and
$$
\| [f^\varepsilon] \|_{H^1({\mathbb R}^d) \to L_2({\mathbb R}^d)} \leqslant \|f\|_{L_p(\Omega)} C^{(p)}_{\Omega},
$$
where $C^{(p)}_{\Omega}$ is the norm of the embedding  $H^1(\Omega )\hookrightarrow L_{2(p/2)'}(\Omega)$.
Here \hbox{$(p/2)'= \infty$} for $d=1$, and $(p/2)'= p (p-2)^{-1}$ for $d\geqslant 2$.
\end{lemma}

\begin{proof}
Let $d\geqslant 2$, and let ${u} \in H^1({\mathbb R}^d)$.
Substituting ${\mathbf x}=\varepsilon {\mathbf y}$, ${u}({\mathbf x})= {v}({\mathbf y})$,
and using the H\"older inequality and the Sobolev embedding theorem, for $0< \varepsilon \leqslant 1$ we obtain
\begin{equation*}
\begin{split}
&\int_{{\mathbb R}^d} | f^\varepsilon ({\mathbf x}) |^2 | {u}({\mathbf x})|^2 \, d{\mathbf x}=
\varepsilon^d \int_{{\mathbb R}^d} | f ({\mathbf y})|^2| {v}({\mathbf y})|^2 \, d{\mathbf y}
\cr
&=
\varepsilon^d  \sum_{{\mathbf a}\in \Gamma} \int_{\Omega + {\mathbf a}}
| f({\mathbf y})|^2| {v}({\mathbf y})|^2 \, d{\mathbf y}
\cr
&\leqslant \varepsilon^d  \|f \|^2_{L_p(\Omega)} \sum_{{\mathbf a}\in \Gamma} \left(\int_{\Omega + {\mathbf a}}
| {v}({\mathbf y})|^{2(p/2)'} \, d{\mathbf y}\right)^{1/(p/2)'}
\cr
&\leqslant
\varepsilon^d  \|f \|^2_{L_p(\Omega)} (C_\Omega ^{(p)})^2
\|{v}\|^2_{H^1({\mathbb R}^d)}
\leqslant \|f \|^2_{L_p(\Omega)} (C_\Omega ^{(p)})^2
\|{u}\|^2_{H^1({\mathbb R}^d)}.
\end{split}
\end{equation*}
Here $(p/2)^{-1} + ((p/2)')^{-1}=1$. For the case where $d=1$ the proof is similar
(the necessary changes are obvious).
\end{proof}

Lemma \ref{Lemma Lambda-tilda-varepsilon} and Lemma \ref{Lemma_Lambda_tilda3}
directly imply the following corollary.

\begin{corollary}
\label{Lemma_Lambda_tilda4}
Suppose that $\widetilde{\Lambda}$ is the $\Gamma$-periodic solution of problem \eqref{tildeLambda_problem}.
Suppose also that $\widetilde{\Lambda}$ satisfies condition of the form
\eqref{condition_p}. Then for any $u \in H^2({\mathbb R}^d)$ and $0< \varepsilon \leqslant 1$ we have
\begin{equation*}
\begin{split}
\int _{\mathbb{R}^d}\vert (\mathbf{D}\widetilde{\Lambda})^\varepsilon(\mathbf{x})\vert ^2\vert u(\mathbf{x})\vert ^2\,d\mathbf{x}
\leqslant \widetilde{\beta} _1 \Vert u\Vert ^2_{H^1(\mathbb{R}^d)}
+\widetilde{\beta} _2 \varepsilon ^2
\| \widetilde{\Lambda} \|_{L_p(\Omega)}^2 (C_\Omega^{(p)})^2 \| \mathbf{D}u \|^2_{H^1({\mathbb R}^d)}.
\end{split}
\end{equation*}
\end{corollary}

\subsection{Lemma about $Q_0^\varepsilon -\overline{Q_0}$}
The following lemma will be needed for the proof of the main results.

\begin{lemma}
\label{Lemma Q0eps-mean value}
Suppose that $Q_0(\mathbf{x})$ is a $\Gamma$-periodic $(n\times n)$-matrix-valued function such that $Q_0 \in L_\infty$.
 Then for $\varepsilon >0$ the operator $[Q_0^\varepsilon -\overline{Q_0}]$ of multiplication by
the function  $Q_0^\varepsilon -\overline{Q_0}$ is a continuous mapping of
$H^1({\mathbb R}^d;{\mathbb C}^n)$ to $H^{-1}({\mathbb R}^d;{\mathbb C}^n)$, and
\begin{equation}
\label{lemma Q 0 eps -Q}
\Vert [Q_0^\varepsilon -\overline{Q_0}]\Vert _{H^1(\mathbb{R}^d)\rightarrow H^{-1}(\mathbb{R}^d)}\leqslant C_{Q_0}\varepsilon,
\quad \varepsilon >0.
\end{equation}
The constant $C_{Q_0}$ is controlled in terms of $d$, $\Vert Q_0\Vert _{L_\infty}$, and the parameters of the lattice $\Gamma$.
\end{lemma}

\begin{proof}
Since $Q_0-\overline{Q_0}\in L_\infty$ and
\begin{equation}
\label{int Q0-overline Q0}
\int _\Omega (Q_0(\mathbf{x})-\overline{Q_0})\,d\mathbf{x}=0,
\end{equation}
we have the following representation
\begin{equation}
\label{tozd h_j}
Q^\varepsilon _0(\mathbf{x})-\overline{Q}_0=-\varepsilon \sum _{j=1}^d D_jh_j^\varepsilon (\mathbf{x}),
\end{equation}
where $h_j$, $j=1,\dots,d,$ are $\Gamma$-periodic $(n\times n)$-matrix-valued functions such that $h_j\in L_\infty$.

Indeed, let $\Phi (\mathbf{x})$ be the $\Gamma$-periodic solution of the problem
\begin{equation*}
\bigtriangleup \Phi (\mathbf{x})=Q_0(\mathbf{x})-\overline{Q_0},\quad \int _\Omega \Phi (\mathbf{x})\,d\mathbf{x}=0.
\end{equation*}
By \eqref{int Q0-overline Q0}, the solvability condition is satisfied.
Since $Q_0-\overline{Q_0}\in L_\infty$, then $\Phi \in W_p^2(\Omega)$ for any $1\leqslant p<\infty$, and
\begin{equation}
\label{5.3.star}
\Vert \Phi \Vert _{W_p^2 (\Omega)}\leqslant \mathfrak{c}_1(p)\Vert Q_0 -\overline{Q_0}\Vert _{L_p(\Omega)}\leqslant \widetilde{\mathfrak{c}}_1(p)
\Vert Q_0\Vert _{L_\infty}.
\end{equation}
The constants $\mathfrak{c}_1(p)$ and $\widetilde{\mathfrak{c}}_1(p)$ depend only on $p$ and the parameters of the lattice $\Gamma$.
This follows from the Marcinkiewicz theorem about the Fourier multipliers \cite{Ma}.
We put $h_j(\mathbf{x})=D_j\Phi (\mathbf{x})$. Then
$$
Q_0(\mathbf{x})-\overline{Q_0}=-\sum _{j=1}^d D_jh_j(\mathbf{x}),
$$
and $h_j\in W_p^1(\Omega)$ for any $1\leqslant p<\infty$.
Let $p=d+1$. Then, by the embedding theorem,  $h_j\in L_\infty$. Combining this with \eqref{5.3.star}, we obtain
\begin{equation*}
\Vert h_j\Vert _{L_\infty}\leqslant \mathfrak{c}_2\Vert h_j\Vert _{W^1_p(\Omega)}\leqslant \mathfrak{c}_2\Vert \Phi \Vert _{W^2_p(\Omega)}\leqslant \mathfrak{c}_2\widetilde{\mathfrak{c}}_1(d+1)\Vert Q_0\Vert _{L_\infty}.
\end{equation*}
(Here $\mathfrak{c}_2$ is the norm of the corresponding embedding.)

Let $\mathbf{F}\in H^1(\mathbb{R}^d;\mathbb{C}^n)$. By \eqref{tozd h_j},
\begin{equation}
\label{Vert Q0-overlineQ0...=eps...}
\begin{split}
\Vert (Q^\varepsilon _0-\overline{Q_0})\mathbf{F}\Vert _{H^{-1}(\mathbb{R}^d)}
&=\sup \limits _{0\neq \mathbf{v}\in H^1(\mathbb{R}^d;\mathbb{C}^n)} \frac{\left\vert \left((Q^\varepsilon _0-\overline{Q_0})\mathbf{F},\mathbf{v}\right)_{L_2(\mathbb{R}^d)}\right\vert}{\Vert \mathbf{v}\Vert _{H^1(\mathbb{R}^d)}}
\\
&\leqslant \varepsilon \sup \limits _{0\neq \mathbf{v}\in H^1(\mathbb{R}^d;\mathbb{C}^n)}\frac{\sum \limits_{j=1}^d\left\vert\left((D_j h_j^\varepsilon)\mathbf{F},\mathbf{v}\right)_{L_2(\mathbb{R}^d)}\right\vert}{\Vert \mathbf{v}\Vert _{H^1(\mathbb{R}^d)}}.
\end{split}
\end{equation}
Using the identity $(D_jh_j^\varepsilon)\mathbf{F}=D_j(h_j^\varepsilon\mathbf{F})-h_j^\varepsilon D_j\mathbf{F}$
and integrating by parts, we obtain
\begin{equation*}
\begin{split}
\left((D_j h_j^\varepsilon)\mathbf{F},\mathbf{v}\right)_{L_2(\mathbb{R}^d)}
=\left(h_j^\varepsilon \mathbf{F},D_j\mathbf{v}\right)_{L_2(\mathbb{R}^d)}
-\left(h_j^\varepsilon D_j\mathbf{F},\mathbf{v}\right)_{L_2(\mathbb{R}^d)},\quad j=1,\dots ,d.
\end{split}
\end{equation*}
Together with \eqref{Vert Q0-overlineQ0...=eps...} this yields
\begin{equation}
\label{Vert Q0 after int by parts}
\begin{split}
\Vert (Q^\varepsilon _0-\overline{Q_0})\mathbf{F}\Vert _{H^{-1}(\mathbb{R}^d)}
&\leqslant \varepsilon  \sup \limits _{0\neq \mathbf{v}\in H^1(\mathbb{R}^d;\mathbb{C}^n)}\frac{ \sum \limits_{j=1}^d \left\vert\left(h_j^\varepsilon \mathbf{F},D_j\mathbf{v}\right)_{L_2(\mathbb{R}^d)}\right\vert}{\Vert \mathbf{v}\Vert _{H^1(\mathbb{R}^d)}}\\
&+\varepsilon  \sup \limits _{0\neq \mathbf{v}\in H^1(\mathbb{R}^d;\mathbb{C}^n)}\frac{ \sum \limits _{j=1}^d \left\vert \left(h_j^\varepsilon D_j\mathbf{F},\mathbf{v}\right)_{L_2(\mathbb{R}^d)}\right\vert}{\Vert \mathbf{v}\Vert _{H^1(\mathbb{R}^d)}}.
\end{split}
\end{equation}
Obviously,
\begin{align}
\label{Vert Q0 ... pervoe slag}
\begin{split}
\sum _{j=1}^d\left\vert \left( h_j^\varepsilon \mathbf{F},D_j\mathbf{v}\right)_{L_2(\mathbb{R}^d)}\right\vert
\leqslant C_h \Vert \mathbf{F}\Vert _{L_2(\mathbb{R}^d)}\Vert \mathbf{D} \mathbf{v}\Vert _{L_2(\mathbb{R}^d)},
\end{split}\\
\label{Vert Q0 part of second slagaemoe}
\begin{split}
\sum _{j=1}^d\left\vert \left( h_j^\varepsilon D_j\mathbf{F},\mathbf{v}\right)_{L_2(\mathbb{R}^d)}\right\vert
\leqslant C_h\Vert \mathbf{D}\mathbf{F}\Vert _{L_2(\mathbb{R}^d)}\Vert \mathbf{v}\Vert _{L_2(\mathbb{R}^d)},
\end{split}
\end{align}
where $C_h^2:=\esssup\limits_{\mathbf{x}\in\mathbb{R}^d}\,\sum \limits _{j=1}^d \vert h_j (\mathbf{x})\vert ^2 .$
Note that  $C_h\leqslant \mathfrak{c}_3 \Vert Q_0\Vert _{L_\infty}$, where the constant $\mathfrak{c}_3$ depends only on $d$
and the parameters of the lattice $\Gamma$.

Relations \eqref{Vert Q0 after int by parts}--\eqref{Vert Q0 part of second slagaemoe} imply
 \eqref{lemma Q 0 eps -Q} with the constant $C_{Q_0}=2C_h$.
\end{proof}

\section{The Steklov smoothing. Another approximation \\ of the generalized resolvent $(B_\varepsilon +\lambda_0 Q_0^\varepsilon )^{-1}$}
\label{Section Steklov smoothing}
\setcounter{section}{3}
\setcounter{equation}{0}
\setcounter{theorem}{0}

\subsection{The Steklov smoothing operator} The operator $S_\varepsilon^{(k)}$, $\varepsilon >0$,
acting in $L_2(\mathbb{R}^d;\mathbb{C}^k)$ (where $k\in\mathbb{N}$) and defined by
\begin{equation}
\label{S_eps}
\begin{split}
(S_\varepsilon^{(k)} \mathbf{u})(\mathbf{x})=\vert \Omega \vert ^{-1}\int _\Omega \mathbf{u}(\mathbf{x}-\varepsilon \mathbf{z})\,d\mathbf{z},\quad \mathbf{u}\in L_2(\mathbb{R}^d;\mathbb{C}^k),
\end{split}
\end{equation}
is called the Steklov smoothing operator.
We will omit the index $k$ in the notation and write simply $S_\varepsilon$.
Obviously, $S_\varepsilon \mathbf{D}^\alpha \mathbf{u}=\mathbf{D}^\alpha S_\varepsilon \mathbf{u}$ for $\mathbf{u}\in H^s(\mathbb{R}^d;\mathbb{C}^k)$  and any multiindex $\alpha$ such that $\vert \alpha\vert \leqslant s$.

We need the following properties of the operator $S_\varepsilon$
(see \cite[Lemmas 1.1 and 1.2]{ZhPas} or \cite[Propositions 3.1 and 3.2]{PSu}).

\begin{proposition}
\label{Proposition S_eps - I}
For any $\mathbf{u}\in H^1(\mathbb{R}^d;\mathbb{C}^k)$
and $\varepsilon >0$ we have
\begin{equation*}
\Vert S_\varepsilon \mathbf{u}-\mathbf{u}\Vert _{L_2(\mathbb{R}^d)}\leqslant \varepsilon r_1\Vert \mathbf{D}\mathbf{u}\Vert _{L_2(\mathbb{R}^d)},
\end{equation*}
where $2r_1=\mathrm{diam}\,\Omega$.
\end{proposition}

\begin{proposition}
\label{Proposition S_eps L2->L2}
Let $f$ be a  $\Gamma$-periodic function in $\mathbb{R}^d$ such that $f\in L_2(\Omega)$.
Then the operator $[f^\varepsilon ]S_\varepsilon $ is continuous in  $L_2(\mathbb{R}^d)$, and
\begin{equation*}
\Vert [f^\varepsilon]S_\varepsilon \Vert _{L_2(\mathbb{R}^d)\rightarrow L_2(\mathbb{R}^d)}\leqslant \vert \Omega \vert ^{-1/2}\Vert f\Vert _{L_2(\Omega)}, \quad \varepsilon >0.
\end{equation*}
\end{proposition}

\subsection{Another approximation of the operator $(B_\varepsilon +\lambda_0 Q_0^\varepsilon)^{-1}$} We put
\begin{equation}
\label{tilde K(eps)}
\widetilde{K} (\varepsilon)=\left( [\Lambda ^\varepsilon]b(\mathbf{D})+[\widetilde{\Lambda}^\varepsilon ]\right) S_\varepsilon
(B^0+\lambda_0 \overline{Q_0})^{-1}.
\end{equation}
The operator $\widetilde{K}(\varepsilon )$ is a continuous mapping of $L_2(\mathbb{R}^d;\mathbb{C}^n)$ to $H^1(\mathbb{R}^d;\mathbb{C}^n)$.
This follows from Proposition \ref{Proposition S_eps L2->L2} and the relations $\Lambda ,\widetilde{\Lambda}\in\widetilde{H}^1(\Omega)$.

Along with Theorem \ref{Theorem Su L2->H1}, the following result is true.

\begin{theorem}
\label{Theorem S_eps z fixed}
Suppose that the assumptions of Theorem~\textnormal{\ref{Theorem Su L2}} are satisfied.
Let $\widetilde{K}(\varepsilon)$ be defined by \eqref{tilde K(eps)}. Then for $0<\varepsilon \leqslant 1$ we have
\begin{equation*}
\Vert (B_\varepsilon +\lambda_0 Q_0^\varepsilon )^{-1}-(B^0+\lambda_0 \overline{Q_0})^{-1}-\varepsilon \widetilde{K}(\varepsilon )\Vert _{L_2(\mathbb{R}^d)\rightarrow H^1(\mathbb{R}^d)}\leqslant C_3\varepsilon.
\end{equation*}
The constant $C_3$ depends only on the initial data \eqref{problem data}.
\end{theorem}

\begin{remark}
\label{Remark on smoothing}
Theorems \textnormal{\ref{Theorem Su L2->H1}} and \textnormal{\ref{Theorem S_eps z fixed}} show that
for homogenization problem in ${\mathbb R}^d$ different smoothing operators can be involved in the corrector
{\rm (}both $\Pi_\varepsilon$ and $S_\varepsilon$ are suitable{\rm )}.
However, for homogenization problems in a bounded domain {\rm (}see, e.~g., {\rm \cite{ZhPas}, \cite{PSu}, \cite{Su13, Su_SIAM, Su15})},
it is more convenient to use the Steklov smoothing. Since we are aimed on application of the results of the present paper to
study homogenization problems in a bounded domain  \textnormal{\cite{MSu_prep}}, we have turned to the Steklov smoothing.
\end{remark}

\begin{remark}
\label{Remark Constants C5}
Under the assumptions of Remark~{\rm \ref{Remark C1 C2 in Theorems by Su}}, the constant
$C_3$ can be taken independent of the parameter $\vartheta \in (0,1]$.
\end{remark}

\subsection{Proof of Theorem~\ref{Theorem S_eps z fixed}}
We deduce Theorem~\ref{Theorem S_eps z fixed} from Theorem~\ref{Theorem Su L2->H1}.

\begin{lemma}
\label{Lemma need for proof theorem 3.3}
For any $u\in H^1(\mathbb{R}^d)$ and $0<\varepsilon\leqslant 1$ we have
\begin{equation}
\label{lemma tilde Lambda u in H1}
\begin{split}
\int _{\mathbb{R}^d}&\vert (\mathbf{D}\widetilde{\Lambda})^\varepsilon (\mathbf{x})\vert ^2 \vert (\Pi _\varepsilon -S_\varepsilon )u\vert ^2\,d\mathbf{x}
\leqslant \widetilde{\beta }_1\Vert (\Pi _\varepsilon -S_\varepsilon )u\Vert ^2 _{H^1(\mathbb{R}^d)}\\
&+\widetilde{\beta}_2\varepsilon ^2 \int _{\mathbb{R}^d}\vert \widetilde{\Lambda}^\varepsilon (\mathbf{x}) \vert ^2\vert (\Pi _\varepsilon -S_\varepsilon )\mathbf{D}u\vert ^2\,d\mathbf{x}.
\end{split}
\end{equation}
\end{lemma}

\begin{proof}
By Propositions \ref{Proposition Pi_eps L2->L2}, \ref{Proposition S_eps L2->L2}
and the relation $\widetilde{\Lambda}\in \widetilde{H}^1(\mathbb{R}^d)$, all the terms in  \eqref{lemma tilde Lambda u in H1}
are continuous functionals of $u$ in the $H^1(\mathbb{R}^d)$-norm.
Therefore, it suffices to check \eqref{lemma tilde Lambda u in H1} for $u\in C_0^\infty (\mathbb{R}^d)$.

We fix a function $F\in C ^\infty(\mathbb{R}_+)$ such that
$0\leqslant F(t)\leqslant 1$, $F(t)=1$ for $0\leqslant t\leqslant 1$, and $F(t)=0$ for $t\geqslant 2$.
Next, define the function $F_R(\mathbf{x}):=F(\vert \mathbf{x}\vert /R)$ in $\mathbb{R}^d$ depending on the parameter $R>0$.
If  $u\in C_0^\infty (\mathbb{R}^d)$, then $F_R(\Pi _\varepsilon -S_\varepsilon )u\in C_0^\infty (\mathbb{R}^d)$.
Hence, by Lemma \ref{Lemma Lambda-tilda-varepsilon},
\begin{equation*}
\begin{split}
\int _{\mathbb{R}^d}& \vert (\mathbf{D}\widetilde{\Lambda})^\varepsilon \vert ^2\vert F_R(\Pi _\varepsilon -S_\varepsilon )u\vert ^2\,d\mathbf{x}
\leqslant \widetilde{\beta}_1\Vert F_R(\Pi _\varepsilon -S_\varepsilon )u\Vert ^2_{L_2(\mathbb{R}^d)}\\
&+\widetilde{\beta}_1\sum _{j=1}^d\int _{\mathbb{R}^d}\vert (\partial _j F_R)(\Pi _\varepsilon -S_\varepsilon )u +F_R(\Pi _\varepsilon -S_\varepsilon )\partial _j u\vert ^2\,d\mathbf{x}\\
&+\widetilde{\beta}_2\varepsilon ^2 \sum _{j=1}^d\int _{\mathbb{R}^d}\vert \widetilde{\Lambda}^\varepsilon \vert ^2\vert (\partial _j F_R)(\Pi _\varepsilon -S_\varepsilon )u+F_R(\Pi _\varepsilon -S_\varepsilon )\partial _j u\vert ^2\,d\mathbf{x}.
\end{split}
\end{equation*}
Using the estimate $\max \vert \partial _jF_R\vert \leqslant c /R$ and passing to the limit as $R\rightarrow \infty$,  we deduce \eqref{lemma tilde Lambda u in H1} with the help of the Lebesgue theorem.
\end{proof}

Now we are ready to prove Theorem \ref{Theorem S_eps z fixed}. Obviously,
\begin{equation}
\label{K(eps)-tilde K(eps)}
\begin{split}
\varepsilon\Vert K(\varepsilon)-\widetilde{K}(\varepsilon)\Vert _{L_2\rightarrow H^1}&\leqslant \varepsilon\Vert [\Lambda ^\varepsilon ](\Pi _\varepsilon -S_\varepsilon )b(\mathbf{D})(B^0+\lambda_0 \overline{Q_0})^{-1}\Vert _{L_2\rightarrow H^1}\\
&+\varepsilon\Vert [\widetilde{\Lambda }^\varepsilon] (\Pi _\varepsilon -S_\varepsilon )(B^0+\lambda_0 \overline{Q_0})^{-1}
\Vert _{L_2\rightarrow H^1}.
\end{split}
\end{equation}
For the first term in the right-hand side of \eqref{K(eps)-tilde K(eps)}, we have:
\begin{equation}
\label{lm k-tilde k 1}
\begin{split}
\varepsilon &\Vert [\Lambda ^\varepsilon ](\Pi _\varepsilon -S_\varepsilon )b(\mathbf{D})(B^0+\lambda_0 \overline{Q_0})^{-1}\Vert _{L_2\rightarrow H^1}\\
&\leqslant \varepsilon \Vert [\Lambda ^\varepsilon ](\Pi _\varepsilon -S_\varepsilon )b(\mathbf{D})\Vert _{H^2\rightarrow H^1}\Vert
(B^0+\lambda_0 \overline{Q_0})^{-1}\Vert _{L_2\rightarrow H^2}.
\end{split}
\end{equation}
By \eqref{L>=},
  \begin{equation}
\label{B_0 res L2 ->H2}
\Vert (B^0+\lambda_0 \overline{Q_0})^{-1}\Vert _{L_2(\mathbb{R}^d)\rightarrow H^2(\mathbb{R}^d)} =
\sup \limits _{\boldsymbol{\xi}\in\mathbb{R}^d}(\vert \boldsymbol{\xi}\vert ^2+1)\vert L(\boldsymbol{\xi})^{-1}\vert \leqslant \check{c}_*^{-1}.
\end{equation}
According to  \cite[Lemma 3.5]{PSu},
\begin{equation}
\label{lemma 3.5 PSu}
\varepsilon\Vert [\Lambda ^\varepsilon ](\Pi _\varepsilon -S_\varepsilon )b(\mathbf{D})\Vert _{H^2(\mathbb{R}^d)\rightarrow H^1(\mathbb{R}^d)}\leqslant C_\Lambda \varepsilon,
\end{equation}
where the constant $C_\Lambda $ depends only on $m$, $d$, $\Vert g\Vert _{L_\infty}$, $\Vert g^{-1}\Vert _{L_\infty}$, $\alpha _0$, $\alpha _1$,
and the parameters of the lattice $\Gamma$. Combining \eqref{lm k-tilde k 1}--\eqref{lemma 3.5 PSu}, we arrive at
\begin{equation}
\label{lemma k-tilde k 1}
\varepsilon\Vert [\Lambda ^\varepsilon ](\Pi _\varepsilon -S_\varepsilon )b(\mathbf{D})(B^0+\lambda_0 \overline{Q_0})^{-1}\Vert _{L_2(\mathbb{R}^d)\rightarrow H^1(\mathbb{R}^d)}\leqslant
\check{c}_*^{-1} C_\Lambda \varepsilon  .
\end{equation}

Now we estimate the second term in the right-hand side of \eqref{K(eps)-tilde K(eps)}. By \eqref{B_0 res L2 ->H2},
\begin{equation}
\label{lm k-tilde k 2}
\varepsilon \Vert [\widetilde{\Lambda }^\varepsilon] (\Pi _\varepsilon -S_\varepsilon )(B^0+\lambda_0 \overline{Q_0})^{-1}\Vert _{L_2\rightarrow H^1}\leqslant \varepsilon \check{c}_*^{-1}\Vert [\widetilde{\Lambda }^\varepsilon] (\Pi _\varepsilon -S_\varepsilon )\Vert _{H^2\rightarrow H^1}.
\end{equation}
Let $\boldsymbol{\Phi}\in H^2 (\mathbb{R}^d;\mathbb{C}^n)$. Then
\begin{equation}
\label{eps ... Phi <=}
\begin{split}
\Vert &\varepsilon [\widetilde{\Lambda}^\varepsilon](\Pi _\varepsilon -S_\varepsilon)\boldsymbol{\Phi}\Vert _{H^1(\mathbb{R}^d)}
\leqslant
\varepsilon \Vert [\widetilde{\Lambda}^\varepsilon ](\Pi _\varepsilon -S_\varepsilon )\boldsymbol{\Phi}\Vert _{L_2(\mathbb{R}^d)}
\\
&+\Vert [(\mathbf{D}\widetilde{\Lambda})^\varepsilon ](\Pi _\varepsilon -S_\varepsilon )\boldsymbol{\Phi}\Vert _{L_2(\mathbb{R}^d)}
+\varepsilon \Vert [\widetilde{\Lambda}^\varepsilon ](\Pi _\varepsilon -S_\varepsilon )\mathbf{D}\boldsymbol{\Phi}\Vert _{L_2(\mathbb{R}^d)}.
\end{split}
\end{equation}
Using Propositions \ref{Proposition Pi_eps L2->L2}, \ref{Proposition S_eps L2->L2} and taking \eqref{tildeLambda <=}
into account, we obtain
\begin{align}
\label{tildeLambda Pi_eps}
&\Vert [\widetilde{\Lambda}^\varepsilon ]\Pi _\varepsilon \Vert _{L_2\rightarrow L_2}
\leqslant \vert \Omega \vert ^{-1/2}(2r_0)^{-1}C_a n^{1/2}\alpha _0^{-1}\Vert g^{-1}\Vert _{L_\infty}=:\widetilde{M}_1,
\ \varepsilon >0,\\
\label{tildeLambda S_eps}
&\Vert [\widetilde{\Lambda}^\varepsilon ]S _\varepsilon \Vert _{L_2\rightarrow L_2}\leqslant \widetilde{M}_1,
\quad \varepsilon >0.
\end{align}
By \eqref{tildeLambda Pi_eps} and \eqref{tildeLambda S_eps},
\begin{align}
\label{tilde Lambda (Pi-S)Phi}
&\Vert [\widetilde{\Lambda }^\varepsilon ](\Pi _\varepsilon -S_\varepsilon)\boldsymbol{\Phi}\Vert _{L_2(\mathbb{R}^d)}\leqslant 2\widetilde{M}_1\Vert \boldsymbol{\Phi}\Vert _{L_2(\mathbb{R}^d)},\\
\label{tilde Lambda (Pi-S)DPhi}
&\Vert [\widetilde{\Lambda }^\varepsilon ](\Pi _\varepsilon -S_\varepsilon)\mathbf{D}\boldsymbol{\Phi}\Vert _{L_2(\mathbb{R}^d)}\leqslant 2\widetilde{M}_1\Vert \mathbf{D}\boldsymbol{\Phi}\Vert _{L_2(\mathbb{R}^d)}.
\end{align}
For $0< \varepsilon \leqslant 1$, the second term in the right-hand side of \eqref{eps ... Phi <=} is estimated with the help of
Lemma~\ref{Lemma need for proof theorem 3.3}:
\begin{equation}
\label{DtildeLambda ... Phi<=}
\begin{split}
\Vert (\mathbf{D}\widetilde{\Lambda})^\varepsilon (\Pi _\varepsilon -S_\varepsilon )\boldsymbol{\Phi}\Vert ^2_{L_2(\mathbb{R}^d)}&\leqslant \widetilde{\beta} _1\Vert (\Pi _\varepsilon -S_\varepsilon )\boldsymbol{\Phi}\Vert ^2 _{H^1(\mathbb{R}^d)}\\
&+\widetilde{\beta} _2\varepsilon ^2\int _{\mathbb{R}^d}\vert \widetilde{\Lambda}^\varepsilon\vert ^2\vert (\Pi _\varepsilon -S_\varepsilon)\mathbf{D}\boldsymbol{\Phi}\vert ^2\,d\mathbf{x}.
\end{split}
\end{equation}
Using Propositions  \ref{Proposition Pi_eps -I} and \ref{Proposition S_eps - I}, we obtain
\begin{equation*}
\Vert (\Pi _\varepsilon -S_\varepsilon)\boldsymbol{\Phi}\Vert _{H^1(\mathbb{R}^d)}\leqslant \varepsilon (r_0^{-1}+r_1 )\Vert \boldsymbol{\Phi}\Vert _{H^2(\mathbb{R}^d)}.
\end{equation*}
Together with \eqref{tilde Lambda (Pi-S)DPhi} and \eqref{DtildeLambda ... Phi<=} this implies
\begin{equation}
\label{DtildeLambda...Phi<=eps...}
\Vert (\mathbf{D}\widetilde{\Lambda})^\varepsilon (\Pi _\varepsilon -S_\varepsilon)\boldsymbol{\Phi}\Vert _{L_2(\mathbb{R}^d)}
\leqslant \varepsilon \left(\widetilde{\beta} _1(r_0^{-1}+r_1 )^2+4\widetilde{\beta} _2\widetilde{M}_1^2\right)^{1/2}\Vert \boldsymbol{\Phi}\Vert _{H^2(\mathbb{R}^d)}.
\end{equation}
Combining \eqref{eps ... Phi <=}, \eqref{tilde Lambda (Pi-S)Phi}, \eqref{tilde Lambda (Pi-S)DPhi}, and \eqref{DtildeLambda...Phi<=eps...},
we conclude that
\begin{equation}
\label{eps tilde Lambda (Pi-S)<=C eps}
\Vert \varepsilon [\widetilde{\Lambda}^\varepsilon ](\Pi _\varepsilon -S_\varepsilon )\Vert _{H^2(\mathbb{R}^d)\rightarrow H^1(\mathbb{R}^d)}\leqslant C_{\widetilde{\Lambda}}\varepsilon,
\end{equation}
where $C_{\widetilde{\Lambda}}=4\widetilde{M}_1+ \left(\widetilde{\beta} _1(r_0^{-1}+r_1 )^2+4\widetilde{\beta} _2\widetilde{M}_1^2\right)^{1/2}$. Relations \eqref{lm k-tilde k 2} and \eqref{eps tilde Lambda (Pi-S)<=C eps} yield
\begin{equation}
\label{lemma k-tilde k 2}
 \varepsilon \Vert [\widetilde{\Lambda }^\varepsilon] (\Pi _\varepsilon -S_\varepsilon )(B^0+\lambda_0 \overline{Q_0})^{-1}\Vert _{L_2(\mathbb{R}^d)\rightarrow H^1(\mathbb{R}^d)}\leqslant \check{c}_*^{-1}C_{\widetilde{\Lambda}}\varepsilon.
\end{equation}

Finally, applying Theorem \ref{Theorem Su L2->H1} and estimates \eqref{K(eps)-tilde K(eps)}, \eqref{lemma k-tilde k 1}, \eqref{lemma k-tilde k 2},
we complete the proof.

\qed

\section{Main results}
\label{Section main results}
\setcounter{section}{4}
\setcounter{equation}{0}
\setcounter{theorem}{0}

\subsection{Formulation of the results}
In the present subsection, we formulate the main results of the paper.

\begin{theorem}
\label{Theorem main result L2}
Suppose that the assumptions of Subsections~\textnormal{\ref{Subsection Operator A}--\ref{Subsection generalized resolvent}}
are satisfied.
Let \hbox{$\zeta \in\mathbb{C}\setminus \mathbb{R}_+$}, $\zeta =\vert \zeta \vert e^{i\phi}$, $\phi \in (0,2\pi)$, and let $\vert \zeta \vert \geqslant 1$.
We put
\begin{equation}
\label{c(phi)}
c(\phi)=\begin{cases}
\vert \sin \phi \vert ^{-1}, &\phi\in (0,\pi /2)\cup (3\pi /2 ,2\pi),\\
1, &\phi\in [\pi /2,3\pi /2].
\end{cases}
\end{equation}
Then for $0<\varepsilon \leqslant 1$ we have
\begin{equation}
\label{Th L2 principal part}
\Vert(B_\varepsilon -\zeta Q_0^\varepsilon )^{-1}-(B^0-\zeta\overline{Q_0})^{-1}\Vert _{L_2(\mathbb{R}^d)\rightarrow L_2(\mathbb{R}^d)}\leqslant C_4\varepsilon c(\phi)^2\vert \zeta \vert ^{-1/2}.
\end{equation}
The constant $C_4$ depends only on the initial data \eqref{problem data}.
\end{theorem}

To formulate the result about approximation of the operator $(B_\varepsilon -\zeta Q_0^\varepsilon )^{-1}$
in the $(L_2\rightarrow H^1)$-operator norm, we introduce a corrector
\begin{equation}
\label{K(eps;zeta)}
K(\varepsilon ;\zeta)=\bigl([\Lambda ^\varepsilon]b(\mathbf{D})+[\widetilde{\Lambda}^\varepsilon ]\bigr)S_\varepsilon (B^0 -\zeta \overline{Q_0})^{-1}.
\end{equation}
Here $S_\varepsilon$ is the Steklov smoothing operator given by \eqref{S_eps}.
The operator \eqref{K(eps;zeta)} is a continuous mapping of $L_2(\mathbb{R}^d;\mathbb{C}^n)$ to $H^1(\mathbb{R}^d;\mathbb{C}^n)$.
This follows from Proposition \ref{Proposition S_eps L2->L2} and the relations $\Lambda ,\,\widetilde{\Lambda}\in \widetilde{H}^1(\Omega)$.

\begin{theorem}
\label{Theorem main result H1}
Suppose that the assumptions of Theorem \textnormal{\ref{Theorem main result L2}} are satisfied.
Let $K(\varepsilon ;\zeta)$ be the operator given by \eqref{K(eps;zeta)}.
Then for  $0<\varepsilon \leqslant 1$ and \hbox{$\zeta\in \mathbb{C}\setminus\mathbb{R}_+$}, $\vert \zeta \vert \geqslant 1$, we have
\begin{align}
\label{Th 3.2 L_2->L_2}
\begin{split}
\Vert (B_\varepsilon -\zeta Q_0^\varepsilon )^{-1}-(B^0-\zeta \overline{Q_0})^{-1}-\varepsilon K(\varepsilon ;\zeta )\Vert _{L_2(\mathbb{R}^d)\rightarrow L_2(\mathbb{R}^d)}
\\
\leqslant C_5 c(\phi )^2\varepsilon \vert \zeta \vert ^{-1/2},
\end{split}
\\
\label{Th 3.2 D L_2->L_2}
\begin{split}
\Vert \mathbf{D}\left(
(B_\varepsilon -\zeta Q_0^\varepsilon )^{-1}-(B^0-\zeta \overline{Q_0})^{-1}-\varepsilon K(\varepsilon ;\zeta )
\right)
\Vert _{L_2(\mathbb{R}^d)\rightarrow L_2(\mathbb{R}^d)}
\\
\leqslant C_6 c(\phi )^2\varepsilon .
\end{split}
\end{align}
The constants $C_5$ and $C_6$ are controlled in terms of the problem data \eqref{problem data}.
\end{theorem}

Theorem \ref{Theorem main result H1} directly implies the following corollary.

\begin{corollary}
\label{Corollary main result L2->H1}
Under the assumptions of Theorem \textnormal{\ref{Theorem main result H1}}, we have
\begin{equation*}
\begin{split}
\Vert &(B_\varepsilon -\zeta Q_0^\varepsilon )^{-1}-(B^0-\zeta \overline{Q_0})^{-1}-\varepsilon K(\varepsilon ;\zeta )\Vert _{L_2(\mathbb{R}^d)\rightarrow H^1(\mathbb{R}^d)}
\\
&\leqslant (C_5+C_6)c(\phi )^2\varepsilon ,\quad 0<\varepsilon \leqslant 1.
\end{split}
\end{equation*}
\end{corollary}

Using Theorem \ref{Theorem main result H1}, we obtain the result about approximation of the operator
$g^\varepsilon b(\mathbf{D})(B_\varepsilon -\zeta Q_0^\varepsilon )^{-1}$ (corresponding to the ,,flux'').

\begin{theorem}
\label{Theorem main result fluxes}
Suppose that the assumptions of Theorem \textnormal{\ref{Theorem main result L2}} are satisfied.
Let $\widetilde{g}(\mathbf{x})$ be the matrix-valued function defined by~\eqref{tilde g}. Denote
\begin{equation}
\label{G(eps,zeta)}
G(\varepsilon ;\zeta):=\widetilde{g}^\varepsilon S_\varepsilon b(\mathbf{D})(B^0-\zeta \overline{Q_0})^{-1}+g^\varepsilon \bigl(b(\mathbf{D})\widetilde{\Lambda}\bigr
)^\varepsilon S_\varepsilon (B^0-\zeta \overline{Q_0})^{-1}.
\end{equation}
Then for $0<\varepsilon \leqslant 1$ and $\zeta \in\mathbb{C}\setminus\mathbb{R}_+$, $\vert \zeta \vert \geqslant 1$, we have
\begin{equation}
\label{Th 4.4}
\Vert g^\varepsilon b(\mathbf{D})(B_\varepsilon -\zeta Q_0^\varepsilon )^{-1}-G(\varepsilon;\zeta)\Vert _{L_2(\mathbb{R}^d)\rightarrow L_2(\mathbb{R}^d)}\leqslant C_7 c(\phi)^2\varepsilon .
\end{equation}
The constant $C_7$ depends only on the initial data \eqref{problem data}.
\end{theorem}

\subsection{Discussion}
Homogenization of the resolvent of the operator $\mathcal{A}_\varepsilon$ in dependence of the spectral parameter
was studied in \cite{Su15}. The two-parametric error estimates that we obtain (Theorems \ref{Theorem main result L2},
\ref{Theorem main result H1}, and \ref{Theorem main result fluxes}) have the same behavior as
the estimates from  \cite[Theorems 2.2, 2.4, and 2.6]{Su15}.

However, there is a difference between our results and the results of \cite{Su15}.
Approximations of the resolvent $(\mathcal{A}_\varepsilon - \zeta I)^{-1}$
with error estimates of the form \eqref{Th L2 principal part}, \eqref{Th 3.2 L_2->L_2}, \eqref{Th 3.2 D L_2->L_2}, and \eqref{Th 4.4}
were proved for any $\zeta \in {\mathbb C} \setminus {\mathbb R}_+$.
For the operator $B_\varepsilon$ the situation is different: in Theorems
\ref{Theorem main result L2}, \ref{Theorem main result H1}, and \ref{Theorem main result fluxes} it is assumed in addition that
$|\zeta|\ge 1$. This restriction is related to the presence of the lower order terms in $B_\varepsilon$.
Below in Section \ref{Section another approximation} we extend the domain of admissible values of $\zeta$,
but the order of estimates (with respect to $\zeta$) will be less precise.

\section{Proof of Theorem \ref{Theorem main result L2}}
\label{Section proof of main result L2}
\setcounter{section}{5}
\setcounter{equation}{0}
\setcounter{theorem}{0}

\subsection{The operator $B(\varepsilon ;\vartheta)$}
\label{51}
The proof of Theorems \ref{Theorem main result L2} and \ref{Theorem main result H1}
is based on application of Theorems \ref{Theorem Su L2} and \ref{Theorem S_eps z fixed}
to the auxiliary family of operators depending on the additional parameter $0 < \vartheta \leqslant 1$.
In Subsections \ref{51}--\ref{53}, we introduce the necessary objects.

Suppose that the assumptions of Subsections \ref{Subsection Operator A}--\ref{Subsection form q} are satisfied.
Let $Q_0$ be the matrix-valued function introduced in Subsection \ref{Subsection generalized resolvent}, and let
$c_5=(c_0+c_4)\Vert  Q_0^{-1}\Vert _{L_\infty}$.
In $L_2(\mathbb{R}^d;\mathbb{C}^n)$, consider the quadratic form
\begin{equation}
\label{forma b(e,theta)}
\begin{split}
b(\varepsilon ;\vartheta)[\mathbf{u},\mathbf{u}]&=\mathfrak{a}[\mathbf{u},\mathbf{u}]+2\vartheta \varepsilon \textnormal{Re}\,(\mathcal{Y}\mathbf{u},\mathcal{Y}_2\mathbf{u})_{L_2(\mathbb{R}^d)}\\
&+\vartheta ^2\varepsilon ^2 q[\mathbf{u},\mathbf{u}]+\vartheta ^2\varepsilon ^2 c_5(Q_0\mathbf{u},\mathbf{u})_{L_2(\mathbb{R}^d)},\quad
 \mathbf{u}\in H^1(\mathbb{R}^d;\mathbb{C}^n).
\end{split}
\end{equation}
Here $\varepsilon >0$ and $0 < \vartheta \leqslant 1$. \textit{We assume that $0 < \varepsilon \vartheta \leqslant 1$.}

Note that, by \eqref{b(eps)} and \eqref{forma b(e,theta)}, we have
$b(\varepsilon ;\vartheta)[\mathbf{u},\mathbf{u}]=\mathfrak{b}(\varepsilon \vartheta)[\mathbf{u},\mathbf{u}]+c_5\varepsilon ^2\vartheta ^2(Q_0\mathbf{u},\mathbf{u})_{L_2}$ for any $\mathbf{u}\in H^1(\mathbb{R}^d;\mathbb{C}^n)$.
Combining this with \eqref{1.17a}, \eqref{1.18a} and the formula for $c_5$,
we see that the form $b(\varepsilon ;\vartheta)$ is closed and nonnegative, and
\begin{equation}
\label{b(e,theta)>=}
\begin{split}
&c_* \Vert \mathbf{D} \mathbf{u}\Vert ^2_{L_2(\mathbb{R}^d)} \leqslant
b(\varepsilon ;\vartheta)[\mathbf{u},\mathbf{u}]\leqslant C_* \Vert \mathbf{D} \mathbf{u}\Vert ^2_{L_2(\mathbb{R}^d)}
\\
&+(C(1)+c_3+c_5\Vert Q_0\Vert _{L_\infty})\varepsilon ^2\Vert \mathbf{u}\Vert ^2_{L_2(\mathbb{R}^d)},
\quad \mathbf{u}\in H^1(\mathbb{R}^d;\mathbb{C}^n),\quad 0< \varepsilon \leqslant \vartheta^{-1}.
\end{split}
\end{equation}
By $B(\varepsilon ;\vartheta)$ we denote the selfadjoint operator in $L_2(\mathbb{R}^d;\mathbb{C}^n)$ generated by this form.
Formally, we have
\begin{equation}
\label{3.7a B(eps;theta)=}
\begin{split}
B(\varepsilon ;\vartheta )&=b(\mathbf{D})^*g(\mathbf{x})b(\mathbf{D})+\vartheta \varepsilon \sum \limits _{j=1}^d (a_j(\mathbf{x})D_j+D_j a_j(\mathbf{x})^*)
\\
&+\vartheta ^2\varepsilon ^2 Q(\mathbf{x})+\vartheta ^2\varepsilon ^2 c_5Q_0(\mathbf{x}).
\end{split}
\end{equation}

\subsection{The operator $B_\varepsilon (\vartheta)$}
\label{Subsection operator B_eps (vartheta)}

In $L_2(\mathbb{R}^d;\mathbb{C}^n)$, we consider the quadratic form
\begin{equation}
\label{b_e(theta)}
\begin{split}
b_\varepsilon (\vartheta)[\mathbf{u},\mathbf{u}]&=\mathfrak{a}_\varepsilon [\mathbf{u},\mathbf{u}]+2\vartheta \textnormal{Re}\,(\mathcal{Y}\mathbf{u},\mathcal{Y}_{2,\varepsilon}\mathbf{u})_{L_2(\mathbb{R}^d)}+\vartheta ^2 q_\varepsilon [\mathbf{u},\mathbf{u}]\\&+\vartheta ^2c_5(Q_0^\varepsilon \mathbf{u},\mathbf{u})_{L_2(\mathbb{R}^d)},\quad
 \mathbf{u}\in H^1(\mathbb{R}^d;\mathbb{C}^n),
\end{split}
\end{equation}
where $0< \vartheta \leqslant 1$ and $0< \varepsilon \leqslant \vartheta^{-1}$.
We have
\begin{equation}
\label{b_e(theta)=}
b_\varepsilon (\vartheta)[\mathbf{u},\mathbf{u}]=\varepsilon ^{-2}b(\varepsilon ;\vartheta)[T_\varepsilon\mathbf{u},T_\varepsilon \mathbf{u}],\quad \mathbf{u}\in H^1(\mathbb{R}^d;\mathbb{C}^n),
\end{equation}
where $T_\varepsilon$ is the scaling transformation given by \eqref{T_eps}.
Relations \eqref{b(e,theta)>=} and \eqref{b_e(theta)=} imply that the form \eqref{b_e(theta)}
is closed and nonnegative, and satisfies the estimates 
\begin{equation}
\label{3/10a b_eps(theta)>=}
\begin{split}
c_* \Vert \mathbf{D} \mathbf{u}\Vert ^2_{L_2(\mathbb{R}^d)} \leqslant
b_\varepsilon (\vartheta )[\mathbf{u},\mathbf{u}]\leqslant
c_6 \Vert \mathbf{u}\Vert ^2_{H^1(\mathbb{R}^d)},
\\
\quad
\mathbf{u}\in H^1(\mathbb{R}^d;\mathbb{C}^n),\quad 0< \varepsilon \leqslant \vartheta^{-1},
\end{split}
\end{equation}
where $c_6=\max \lbrace C_*; C(1)+c_3 +c_5\Vert Q_0\Vert _{L_\infty}\rbrace $.
Let $B_\varepsilon (\vartheta)$ be the selfadjoint operator in $L_2(\mathbb{R}^d;\mathbb{C}^n)$
corresponding to the form \eqref{b_e(theta)=}.

We will apply Theorem \ref{Theorem Su L2} to approximate
the generalized resolvent $(B_\varepsilon (\vartheta) + \lambda _0 Q_0^\varepsilon)^{-1}$, where
$\lambda_0$ is given by \eqref{lambda0} (see Remark \ref{Remark C1 C2 in Theorems by Su}).

\subsection{The operators $B^0(\vartheta)$ and $B^0(\varepsilon ;\vartheta)$}
\label{53}
Obviously (see \hbox{\eqref{Lambda problem}--\eqref{tilde g}}, \eqref{tildeLambda_problem}--\eqref{B^0}, and \eqref{B0}),
the effective operator for $B_\varepsilon (\vartheta)$ takes the form
\begin{equation*}
\begin{split}
B^0(\vartheta)&=\mathcal{A}^0+ \vartheta \biggl(-b(\mathbf{D})^*V-V^*b(\mathbf{D})+\sum _{j=1}^d(\overline{a_j+a_j^*})D_j\biggr)\\&+\vartheta ^2(\overline{Q}+c_5\overline{Q_0}-W).
\end{split}
\end{equation*}
According to \eqref{b^0[u,u]>=}, the quadratic form $b^0(\vartheta )$ of the operator $B^0(\vartheta )$ satisfies
\begin{equation}
\label{b^0(theta)>=}
b^0(\vartheta )[\mathbf{u},\mathbf{u}]\geqslant c_* \Vert \mathbf{D}\mathbf{u}\Vert ^2_{L_2(\mathbb{R}^d)},\quad \mathbf{u}\in H^1(\mathbb{R}^d;\mathbb{C}^n).
\end{equation}
This is equivalent to the following estimate for the symbol
$L_0(\boldsymbol{\xi} ;\vartheta)$ of the operator $B^0(\vartheta)$:
\begin{equation}
\label{L0>=}
L_0(\boldsymbol{\xi};\vartheta )\geqslant {c}_* \vert \boldsymbol{\xi}\vert ^2 \mathbf{1}_n,
\quad \boldsymbol{\xi}\in\mathbb{R}^d,\quad 0< \vartheta \leqslant 1.
\end{equation}
Hence, by \eqref{lambda0}, the symbol
\begin{equation*}
\begin{split}
L(\boldsymbol{\xi} ;\vartheta)&=b(\boldsymbol{\xi})^*g^0b(\boldsymbol{\xi})-\vartheta b(\boldsymbol{\xi})^*V-\vartheta V^*b(\boldsymbol{\xi})\\
&+\vartheta \sum \limits _{j=1}^d (\overline{a_j+a_j^*})\xi _j +\vartheta ^2(\overline{Q}+c_5\overline{Q_0})-\vartheta ^2W+\lambda_0 \overline{Q_0}
\end{split}
\end{equation*}
of the operator $B^0(\vartheta)+\lambda_0 \overline{Q_0}$ satisfies
\begin{equation}
\label{L(xi,theta)>=}
L(\boldsymbol{\xi};\vartheta )\geqslant \check{c}_*(\vert \boldsymbol{\xi}\vert ^2+1)\mathbf{1}_n,\quad \boldsymbol{\xi}\in\mathbb{R}^d,\quad 0< \vartheta \leqslant 1,\quad\check{c}_*=\min \lbrace c_*;2c_4\rbrace.
\end{equation}

Note that
\begin{equation*}
B^0(\vartheta)=\varepsilon ^{-2}T_\varepsilon ^*B^0(\varepsilon ;\vartheta)T_\varepsilon ,
\end{equation*}
where $B^0(\varepsilon ;\vartheta)$ is the selfadjoint operator in $L_2(\mathbb{R}^d;\mathbb{C}^n)$ given by
\begin{equation}
\label{B0(eps;theta)}
\begin{split}
B^0(\varepsilon ;\vartheta)&=\mathcal{A}^0+\varepsilon \vartheta
\biggl(-b(\mathbf{D})^*V-V^*b(\mathbf{D})+\sum _{j=1}^d(\overline{a_j+a_j^*})D_j\biggr)\\
&+\varepsilon ^2\vartheta ^2(\overline{Q}-W+c_5\overline{Q_0}).
\end{split}
\end{equation}

\subsection{The operators $\widetilde{B}_\varepsilon (\vartheta)$ and $\widetilde{B}^0(\vartheta)$}
\label{Subsection tile-operators(theta)}
We factorize the matrices $Q_0(\mathbf{x})$ and $\overline{Q_0}$ as follows:
\begin{equation}
\label{6.10aa}
Q_0(\mathbf{x})=(f(\mathbf{x})^*)^{-1}f(\mathbf{x})^{-1},
\end{equation}
$\overline{Q_0}=f_0^{-2}$. Note that
\begin{equation}
\label{f0_estimates}
\vert f_0\vert \leqslant \Vert f\Vert _{L_\infty}=\Vert Q_0 ^{-1}\Vert _{L_\infty}^{1/2},
 \quad
 \vert f_0^{-1}\vert \leqslant\Vert f^{-1}\Vert _{L_\infty}=\Vert Q_0\Vert ^{1/2}_{L_\infty}.
\end{equation}
Let $\widetilde{B}_\varepsilon (\vartheta)$ be the selfadjoint operator in $L_2(\mathbb{R}^d;\mathbb{C}^n)$ generated by
 the quadratic form
\begin{equation}
\label{6.10a}
\widetilde{b}_\varepsilon (\vartheta)[\mathbf{u},\mathbf{u}]:=b_\varepsilon (\vartheta)[f^\varepsilon \mathbf{u},f^\varepsilon \mathbf{u}]
\end{equation}
defined on the domain
$$\textnormal{Dom}\,\widetilde{b}_\varepsilon (\vartheta)=\lbrace \mathbf{u}\in L_2(\mathbb{R}^d;\mathbb{C}^n) : f^\varepsilon \mathbf{u}\in H^1(\mathbb{R}^d;\mathbb{C}^n)\rbrace.$$
Here $0< \vartheta \leqslant 1$ and $0< \varepsilon \leqslant \vartheta^{-1}$.
Since the operator $B_\varepsilon (\vartheta)$ is nonnegative, the operator $\widetilde{B}_\varepsilon (\vartheta)$ is also nonnegative.

Let $\widetilde{B}^0(\vartheta)=f_0B^0(\vartheta)f_0$.
 Note that
\begin{align}
\label{Res B eps and tilde B eps}
&(B_\varepsilon (\vartheta)-\zeta Q_0^\varepsilon)^{-1}=f^\varepsilon (\widetilde{B}_\varepsilon (\vartheta )-\zeta I )^{-1}(f^\varepsilon )^*,\quad \zeta \in \mathbb{C}\setminus \mathbb{R}_+,\\
\label{3/14a}
&(B^0(\vartheta )-\zeta \overline{Q_0})^{-1}=f_0(\widetilde{B}^0(\vartheta)-\zeta I)^{-1}f_0 ,\quad \zeta \in \mathbb{C}\setminus \mathbb{R}_+.
\end{align}

\subsection{Proof of Theorem \ref{Theorem main result L2}}
Applying Theorem~\ref{Theorem Su L2} to the operator~$B_\varepsilon (\vartheta )$ and taking Remark~\ref{Remark C1 C2 in Theorems by Su}
into account, we obtain
\begin{equation}
\label{Th9.2 for B(theta) 0<e<1}
\begin{split}
\Vert (B_\varepsilon (\vartheta )+\lambda_0 Q_0^\varepsilon )^{-1}-
(B^0(\vartheta)+\lambda_0 \overline{Q_0})^{-1}\Vert _{L_2(\mathbb{R}^d)\rightarrow L_2(\mathbb{R}^d)}\leqslant C_1\varepsilon ,
\\
0 < \vartheta \leqslant 1,\;0<\varepsilon \leqslant 1.
\end{split}
\end{equation}
Now we carry this estimate over to all $0 < \varepsilon \leqslant \vartheta ^{-1}$.
 By \eqref{Res B eps and tilde B eps},
\begin{equation}
\label{Res Be(theta)<}
\begin{split}
\Vert (B_\varepsilon (\vartheta)+\lambda_0 Q_0^\varepsilon)^{-1}\Vert _{L_2(\mathbb{R}^d)\rightarrow L_2(\mathbb{R}^d)}
&\leqslant \Vert f\Vert ^2 _{L_\infty}\Vert (\widetilde{B}_\varepsilon (\vartheta)+\lambda _0 I)^{-1}\Vert _{L_2(\mathbb{R}^d)\rightarrow L_2(\mathbb{R}^d)}\\
&\leqslant  \lambda_0^{-1}\Vert f\Vert ^2_{L_\infty}.
\end{split}
\end{equation}
Similarly, \eqref{3/14a} and \eqref{f0_estimates} imply that
\begin{equation}
\label{Res B0(theta)<}
\Vert (B^0 (\vartheta)+\lambda_0 \overline{Q_0})^{-1}\Vert _{L_2(\mathbb{R}^d)\rightarrow L_2(\mathbb{R}^d)}\leqslant
\lambda_0^{-1}\Vert f\Vert ^2_{L_\infty}.
\end{equation}
From \eqref{Res Be(theta)<} and \eqref{Res B0(theta)<} it follows that the left-hand side of \eqref{Th9.2 for B(theta) 0<e<1}
does not exceed $2 \lambda_0^{-1}\Vert f\Vert ^2_{L_\infty}\varepsilon$ for $1<\varepsilon \leqslant \vartheta^{-1}$.
Together with \eqref{Th9.2 for B(theta) 0<e<1} this implies
\begin{equation}
\label{Th9.2 for B(theta) lambda 0}
\begin{split}
\Vert (B_\varepsilon (\vartheta )+\lambda_0 Q_0^\varepsilon )^{-1}-(B^0(\vartheta)+\lambda_0 \overline{Q_0})^{-1}\Vert _{L_2(\mathbb{R}^d)\rightarrow L_2(\mathbb{R}^d)}\leqslant \widehat{C}_1\varepsilon ,
\\
0< \vartheta \leqslant 1,\; 0 < \varepsilon \leqslant \vartheta^{-1},
\end{split}
\end{equation}
with $\widehat{C}_1 =\max\lbrace C_1;2\lambda_0^{-1}\Vert f\Vert ^2_{L_\infty}\rbrace $.

Now we obtain an analog of estimate \eqref{Th9.2 for B(theta) lambda 0} for
$(B_\varepsilon (\vartheta)-\widehat{\zeta} Q_0^\varepsilon )^{-1}$, where $\widehat{\zeta}=e^{i\phi}$ with $\phi \in (0,2\pi)$.
We rely on the identity
\begin{equation}
\label{tozd generalized resolvent}
\begin{split}
&(B_\varepsilon (\vartheta )-\widehat{\zeta} Q_0^\varepsilon )^{-1}-(B^0(\vartheta)-\widehat{\zeta} \overline{Q_0})^{-1}\\
&=(B_\varepsilon (\vartheta)-\widehat{\zeta} Q_0^\varepsilon )^{-1}(B_\varepsilon (\vartheta)+\lambda _0 Q_0^\varepsilon )\\
&\times\left( (B_\varepsilon (\vartheta)+\lambda _0 Q_0^\varepsilon )^{-1}-(B^0(\vartheta)+\lambda _0 \overline{Q_0})^{-1}\right)\\
&\times
(B^0(\vartheta)+\lambda _0 \overline{Q_0})(B^0(\vartheta)-\widehat{\zeta} \overline{Q_0})^{-1}\\
&+(\lambda _0 +\widehat{\zeta} )(B_\varepsilon (\vartheta)-\widehat{\zeta} Q_0^\varepsilon )^{-1}(Q^\varepsilon _0-\overline{Q_0})(B^0(\vartheta)-\widehat{\zeta} \overline{Q_0})^{-1}.
\end{split}
\end{equation}

By \eqref{Res B eps and tilde B eps}, we have
\begin{equation}
\label{6.17a}
\begin{split}
\Vert &(B_\varepsilon (\vartheta)-\widehat{\zeta} Q_0^\varepsilon)^{-1}(B_\varepsilon (\vartheta)+\lambda_0 Q_0^\varepsilon )\Vert _{L_2(\mathbb{R}^d)\rightarrow L_2(\mathbb{R}^d)}\\
&=\Vert f^\varepsilon (\widetilde{B}_\varepsilon (\vartheta)-\widehat{\zeta} I)^{-1}(\widetilde{B}_\varepsilon (\vartheta) +\lambda_0 I)(f^\varepsilon )^{-1}\Vert _{L_2(\mathbb{R}^d)\rightarrow L_2(\mathbb{R}^d)}\\
&\leqslant
\Vert f\Vert _{L_\infty}\Vert f^{-1}\Vert _{L_\infty} \Vert (\widetilde{B}_\varepsilon (\vartheta)-\widehat{\zeta} I)^{-1}(\widetilde{B}_\varepsilon (\vartheta) +\lambda_0 I) \Vert _{L_2(\mathbb{R}^d)\rightarrow L_2(\mathbb{R}^d)}.
\end{split}
\end{equation}
Since $\widetilde{B}_\varepsilon(\vartheta) \ge 0$, then
\begin{equation}
\label{6.17b}
\begin{split}
\Vert& (\widetilde{B}_\varepsilon (\vartheta)-\widehat{\zeta} I)^{-1}(\widetilde{B}_\varepsilon (\vartheta) +\lambda_0 I) \Vert _{L_2(\mathbb{R}^d)\rightarrow L_2(\mathbb{R}^d)}
\leqslant\sup _{\nu \geqslant 0}\frac{\nu +\lambda_0}{\vert \nu -\widehat{\zeta}\vert}\\
&\leqslant \sup _{\nu \geqslant 0}\frac{\nu +\lambda_0}{\nu+1} \sup _{\nu \geqslant 0}\frac{\nu +1}{\vert \nu -\widehat{\zeta}\vert}
\leqslant 2(1+\lambda_0)c(\phi).
\end{split}
\end{equation}
Here $c(\phi)$ is given by \eqref{c(phi)}. Relations \eqref{6.17a} and \eqref{6.17b} imply that
\begin{equation}
\label{tozd levaja obveska}
\begin{split}
\Vert &(B_\varepsilon (\vartheta)-\widehat{\zeta} Q_0^\varepsilon)^{-1}(B_\varepsilon (\vartheta)+\lambda_0 Q_0^\varepsilon )\Vert _{L_2(\mathbb{R}^d)\rightarrow L_2(\mathbb{R}^d)}\\
&\leqslant 2(1+\lambda_0)\Vert f\Vert _{L_\infty}\Vert f^{-1}\Vert _{L_\infty}c(\phi).
\end{split}
\end{equation}
Similarly,
\begin{equation}
\label{tozd prav obveska}
\begin{split}
\Vert &(B^0(\vartheta)+\lambda_0 \overline{Q_0})(B^0(\vartheta)-\widehat{\zeta} \overline{Q_0})^{-1}\Vert _{L_2(\mathbb{R}^d)\rightarrow L_2(\mathbb{R}^d)}\\
&\leqslant
2(1+\lambda_0)\Vert f\Vert _{L_\infty}\Vert f^{-1}\Vert _{L_\infty}c(\phi).
\end{split}
\end{equation}

Let us estimate the norm of the second term in the right-hand side of~\eqref{tozd generalized resolvent}:
\begin{equation}
\label{norm second summand in tozd}
\begin{split}
\Vert &(\widehat{\zeta} +\lambda_0)(B_\varepsilon (\vartheta)-\widehat{\zeta} Q_0^\varepsilon )^{-1}(Q^\varepsilon _0-\overline{Q_0})(B^0(\vartheta)-\widehat{\zeta} \overline{Q_0})^{-1}\Vert _{L_2(\mathbb{R}^d)\rightarrow L_2(\mathbb{R}^d)}\\
&\leqslant (1+\lambda_0)\Vert (B_\varepsilon (\vartheta)-\widehat{\zeta} Q_0^\varepsilon )^{-1}\Vert _{H^{-1}(\mathbb{R}^d)\rightarrow L_2(\mathbb{R}^d)}\\
&\times\Vert [Q_0^\varepsilon -\overline{Q_0}]\Vert _{H^1(\mathbb{R}^d)\rightarrow H^{-1}(\mathbb{R}^d)}\Vert (B^0(\vartheta)-\widehat{\zeta} \overline{Q_0})^{-1}\Vert _{L_2(\mathbb{R}^d)\rightarrow H^{1}(\mathbb{R}^d)}.
\end{split}
\end{equation}

By the duality argument, we have
\begin{equation}
\label{dvoystvennost'}
\begin{split}
\Vert (B_\varepsilon (\vartheta)-\widehat{\zeta} Q_0^\varepsilon )^{-1}\Vert _{H^{-1}(\mathbb{R}^d)\rightarrow L_2(\mathbb{R}^d)}
=\Vert( B_\varepsilon (\vartheta)-\widehat{\zeta}^*Q_0^\varepsilon )^{-1}\Vert _{L_2(\mathbb{R}^d)\rightarrow H^1(\mathbb{R}^d)}.
\end{split}
\end{equation}
By \eqref{Res B eps and tilde B eps},
\begin{equation}
\label{Gen res B_eps L_2->L_2}
\begin{split}
\Vert ( B_\varepsilon (\vartheta)-\widehat{\zeta}^*Q_0^\varepsilon )^{-1}\Vert _{L_2(\mathbb{R}^d)\rightarrow L_2(\mathbb{R}^d)}
&\leqslant \Vert f\Vert ^2_{L_\infty}\Vert (\widetilde{B}_\varepsilon (\vartheta)-\widehat{\zeta}^* I)^{-1}\Vert _{L_2(\mathbb{R}^d)\rightarrow L_2(\mathbb{R}^d)}\\
&\leqslant \Vert f\Vert ^2_{L_\infty}c(\phi).
\end{split}
\end{equation}
Next, the lower estimate \eqref{3/10a b_eps(theta)>=}, \eqref{6.10a}, and \eqref{Res B eps and tilde B eps} imply that
\begin{equation}
\label{6.22a}
\begin{split}
\Vert & \mathbf{D}( B_\varepsilon (\vartheta)-\widehat{\zeta}^*Q_0^\varepsilon )^{-1}\Vert _{L_2(\mathbb{R}^d)\rightarrow L_2(\mathbb{R}^d)}\\
&\leqslant c_*^{-1/2}\Vert B_\varepsilon (\vartheta )^{1/2}( B_\varepsilon (\vartheta)-\widehat{\zeta}^*Q_0^\varepsilon )^{-1}\Vert _{L_2(\mathbb{R}^d)\rightarrow L_2(\mathbb{R}^d)}\\
&=
c_*^{-1/2} \Vert \widetilde{B}_\varepsilon (\vartheta )^{1/2}(\widetilde{B}_\varepsilon (\vartheta )-\widehat{\zeta}^*I)^{-1}(f^\varepsilon )^*\Vert _{L_2(\mathbb{R}^d)\rightarrow L_2(\mathbb{R}^d)}\\
&\leqslant c_*^{-1/2} \Vert f\Vert _{L_\infty}\sup _{\nu \geqslant 0}\frac{\nu ^{1/2}}{\vert \nu-\widehat{\zeta}^*\vert }
\leqslant c_*^{-1/2}\Vert f\Vert _{L_\infty}c(\phi).
\end{split}
\end{equation}
Combining this with \eqref{dvoystvennost'} and \eqref{Gen res B_eps L_2->L_2}, we see that \begin{equation}
\label{Beps+Q_0 H-1 ->L2}
\Vert (B_\varepsilon (\vartheta)-\widehat{\zeta} Q_0^\varepsilon )^{-1}\Vert _{H^{-1}(\mathbb{R}^d)\rightarrow L_2(\mathbb{R}^d)}\leqslant \mathfrak{C}_1c(\phi),
\end{equation}
where $\mathfrak{C}_1:=\Vert f\Vert ^2_{L_\infty}+c_*^{-1/2} \Vert f\Vert _{L_\infty}$.

By analogy with \eqref{Gen res B_eps L_2->L_2} and \eqref{6.22a},
we estimate the $(L_2\rightarrow H^1)$-norm of the operator $(B^0(\vartheta)-\widehat{\zeta}\overline{Q_0})^{-1}$,
using \eqref{b^0(theta)>=}, \eqref{f0_estimates}, and \eqref{3/14a}. This yields 
\begin{equation}
\label{B0+Q0 L2->H1}
\begin{split}
\Vert (B^0(\vartheta )-\widehat{\zeta}\overline{Q_0})^{-1}\Vert _{L_2(\mathbb{R}^d)\rightarrow H^1(\mathbb{R}^d)}
\leqslant \mathfrak{C}_1 c(\phi) .
\end{split}
\end{equation}

Relations \eqref{lemma Q 0 eps -Q}, \eqref{norm second summand in tozd}, \eqref{Beps+Q_0 H-1 ->L2}, and \eqref{B0+Q0 L2->H1} imply that
\begin{equation*}
\Vert (\widehat{\zeta} +\lambda_0)(B_\varepsilon (\vartheta)-\widehat{\zeta} Q_0^\varepsilon )^{-1}(Q^\varepsilon _0-\overline{Q_0})(B^0(\vartheta)-\widehat{\zeta} \overline{Q_0})^{-1}\Vert _{L_2(\mathbb{R}^d)\rightarrow L_2(\mathbb{R}^d)}\leqslant \mathfrak{C}_2\varepsilon c(\phi)^2 ,
\end{equation*}
where $\mathfrak{C}_2=(1+\lambda_0)C_{Q_0}\mathfrak{C}_1^2$.
Combining this with \eqref{Th9.2 for B(theta) lambda 0}, \eqref{tozd generalized resolvent}, \eqref{tozd levaja obveska},
 and \eqref{tozd prav obveska}, we obtain
\begin{equation}
\label{Beps +Q0 eps - ... <= C eps}
\begin{split}
\Vert (B_\varepsilon (\vartheta )-\widehat{\zeta}Q_0^\varepsilon )^{-1}-(B^0(\vartheta )-\widehat{\zeta}\overline{Q_0})^{-1}\Vert _{L_2(\mathbb{R}^d)\rightarrow L_2(\mathbb{R}^d)}\leqslant C_4\varepsilon c(\phi)^2 ,\\
0 <  \vartheta \leqslant 1,\quad 0< \varepsilon \leqslant \vartheta^{-1},\quad \widehat{\zeta}=e^{i\phi},\quad \phi \in (0,2\pi),
\end{split}
\end{equation}
with $C_4=4(1+\lambda_0)^2\Vert f\Vert ^2_{L_\infty}\Vert f^{-1}\Vert ^2_{L_\infty}\widehat{C}_1+\mathfrak{C}_2$.

By the scaling transformation, from \eqref{Beps +Q0 eps - ... <= C eps} we deduce
\begin{equation}
\label{posle T_eps L2Rd}
\begin{split}
\Vert (B(\varepsilon ;\vartheta)-\widehat{\zeta} \varepsilon ^2 Q_0)^{-1}-(B^0(\varepsilon ;\vartheta)-\widehat{\zeta} \varepsilon ^2\overline{Q_0})^{-1}\Vert _{L_2(\mathbb{R}^d)\rightarrow L_2(\mathbb{R}^d)}
\leqslant  C_4 c(\phi )^2\varepsilon ^{-1}.
\end{split}
\end{equation}
In this estimate, we substitute $\widetilde{\varepsilon }\vert \zeta \vert ^{1/2}$ in place of $\varepsilon $, assuming that
$0<\widetilde{\varepsilon} \leqslant 1$, $\zeta =\vert \zeta \vert e^{i\phi }\in \mathbb{C}\setminus \mathbb{R}_+$, $\phi \in (0,2\pi)$,
$\vert \zeta \vert \geqslant 1$, and put $\vartheta = \vert \zeta \vert ^{-1/2}$.
Then the conditions $0<\vartheta \leqslant 1$ and $0<\varepsilon \vartheta \leqslant 1$ are satisfied.
We have (see \eqref{3.7a B(eps;theta)=} and \eqref{B0(eps;theta)})
\begin{align*}
&B(\widetilde{\varepsilon }\vert \zeta \vert ^{1/2};\vert \zeta \vert ^{-1/2})=B(\widetilde{\varepsilon } ;1),\\
 &B^0(\widetilde{\varepsilon }\vert \zeta \vert ^{1/2};\vert \zeta \vert ^{-1/2} )=B^0(\widetilde{\varepsilon };1).
\end{align*}
Therefore, \eqref{posle T_eps L2Rd} implies that
\begin{equation*}
\begin{split}
\Vert &(B(\widetilde{\varepsilon };1)-\zeta \widetilde{\varepsilon }^{\,2} Q_0)^{-1}-(B^0(\widetilde{\varepsilon } ;1)-
\zeta \widetilde{\varepsilon}^{\,2} \overline{Q_0})^{-1}\Vert _{L_2(\mathbb{R}^d)\rightarrow L_2(\mathbb{R}^d)}\\
&\leqslant  C_4 c(\phi)^2\widetilde{\varepsilon} ^{\,-1}\vert \zeta \vert ^{-1/2},\quad 0<\widetilde{\varepsilon } \leqslant 1,\quad \zeta =\vert \zeta \vert e^{i\phi},\quad \vert \zeta \vert \geqslant 1,\quad 0<\phi <2\pi.
\end{split}
\end{equation*}
Renaming $\varepsilon :=\widetilde{\varepsilon}$
and applying the inverse scaling transformation, we arrive at \eqref{Th L2 principal part}.
This completes the proof of Theorem~\ref{Theorem main result L2}.

 \qed

\section{Proof of Theorems \ref{Theorem main result H1} and \ref{Theorem main result fluxes}}
\label{Section Proof of main result in H1}

\subsection{The operator $\mathcal{K}(\varepsilon ;\vartheta)$}
As above, assume that $0< \vartheta \leqslant 1$ and \hbox{$0< \varepsilon \leqslant \vartheta^{-1}$}.
Consider the generalized resolvent $(B_\varepsilon (\vartheta)+\lambda _0 Q_0^\varepsilon )^{-1}$.
Taking into account the form of the problems for $\Lambda$ and $\widetilde{\Lambda}$
(see \eqref{Lambda problem} and \eqref{tildeLambda_problem}),
we see that the analog of the corrector~\eqref{tilde K(eps)} for the generalized resolvent
 $(B_\varepsilon (\vartheta)+\lambda _0 Q_0^\varepsilon )^{-1}$ takes the form
\begin{equation}
\label{mathscr K(eps;theta)}
\begin{split}
\mathcal{K}(\varepsilon ;\vartheta)
=\left([\Lambda ^\varepsilon ]b(\mathbf{D})+\vartheta [\widetilde{\Lambda}^\varepsilon ]\right)S_\varepsilon (B^0(\vartheta)+\lambda_0\overline{Q_0})^{-1}
.
\end{split}
\end{equation}
The operator $\mathcal{K}(\varepsilon ;\vartheta)$ is a continuous mapping of $L_2(\mathbb{R}^d;\mathbb{C}^n)$ to
$H^1(\mathbb{R}^d;\mathbb{C}^n)$. This follows from the next lemma.

\begin{lemma}
\label{Lemma K(eps,theta)}
Let $\mathcal{K}(\varepsilon ;\vartheta)$ be the operator given by \eqref{mathscr K(eps;theta)}.
Then for $0 < \vartheta \leqslant 1$ and $\varepsilon >0$ the operator $\mathcal{K}(\varepsilon ;\vartheta)$
is a continuous mapping of $L_2(\mathbb{R}^d;\mathbb{C}^n)$ to $H^1(\mathbb{R}^d;\mathbb{C}^n)$, and
\begin{align}
\label{K<= L_2 Lm}
&\Vert \mathcal{K}(\varepsilon ;\vartheta)\Vert _{L_2(\mathbb{R}^d)\rightarrow L_2(\mathbb{R}^d)}\leqslant C_K^{(1)},\\
\label{epsDK<= L_2 Lm}
&\Vert \varepsilon \mathbf{D}\mathcal{K}(\varepsilon ;\vartheta)\Vert _{L_2(\mathbb{R}^d)\rightarrow L_2(\mathbb{R}^d)}\leqslant C_K^{(2)}\varepsilon +C_K^{(3)}.
\end{align}
The constants $C_K^{(1)}$, $C_K^{(2)}$, and $C_K^{(3)}$ depend only on the initial data \eqref{problem data}.
\end{lemma}

\begin{proof}
First, we estimate the $(L_2\rightarrow L_2)$-norm of the corrector:
\begin{equation}
\label{K(e,lambda0)<= in L_2}
\begin{split}
\Vert &\mathcal{K}(\varepsilon ;\vartheta)\Vert _{L_2(\mathbb{R}^d)\rightarrow L_2(\mathbb{R}^d)}\\
&\leqslant
\Vert [\Lambda ^\varepsilon ]S _\varepsilon \Vert _{L_2(\mathbb{R}^d)\rightarrow L_2(\mathbb{R}^d)}
\Vert b(\mathbf{D})(B^0(\vartheta) +\lambda_0 \overline{Q_0})^{-1}\Vert _{L_2(\mathbb{R}^d)\rightarrow L_2(\mathbb{R}^d)}\\
&+\Vert [\widetilde{\Lambda}^\varepsilon ]S _\varepsilon \Vert _{L_2(\mathbb{R}^d)\rightarrow L_2(\mathbb{R}^d)}\Vert (B^0(\vartheta)+\lambda_0 \overline{Q_0})^{-1}\Vert _{L_2(\mathbb{R}^d)\rightarrow L_2(\mathbb{R}^d)}.
\end{split}
\end{equation}
Proposition \ref{Proposition S_eps L2->L2} and \eqref{Lambda <=} imply that
\begin{align}
\label{Lambda S_eps <=}
\Vert [\Lambda ^\varepsilon ]S _\varepsilon \Vert _{L_2(\mathbb{R}^d)\rightarrow L_2(\mathbb{R}^d)}
\leqslant M_1.
\end{align}
According to \eqref{<b^*b<},
\begin{equation}
\label{7.6.star}
\begin{split}
\Vert & b(\mathbf{D})(B^0(\vartheta) +\lambda_0 \overline{Q_0})^{-1}\Vert _{L_2(\mathbb{R}^d)\rightarrow L_2(\mathbb{R}^d)}\\
&\leqslant \alpha _1^{1/2}\Vert \mathbf{D}(B^0(\vartheta) +\lambda_0\overline{Q_0})^{-1}\Vert _{L_2(\mathbb{R}^d)\rightarrow L_2(\mathbb{R}^d)}.
\end{split}
\end{equation}
Since the symbol of the operator $B^0(\vartheta)+\lambda_0 \overline{Q_0}$ satisfies \eqref{L(xi,theta)>=}, we have
\begin{equation}
\label{DRes}
\Vert \mathbf{D}(B^0(\vartheta)+\lambda_0 \overline{Q_0})^{-1}\Vert _{L_2(\mathbb{R}^d)\rightarrow L_2(\mathbb{R}^d)}\leqslant\check{c}_*^{-1}\sup \limits _{\boldsymbol{\xi}\in \mathbb{R}^d}\frac{\vert \boldsymbol{\xi}\vert}{\vert \boldsymbol{\xi}\vert ^2+1}
 \leqslant (2\check{c}_*)^{-1}.
\end{equation}
Relations \eqref{tildeLambda S_eps}, \eqref{Res B0(theta)<}, and \eqref{K(e,lambda0)<= in L_2}--\eqref{DRes} imply \eqref{K<= L_2 Lm} 
with the constant $C_K^{(1)}=\alpha _1^{1/2}(2\check{c}_*)^{-1}M_1+
\lambda_0^{-1}\widetilde{M}_1\Vert f\Vert ^2_{L_\infty}$.

Now we check \eqref{epsDK<= L_2 Lm}. Obviously,
\begin{equation*}
\begin{split}
\varepsilon D_j \mathcal{K}(\varepsilon ;\vartheta)&=[(D_j\Lambda )^\varepsilon ]S _\varepsilon b(\mathbf{D})(B^0(\vartheta)
+\lambda_0\overline{Q_0})^{-1}\\
&+\varepsilon [\Lambda ^\varepsilon ]S _\varepsilon b(\mathbf{D})D_j(B^0(\vartheta)+\lambda_0\overline{Q_0})^{-1}\\
&+\vartheta [(D_j\widetilde{\Lambda })^\varepsilon ]S _\varepsilon (B^0(\vartheta)+\lambda_0\overline{Q_0})^{-1}\\
&+\varepsilon\vartheta [\widetilde{\Lambda}^\varepsilon]S _\varepsilon D_j(B^0(\vartheta)+\lambda_0 \overline{Q_0})^{-1}.
\end{split}
\end{equation*}
Hence,
\begin{equation}
\label{eps DK ^2 <=}
\begin{split}
\Vert &\varepsilon \mathbf{D}\mathcal{K}(\varepsilon ;\vartheta)\Vert ^2 _{L_2(\mathbb{R}^d)\rightarrow L_2(\mathbb{R}^d)}\\
&\leqslant
4 \Vert [(\mathbf{D}\Lambda )^\varepsilon ]S _\varepsilon b(\mathbf{D})(B^0(\vartheta)+\lambda_0\overline{Q_0})^{-1}\Vert ^2 _{L_2(\mathbb{R}^d)\rightarrow L_2(\mathbb{R}^d)}\\
&+4\Vert \varepsilon [\Lambda ^\varepsilon ]S _\varepsilon b(\mathbf{D})\mathbf{D}(B^0(\vartheta)+\lambda_0\overline{Q_0})^{-1}\Vert ^2 _{L_2(\mathbb{R}^d)\rightarrow L_2(\mathbb{R}^d)}\\
&+4\Vert [(\mathbf{D}\widetilde{\Lambda})^\varepsilon ]S _\varepsilon (B^0(\vartheta)+\lambda_0\overline{Q_0})^{-1}\Vert ^2 _{L_2(\mathbb{R}^d)\rightarrow L_2(\mathbb{R}^d)}\\
&+4\Vert \varepsilon [\widetilde{\Lambda}^\varepsilon ]S _\varepsilon \mathbf{D}(B^0(\vartheta)+\lambda_0\overline{Q_0})^{-1}\Vert ^2 _{L_2(\mathbb{R}^d)\rightarrow L_2(\mathbb{R}^d)}.
\end{split}
\end{equation}
Applying Proposition \ref{Proposition S_eps L2->L2} and \eqref{DLambda<=}, \eqref{DtildeLambda<=}, we obtain \begin{align}
\label{DLambda S eps<=}
&\Vert [(\mathbf{D}\Lambda )^\varepsilon] S _\varepsilon \Vert _{L_2(\mathbb{R}^d)\rightarrow L_2(\mathbb{R}^d)}
\leqslant M_2,\\
\label{DtildeLambda S eps<=}
&\Vert [(\mathbf{D}\widetilde{\Lambda})^\varepsilon] S _\varepsilon \Vert _{L_2(\mathbb{R}^d)\rightarrow L_2(\mathbb{R}^d)} \leqslant \vert \Omega \vert ^{-1/2}C_a n^{1/2}\alpha _0^{-1}\Vert g^{-1}\Vert _{L_\infty}=:\widetilde{M}_2.
\end{align}
By \eqref{<b^*b<} and \eqref{L0>=},
\begin{equation}
\label{b(D)DRes <=}
\Vert b(\mathbf{D})\mathbf{D}(B^0(\vartheta)+\lambda_0 \overline{Q_0})^{-1}\Vert _{L_2(\mathbb{R}^d)\rightarrow L_2(\mathbb{R}^d)}
\leqslant \alpha _1^{1/2} {c}_*^{-1}.
\end{equation}
Combining \eqref{tildeLambda S_eps}, \eqref{Res B0(theta)<}, and \eqref{Lambda S_eps <=}--\eqref{b(D)DRes <=},
we arrive at \eqref{epsDK<= L_2 Lm} with
\begin{align*}
&C_K^{(2)}= (4\alpha _1 M_1^2 {c}_*^{-2} +\widetilde{M}_1^2 \check{c}_*^{-2})^{1/2},\\
&C_K^{(3)}=(\alpha _1 M_2^2\check{c}_*^{-2}+4\lambda _0^{-2}\Vert f\Vert ^4_{L_\infty}\widetilde{M}_2^2)^{1/2}.
\end{align*}
\end{proof}

\subsection{Proof of Theorem \ref{Theorem main result H1}}
Applying Theorem~\ref{Theorem S_eps z fixed} to $B_\varepsilon (\vartheta)$ and taking Remark \ref{Remark Constants C5}
into account, we obtain
\begin{equation}
\label{Th 3.4 for B eps (theta)}
\begin{split}
\Vert (B_\varepsilon (\vartheta)+\lambda _0 Q_0^\varepsilon )^{-1}-(B^0(\vartheta )+\lambda _0 \overline{Q_0})^{-1}-\varepsilon \mathcal{K}(\varepsilon ;\vartheta)\Vert _{L_2(\mathbb{R}^d)\rightarrow H^1(\mathbb{R}^d)}
\leqslant C_3\varepsilon ,\\
0 < \vartheta \leqslant 1,\quad 0<\varepsilon \leqslant 1.
\end{split}
\end{equation}
Now we carry this estimate over to all $0<\varepsilon \leqslant \vartheta ^{-1}$.
For $1<\varepsilon \leqslant \vartheta ^{-1}$ we use rather rough estimates.
By analogy with \eqref{Gen res B_eps L_2->L_2} and \eqref{6.22a}, we see that
\begin{equation}
\label{(B eps + lambda 1 Q) -1 L2 H1}
\Vert (B_\varepsilon (\vartheta )+\lambda _0 Q_0^\varepsilon )^{-1}\Vert _{L_2(\mathbb{R}^d)
\rightarrow H^1(\mathbb{R}^d)}\leqslant \mathfrak{C}_3,
\end{equation}
where $\mathfrak{C}_3=\lambda _0^{-1}\Vert f\Vert ^2_{L_\infty}+ \frac{1}{2} \lambda _0^{-1/2} c_*^{-1/2} \Vert f\Vert _{L_\infty}$.

Similarly, from \eqref{b^0(theta)>=} and \eqref{f0_estimates} it follows that 
\begin{equation}
\label{B0+lambda 1 Q0 L2->H1}
\Vert (B^0(\vartheta )+\lambda _0 \overline{Q_0})^{-1}\Vert _{L_2(\mathbb{R}^d)\rightarrow H^1(\mathbb{R}^d)}
\leqslant \mathfrak{C}_3.
\end{equation}

For $1< \varepsilon \leqslant \vartheta^{-1}$ we use Lemma~\ref{Lemma K(eps,theta)} and \eqref{(B eps + lambda 1 Q) -1 L2 H1},
\eqref{B0+lambda 1 Q0 L2->H1}, while for $0<\varepsilon \leqslant 1$ we apply \eqref{Th 3.4 for B eps (theta)}.
Then
\begin{equation}
\label{Th 3.4 for B eps (theta) eps >0}
\begin{split}
\Vert (B_\varepsilon (\vartheta )+\lambda _0 Q_0^\varepsilon )^{-1}-(B^0(\vartheta )+\lambda _0 \overline{Q_0})^{-1}-\varepsilon \mathcal{K}(\varepsilon ;\vartheta)\Vert _{L_2(\mathbb{R}^d)\rightarrow H^1(\mathbb{R}^d)}\leqslant \widehat{C}_3\varepsilon ,\\
0 < \vartheta \leqslant 1,\quad 0< \varepsilon \leqslant \vartheta^{-1},
\end{split}
\end{equation}
where $\widehat{C}_3=\max\lbrace C_3 ; 2 \mathfrak{C}_3 +C_K^{(1)}+C_K^{(2)}+C_K^{(3)}\rbrace $.

We put
\begin{equation}
\label{7.15a}
K(\varepsilon ;\vartheta ;\zeta):=\left( [\Lambda ^\varepsilon ]b(\mathbf{D})+\vartheta [\widetilde{\Lambda }^\varepsilon ]\right) S_\varepsilon (B^0(\vartheta )-\zeta \overline{Q_0})^{-1},\quad \zeta \in \mathbb{C}\setminus \mathbb{R}_+.
\end{equation}
Note that $K(\varepsilon ;\vartheta ;-\lambda _0)=\mathcal{K}(\varepsilon ;\vartheta)$.

We prove an analog of estimate \eqref{Th 3.4 for B eps (theta) eps >0} for the operator
$(B_\varepsilon (\vartheta )-\widehat{\zeta} Q_0^\varepsilon )^{-1}$,  where $\widehat{\zeta}=e^{i\phi}$ with $\phi \in  (0,2\pi)$,
with the help of the identity
\begin{equation}
\label{tozd corrector}
\begin{split}
(&B_\varepsilon (\vartheta)-\widehat{\zeta} Q_0^\varepsilon )^{-1}-(B^0(\vartheta )-\widehat{\zeta} \overline{Q_0})^{-1}-\varepsilon K(\varepsilon ;\vartheta ;\widehat{\zeta})\\
&=(B_\varepsilon (\vartheta )-\widehat{\zeta} Q_0^\varepsilon )^{-1}(B_\varepsilon (\vartheta )+\lambda _0 Q_0^\varepsilon )\\
&\times \left( (B_\varepsilon (\vartheta )+\lambda_0 Q_0^\varepsilon )^{-1}-(B^0(\vartheta )+\lambda_0 \overline{Q_0})^{-1}-\varepsilon K(\varepsilon ;\vartheta ;-\lambda_0 )\right)\\
&\times (B^0(\vartheta )+\lambda_0 \overline{Q_0})(B^0(\vartheta )-\widehat{\zeta} \overline{Q_0})^{-1}\\
&+\varepsilon (\lambda _0 +\widehat{\zeta} )(B_\varepsilon (\vartheta)-\widehat{\zeta} Q_0^\varepsilon )^{-1}Q_0^\varepsilon K(\varepsilon ;\vartheta ;\widehat{\zeta} )\\
&+(\lambda_0 +\widehat{\zeta})(B_\varepsilon (\vartheta )-\widehat{\zeta} Q_0^\varepsilon )^{-1}(Q_0^\varepsilon -\overline{Q_0})(B^0(\vartheta )-\widehat{\zeta} \overline{Q_0})^{-1}.
\end{split}
\end{equation}
Denote the consecutive summands in the right-hand side by
${\mathcal J}_l(\varepsilon; \vartheta;\widehat{\zeta})$, \hbox{$l=1,2,3$}.
First, we estimate the $(L_2\rightarrow L_2)$-norm of each summand.
By \eqref{6.10aa}, \eqref{mathscr K(eps;theta)}, and \eqref{7.15a},
\begin{equation*}
\begin{split}
\Vert &Q_0^\varepsilon K(\varepsilon ;\vartheta ;\widehat{\zeta})\Vert _{L_2(\mathbb{R}^d)\rightarrow L_2(\mathbb{R}^d)}\leqslant \Vert  f^{-1}\Vert ^2_{L_\infty} \Vert \mathcal{K}(\varepsilon ;\vartheta )\Vert _{L_2(\mathbb{R}^d)\rightarrow L_2(\mathbb{R}^d)}\\
&\times
\Vert (B^0(\vartheta)+\lambda _0\overline{Q_0})(B^0(\vartheta )-\widehat{\zeta}\overline{Q_0})^{-1}\Vert _{L_2(\mathbb{R}^d)\rightarrow L_2(\mathbb{R}^d)}.
\end{split}
\end{equation*}
Together with \eqref{tozd prav obveska} and \eqref{K<= L_2 Lm} this yields
\begin{equation}
\label{7.17a}
\Vert Q_0^\varepsilon K(\varepsilon ;\vartheta ;\widehat{\zeta})\Vert _{L_2(\mathbb{R}^d)\rightarrow L_2(\mathbb{R}^d)}\leqslant 2(1+\lambda _0) C_K^{(1)} \Vert f\Vert _{L_\infty} \Vert f^{-1}\Vert ^3 _{L_\infty}c(\phi).
\end{equation}
Relations \eqref{Gen res B_eps L_2->L_2} and \eqref{7.17a} imply the following estimate for
the operator ${\mathcal J}_2(\varepsilon; \vartheta;\widehat{\zeta})$:
\begin{equation*}
\Vert {\mathcal J}_2(\varepsilon; \vartheta; \widehat{\zeta})
\Vert _{L_2(\mathbb{R}^d)\rightarrow L_2(\mathbb{R}^d)}
\leqslant 2\varepsilon (\lambda _0+1)^2 \Vert f\Vert ^3_{L_\infty}\Vert f^{-1}\Vert ^3 _{L_\infty}C_K^{(1)}c(\phi)^2.
\end{equation*}
To estimate the $(L_2\rightarrow L_2)$-norm of the operator ${\mathcal J}_1(\varepsilon;\vartheta; \widehat{\zeta})$, we
use \eqref{tozd levaja obveska}, \eqref{tozd prav obveska}, and \eqref{Th 3.4 for B eps (theta) eps >0};
to estimate the term ${\mathcal J}_3(\varepsilon; \vartheta;\widehat{\zeta})$, we apply \eqref{lemma Q 0 eps -Q},
\eqref{norm second summand in tozd}, \eqref{Beps+Q_0 H-1 ->L2}, and \eqref{B0+Q0 L2->H1}.
Then we arrive at
\begin{equation}
\label{<= tilde C7}
\begin{split}
\Vert &(B_\varepsilon (\vartheta )-\widehat{\zeta}Q_0^\varepsilon )^{-1}-(B^0(\vartheta )-\widehat{\zeta}\overline{Q_0})^{-1}-\varepsilon K(\varepsilon ;\vartheta ;\widehat{\zeta})\Vert _{L_2(\mathbb{R}^d)\rightarrow L_2(\mathbb{R}^d)}\\
&\leqslant C_5\varepsilon c(\phi)^2 ,\quad
 0 < \vartheta \leqslant 1,\quad 0 < \varepsilon  \leqslant \vartheta^{-1},\quad \widehat{\zeta}=e^{i\phi},\quad 0<\phi <2\pi .
\end{split}
\end{equation}
Here
\begin{equation*}
\begin{split}
C_5&=4(1+\lambda _0)^2\Vert f\Vert ^2_{L_\infty}\Vert f^{-1}\Vert ^2_{L_\infty}\widehat{C}_3+2(1+\lambda _0)^2\Vert f\Vert ^3_{L_\infty}\Vert f^{-1}\Vert ^3_{L_\infty}C_K^{(1)}\\
&+
(1+\lambda _0) C_{Q_0}\mathfrak{C}_1 ^2.
\end{split}
\end{equation*}

Now, we apply the operator $B_\varepsilon (\vartheta )^{1/2}$ to both sides of \eqref{tozd corrector}:
\begin{equation}
\label{6.20a}
\begin{split}
&B_\varepsilon (\vartheta)^{1/2}
\left((B_\varepsilon (\vartheta)-\widehat{\zeta} Q_0^\varepsilon )^{-1}-(B^0(\vartheta )-\widehat{\zeta} \overline{Q_0})^{-1}-\varepsilon K(\varepsilon ;\vartheta ;\widehat{\zeta}) \right) \\
&=B_\varepsilon (\vartheta)^{1/2} {\mathcal J}_1(\varepsilon; \vartheta;\widehat{\zeta})
 + B_\varepsilon (\vartheta)^{1/2} {\mathcal J}_2(\varepsilon; \vartheta; \widehat{\zeta})
 + B_\varepsilon (\vartheta)^{1/2} {\mathcal J}_3(\varepsilon; \vartheta; \widehat{\zeta}),
 \end{split}
\end{equation}
and estimate the $(L_2\rightarrow L_2)$-norm of each term on the right.

By \eqref{6.10a}, for any $\mathbf{w}\in H^1(\mathbb{R}^d;\mathbb{C}^n)$ we have 
\begin{equation}
\label{rim 1}
\begin{split}
\Vert &B_\varepsilon (\vartheta )^{1/2}\mathbf{w}\Vert ^2_{L_2(\mathbb{R}^d)}
=b_\varepsilon (\vartheta )[\mathbf{w},\mathbf{w}]\\
&=\widetilde{b}_\varepsilon (\vartheta )[(f^\varepsilon )^{-1}\mathbf{w},(f^\varepsilon )^{-1}\mathbf{w}]
=\Vert \widetilde{B}_\varepsilon (\vartheta) ^{1/2}(f^\varepsilon )^{-1}\mathbf{w}\Vert ^2_{L_2(\mathbb{R}^d)}.
\end{split}
\end{equation}
Next, from \eqref{Res B eps and tilde B eps} it follows that
\begin{equation}
\label{rim 2}
(f^\varepsilon)^*(B_\varepsilon (\vartheta )+\lambda _0 Q_0^\varepsilon )
=(\widetilde{B}_\varepsilon (\vartheta )+\lambda _0 I)(f^\varepsilon)^{-1}.
\end{equation}
Using \eqref{Res B eps and tilde B eps}, \eqref{6.17b}, \eqref{rim 1}, and \eqref{rim 2}, we obtain
\begin{equation}
\label{1789}
\begin{split}
\Vert& B_\varepsilon (\vartheta )^{1/2}(B_\varepsilon (\vartheta )-\widehat{\zeta}Q_0^\varepsilon)^{-1}
(B_\varepsilon (\vartheta )+\lambda _0 Q_0^\varepsilon )\Vert _{H^1(\mathbb{R}^d)\rightarrow L_2(\mathbb{R}^d)}\\
&\leqslant \Vert (\widetilde{B}_\varepsilon (\vartheta )-\widehat{\zeta}I)^{-1}(\widetilde{B}_\varepsilon(\vartheta)+\lambda _0 I)\Vert _{L_2(\mathbb{R}^d)\rightarrow L_2(\mathbb{R}^d)}\\
&\times
\Vert \widetilde{B}_\varepsilon (\vartheta)^{1/2}(f^\varepsilon )^{-1}\Vert _{H^1(\mathbb{R}^d)\rightarrow L_2(\mathbb{R}^d)}\\
&\leqslant 2(\lambda _0 +1)c(\phi)\Vert B_\varepsilon (\vartheta )^{1/2}\Vert _{H^1(\mathbb{R}^d)\rightarrow L_2(\mathbb{R}^d)}.
\end{split}
\end{equation}
Combining the upper estimate \eqref{3/10a b_eps(theta)>=},
\eqref{tozd prav obveska}, \eqref{Th 3.4 for B eps (theta) eps >0}, and \eqref{1789}, we deduce the following
estimate for the first term in the right-hand side of~\eqref{6.20a}:
\begin{equation}
\label{6.23a}
\begin{split}
& \Vert B_\varepsilon (\vartheta )^{1/2} {\mathcal J}_1(\varepsilon; \vartheta;\widehat{\zeta}) \Vert_{L_2(\mathbb{R}^d)\rightarrow L_2(\mathbb{R}^d)} \leqslant \mathfrak{C}_4\varepsilon c(\phi )^2,\\
& 0 < \vartheta \leqslant 1,\quad 0< \varepsilon \leqslant \vartheta^{-1},\quad \widehat{\zeta}=e^{i\phi},\quad 0<\phi <2\pi,
\end{split}
\end{equation}
where
$\mathfrak{C}_4:= 4(1+\lambda _0)^2c_6^{1/2}\widehat{C}_3\Vert f\Vert _{L_\infty}\Vert f^{-1}\Vert _{L_\infty}$.

Next,
\begin{equation}
\label{6.24}
\begin{split}
 &\Vert B_\varepsilon (\vartheta )^{1/2} {\mathcal J}_2(\varepsilon; \vartheta;\widehat{\zeta}) \Vert_{L_2(\mathbb{R}^d)\rightarrow L_2(\mathbb{R}^d)}
 \\
 &\leqslant
\varepsilon (\lambda _0+1)\Vert B_\varepsilon (\vartheta )^{1/2}(B_\varepsilon (\vartheta )-\widehat{\zeta}Q_0^\varepsilon )^{-1}\Vert _{L_2\rightarrow L_2}
\Vert Q_0^\varepsilon K(\varepsilon ;\vartheta ;\widehat{\zeta})\Vert_{L_2\rightarrow L_2}.
\end{split}
\end{equation}
By \eqref{Res B eps and tilde B eps} and \eqref{rim 1},
\begin{equation}
\label{7.22b}
\begin{split}
\Vert & B_\varepsilon (\vartheta)^{1/2}(B_\varepsilon (\vartheta)-\widehat{\zeta}Q_0^\varepsilon )^{-1}\Vert _{L_2(\mathbb{R}^d)\rightarrow L_2(\mathbb{R}^d)}\\
&=\Vert \widetilde{B}_\varepsilon (\vartheta)^{1/2}(\widetilde{B}_\varepsilon (\vartheta )-\widehat{\zeta}I)^{-1} (f^\varepsilon )^*\Vert _{L_2(\mathbb{R}^d)\rightarrow L_2(\mathbb{R}^d)}.
\end{split}
\end{equation}
Obviously,
\begin{equation}
\label{7.22a}
\Vert \widetilde{B}_\varepsilon (\vartheta)^{1/2}(\widetilde{B}_\varepsilon (\vartheta )-\widehat{\zeta}I)^{-1} \Vert _{L_2(\mathbb{R}^d)\rightarrow L_2(\mathbb{R}^d)}
\leqslant \sup \limits _{\nu \geqslant 0}\frac{\nu^{1/2}}{\vert \nu -\widehat{\zeta}\vert}\leqslant c(\phi).
\end{equation}
Relations \eqref{7.22b} and \eqref{7.22a} imply that
\begin{equation}
\label{6.27}
\begin{split}
\Vert B_\varepsilon (\vartheta )^{1/2}(B_\varepsilon (\vartheta )-\widehat{\zeta}Q_0^\varepsilon )^{-1}\Vert _{L_2(\mathbb{R}^d)\rightarrow L_2(\mathbb{R}^d)}\leqslant \Vert f\Vert _{L_\infty}c(\phi).
\end{split}
\end{equation}
From \eqref{7.17a}, \eqref{6.24}, and \eqref{6.27} we deduce the following
estimate for the second term in the right-hand side of \eqref{6.20a}:
\begin{equation}
\label{6.27a}
\begin{split}
& \Vert B_\varepsilon (\vartheta )^{1/2} {\mathcal J}_2( \varepsilon;\vartheta; \widehat{\zeta}) \Vert_{L_2(\mathbb{R}^d)\rightarrow L_2(\mathbb{R}^d)} \leqslant \mathfrak{C}_5\varepsilon c(\phi )^2,\\
& 0 < \vartheta \leqslant 1,\quad 0< \varepsilon \leqslant \vartheta^{-1},\quad \widehat{\zeta}=e^{i\phi},\quad 0<\phi <2\pi,
\end{split}
\end{equation}
where
$\mathfrak{C}_5:= 2(1+\lambda _0)^2 C_K^{(1)} \Vert f\Vert _{L_\infty}^2 \Vert f^{-1}\Vert _{L_\infty}^3$.

Now we proceed to the third term in the right-hand side of \eqref{6.20a}. Obviously,
\begin{equation}
\label{6.27b}
\begin{split}
& \Vert B_\varepsilon (\vartheta )^{1/2} {\mathcal J}_3(\varepsilon;\vartheta; \widehat{\zeta}) \Vert_{L_2(\mathbb{R}^d)\rightarrow L_2(\mathbb{R}^d)}
\\
&\leqslant
(\lambda _0 +1)\Vert B_\varepsilon (\vartheta)^{1/2}(B_\varepsilon (\vartheta )-\widehat{\zeta }Q_0^\varepsilon )^{-1}\Vert _{H^{-1}(\mathbb{R}^d)\rightarrow L_2(\mathbb{R}^d)}\\
&\times
\Vert [Q_0^\varepsilon -\overline{Q_0}]\Vert _{H^1(\mathbb{R}^d)\rightarrow H^{-1}(\mathbb{R}^d)}\Vert (B^0(\vartheta )-\widehat{\zeta}\overline{Q_0})^{-1}\Vert _{L_2(\mathbb{R}^d)\rightarrow H^1(\mathbb{R}^d)}.
\end{split}
\end{equation}

Using \eqref{Res B eps and tilde B eps}, \eqref{rim 1}, and the duality arguments, we have
\begin{equation}
\label{rim 5}
\begin{split}
\Vert& B_\varepsilon (\vartheta )^{1/2}(B_\varepsilon (\vartheta )-\widehat{\zeta}Q_0^\varepsilon )^{-1}
\Vert _{H^{-1}(\mathbb{R}^d)\rightarrow L_2(\mathbb{R}^d)}
\\
&=\Vert \widetilde{B}_\varepsilon (\vartheta )^{1/2} (\widetilde{B}_\varepsilon (\vartheta )-\widehat{\zeta} I  )^{-1} (f^\varepsilon)^*
\Vert _{H^{-1}(\mathbb{R}^d)\rightarrow L_2(\mathbb{R}^d)}\\
&=
\Vert f^\varepsilon (\widetilde{B}_\varepsilon (\vartheta )-\widehat{\zeta} ^* I)^{-1}\widetilde{B}_\varepsilon (\vartheta )^{1/2}\Vert _{L_2(\mathbb{R}^d)\rightarrow H^1(\mathbb{R}^d)}\\
&=\Vert f^\varepsilon \widetilde{B}_\varepsilon (\vartheta )^{1/2}(\widetilde{B}_\varepsilon (\vartheta )-\widehat{\zeta} ^* I)^{-1}\Vert _{L_2(\mathbb{R}^d)\rightarrow H^1(\mathbb{R}^d)}.
\end{split}
\end{equation}
The lower estimate \eqref{3/10a b_eps(theta)>=} and \eqref{6.10a} imply that
\begin{equation}
\label{rim 6}
\begin{split}
\Vert &\mathbf{D}[f^\varepsilon ]\widetilde{B}_\varepsilon(\vartheta )^{1/2}(\widetilde{B}_\varepsilon (\vartheta )-\widehat{\zeta}^* I)^{-1}
\Vert _{L_2(\mathbb{R}^d) \rightarrow L_2(\mathbb{R}^d)}\\
&\leqslant c_*^{-1/2}
\Vert B_\varepsilon (\vartheta )^{1/2}f^\varepsilon \widetilde{B}_\varepsilon (\vartheta)^{1/2}(\widetilde{B}_\varepsilon (\vartheta )-\widehat{\zeta}^*I)^{-1}\Vert _{L_2(\mathbb{R}^d) \rightarrow L_2(\mathbb{R}^d)}\\
&= c_*^{-1/2}
\Vert \widetilde{B}_\varepsilon (\vartheta)(\widetilde{B}_\varepsilon (\vartheta )-\widehat{\zeta}^*I)^{-1}\Vert _{L_2(\mathbb{R}^d)\rightarrow L_2(\mathbb{R}^d)}.
\end{split}
\end{equation}
Obviously,
\begin{equation}
\label{7.25a}
\Vert \widetilde{B}_\varepsilon (\vartheta)(\widetilde{B}_\varepsilon (\vartheta )-\widehat{\zeta}^*I)^{-1}\Vert _{L_2(\mathbb{R}^d)\rightarrow L_2(\mathbb{R}^d)}
\leqslant \sup \limits _{\nu \geqslant 0}\frac{\nu}{\vert \nu -\widehat{\zeta}^*\vert}\leqslant c(\phi).
\end{equation}
From  \eqref{7.22a} at the point $\widehat{\zeta}^*$ and \eqref{rim 5}--\eqref{7.25a} it follows that
\begin{equation}
\label{rim 7}
\begin{split}
\Vert B_\varepsilon (\vartheta )^{1/2}(B_\varepsilon (\vartheta )-\widehat{\zeta}Q_0^\varepsilon )^{-1}\Vert _{H^{-1}(\mathbb{R}^d)\rightarrow L_2(\mathbb{R}^d)}\leqslant (\Vert f\Vert _{L_\infty}+ c_*^{-1/2})c(\phi).
\end{split}
\end{equation}
Inequalities  \eqref{lemma Q 0 eps -Q}, \eqref{B0+Q0 L2->H1}, \eqref{6.27b}, and \eqref{rim 7} imply that
\begin{equation}
\label{6.31a}
\begin{split}
& \Vert B_\varepsilon (\vartheta )^{1/2} {\mathcal J}_3(\vartheta; \varepsilon; \widehat{\zeta}) \Vert_{L_2(\mathbb{R}^d)\rightarrow L_2(\mathbb{R}^d)} \leqslant \mathfrak{C}_6\varepsilon c(\phi )^2,\\
& 0 < \vartheta \leqslant 1,\quad 0< \varepsilon \leqslant \vartheta^{-1},\quad \widehat{\zeta}=e^{i\phi},\quad 0<\phi <2\pi,
\end{split}
\end{equation}
where $\mathfrak{C}_6:= (1+\lambda _0)  (\Vert f\Vert _{L_\infty} + c_*^{-1/2}) C_{Q_0} {\mathfrak C}_1$.

Finally, relations \eqref{6.20a}, \eqref{6.23a}, \eqref{6.27a}, and \eqref{6.31a} lead to the estimate
\begin{equation}
\label{7.27}
\begin{split}
\Vert & B_\varepsilon (\vartheta )^{1/2}
\bigl(
(B_\varepsilon (\vartheta )-\widehat{\zeta}Q_0^\varepsilon )^{-1}-(B^0(\vartheta )-\widehat{\zeta}\overline{Q_0})^{-1}-\varepsilon K(\varepsilon ;\vartheta ;\widehat{\zeta})
\bigr)
\Vert _{L_2 \rightarrow L_2}\\
&\leqslant \widehat{C}_6\varepsilon c(\phi)^2,
\quad 0< \vartheta \leqslant 1,\quad 0< \varepsilon \leqslant \vartheta^{-1},\quad \widehat{\zeta}=e^{i\phi},\quad 0<\phi<2\pi ,
\end{split}
\end{equation}
with $\widehat{C}_6=\mathfrak{C}_4+ \mathfrak{C}_5 + \mathfrak{C}_6$.
By the lower estimate \eqref{3/10a b_eps(theta)>=}, from \eqref{7.27} it follows that
\begin{equation}
\label{7.28}
\begin{split}
\Vert &\mathbf{D}\bigl(
(B_\varepsilon (\vartheta )-\widehat{\zeta}Q_0^\varepsilon )^{-1}-(B^0(\vartheta )-\widehat{\zeta}\overline{Q_0})^{-1}-\varepsilon K(\varepsilon ;\vartheta ;\widehat{\zeta})
\bigr)
\Vert _{L_2(\mathbb{R}^d)\rightarrow L_2(\mathbb{R}^d)}
\\
&\leqslant C_6\varepsilon c(\phi)^2,\quad 0 < \vartheta \leqslant 1,\quad
0 < \varepsilon \leqslant \vartheta^{-1},\quad\widehat{\zeta}=e^{i\phi},\quad 0<\phi <2\pi ,
\end{split}
\end{equation}
where $C_6= c_*^{-1/2} \widehat{C}_6$.

Applying the scaling transformation, from \eqref{<= tilde C7} and \eqref{7.28} we deduce
\begin{align}
\label{7.29}
\begin{split}
\Vert &(B(\varepsilon ;\vartheta )-\widehat{\zeta}\varepsilon ^2Q_0)^{-1}-(B^0(\varepsilon ;\vartheta )-\widehat{\zeta}\varepsilon ^2 \overline{Q_0})^{-1}-\check{K}(\varepsilon ;\vartheta ;\widehat{\zeta})\Vert _{L_2\rightarrow L_2}\\
&\leqslant C_5\varepsilon ^{-1}c(\phi)^2,
\end{split}\\
\label{7.30}
\begin{split}
\Vert& \mathbf{D}
\bigl(
(B(\varepsilon ;\vartheta )-\widehat{\zeta}\varepsilon ^2Q_0)^{-1}-(B^0(\varepsilon ;\vartheta )-\widehat{\zeta}\varepsilon ^2 \overline{Q_0})^{-1}-\check{K}(\varepsilon ;\vartheta ;\widehat{\zeta})
\bigr)
\Vert _{L_2\rightarrow L_2}\\
&\leqslant C_6 c(\phi)^2,
\end{split}
\end{align}
where
\begin{equation*}
\check{K}(\varepsilon ;\vartheta ;\widehat{\zeta}):=\left([\Lambda ]b(\mathbf{D})+\varepsilon \vartheta [\widetilde{\Lambda}]\right)S_1(B^0(\varepsilon ;\vartheta )-\widehat{\zeta}\varepsilon ^2\overline{Q_0})^{-1}.
\end{equation*}
As in the proof of Theorem~\ref{Theorem main result L2}, 
in \eqref{7.29} and \eqref{7.30} we substitute $\widetilde{\varepsilon}\vert \zeta \vert ^{1/2}$ in place of $\varepsilon$,
assuming that~$0<\widetilde{\varepsilon}\leqslant 1$, $\zeta =\vert \zeta \vert e^{i\phi}$, $\vert \zeta \vert \geqslant 1$, and put
$\vartheta = \vert \zeta  \vert ^{-1/2}$. Next, we rename $\widetilde{\varepsilon} =: \varepsilon$ and apply the
inverse scaling transformation, taking into account that
\begin{equation*}
\varepsilon K(\varepsilon ;\zeta )=\varepsilon ^2T_\varepsilon ^*\check{K}(\varepsilon \vert \zeta \vert ^{1/2};\vert \zeta \vert ^{-1/2};
\widehat{\zeta} )T_\varepsilon,\quad \zeta =\vert \zeta \vert e^{i\phi}.
\end{equation*}
Here $T_\varepsilon$ is given by \eqref{T_eps},
and the operator $K(\varepsilon ;\zeta)$ is given by~\eqref{K(eps;zeta)}.
This leads to the required estimates \eqref{Th 3.2 L_2->L_2}, \eqref{Th 3.2 D L_2->L_2}, and completes the proof
 of Theorem~\ref{Theorem main result H1}.

\qed

\subsection{Proof of Theorem \ref{Theorem main result fluxes}}
\label{Subsection proof fluxes}
We deduce the statement of Theorem~\ref{Theorem main result fluxes} from Theorem~\ref{Theorem main result H1}.
From \eqref{Th 3.2 D L_2->L_2} and \eqref{<b^*b<} it follows that
\begin{equation}
\label{7.30a}
\begin{split}
\Vert &g^\varepsilon b(\mathbf{D})(B_\varepsilon -\zeta Q_0^\varepsilon )^{-1}\\
&-g^\varepsilon b(\mathbf{D})\left(I+\varepsilon \left( [\Lambda ^\varepsilon]b(\mathbf{D})+[\widetilde{\Lambda}^\varepsilon ]\right)S_\varepsilon\right)(B^0-\zeta \overline{Q_0})^{-1}\Vert _{L_2\rightarrow L_2}\\
&\leqslant \Vert g\Vert _{L_\infty}\alpha _1^{1/2}C_6c(\phi )^2\varepsilon .
\end{split}
\end{equation}
By \eqref{b(D)=},
\begin{equation}
\label{d-vu th 4/4 1}
\begin{split}
\varepsilon &
g^\varepsilon b(\mathbf{D})\left( [\Lambda ^\varepsilon ]b(\mathbf{D})+[\widetilde{\Lambda }^\varepsilon ]\right)S_\varepsilon (B^0-\zeta \overline{Q_0})^{-1}\\
&=g^\varepsilon (b(\mathbf{D})\Lambda )^\varepsilon S_\varepsilon b(\mathbf{D})(B^0-\zeta \overline{Q_0})^{-1}
+g^\varepsilon (b(\mathbf{D})\widetilde{\Lambda})^\varepsilon S_\varepsilon (B^0-\zeta \overline{Q_0})^{-1}\\
&+\varepsilon \sum _{l=1}^d g^\varepsilon b_l\left([\Lambda ^\varepsilon ]S_\varepsilon b(\mathbf{D})D_l +[\widetilde{\Lambda }^\varepsilon ]S_\varepsilon D_l\right) (B^0-\zeta \overline{Q_0})^{-1}.
\end{split}
\end{equation}

Relations \eqref{<b^*b<}, \eqref{bl<=}, \eqref{tildeLambda S_eps}, and \eqref{Lambda S_eps <=}
imply the following estimate for the third term in the right-hand side of \eqref{d-vu th 4/4 1}:
\begin{equation}
\label{7.32}
\begin{split}
\biggl\Vert & \varepsilon \sum _{l=1}^d g^\varepsilon b_l\left([\Lambda ^\varepsilon]S_\varepsilon b(\mathbf{D})D_l +[\widetilde{\Lambda }^\varepsilon ]S_\varepsilon D_l \right) (B^0-\zeta \overline{Q_0})^{-1} \biggr\Vert _{L_2\rightarrow L_2}\\
&\leqslant \varepsilon \Vert g\Vert _{L_\infty}\alpha _1^{1/2} M_1\sum _{l=1}^d \Vert b(\mathbf{D})D_l(B^0-\zeta \overline{Q_0})^{-1}
\Vert _{L_2\rightarrow L_2}\\
&+\varepsilon \Vert g\Vert _{L_\infty}\alpha _1^{1/2}\widetilde{M}_1\sum _{l=1}^d \Vert D_l (B^0-\zeta \overline{Q_0})^{-1}\Vert _{L_2\rightarrow L_2}\\
&\leqslant \varepsilon \Vert g\Vert _{L_\infty}\alpha _1 d^{1/2} M_1
\Vert {\mathbf D}^2 (B^0-\zeta \overline{Q_0})^{-1}\Vert _{L_2\rightarrow L_2}\\
&+ \varepsilon \Vert g\Vert _{L_\infty}\alpha _1^{1/2}d^{1/2}\widetilde{M}_1\Vert \mathbf{D}(B^0-\zeta \overline{Q_0})^{-1}\Vert _{L_2\rightarrow L_2}.
\end{split}
\end{equation}
From \eqref{L0} it follows that
\begin{equation}
\label{rim 8}
\Vert {\mathbf D}^2 (B^0 -\zeta \overline{Q_0})^{-1}\Vert _{L_2 \rightarrow L_2 }\leqslant
{c}_*^{-1}\Vert B^0 (B^0-\zeta \overline{Q_0})^{-1}\Vert _{L_2\rightarrow L_2}.
\end{equation}
Using \eqref{f0_estimates} and the relation between the operators
$B^0=B^0(1)$ and $\widetilde{B}^0= \widetilde{B}^0(1)$ (see Subsection~\ref{Subsection tile-operators(theta)}), we obtain
\begin{equation}
\label{7.39_star}
\begin{split}
\Vert &B^0 (B^0-\zeta \overline{Q_0})^{-1}\Vert _{L_2\rightarrow L_2}
=
\Vert f_0^{-1} \widetilde{B}^0  (\widetilde{B}^0-\zeta I )^{-1} f_0 \Vert _{L_2 \rightarrow L_2 }
\\
&\leqslant \Vert f\Vert _{L_\infty}\Vert f^{-1}\Vert _{L_\infty}\sup \limits _{\nu \geqslant 0}\frac{\nu}{\vert \nu -\zeta \vert}
\leqslant \Vert f\Vert _{L_\infty}\Vert f^{-1}\Vert _{L_\infty} c(\phi)
.
\end{split}
\end{equation}
Together with \eqref{rim 8} this implies
\begin{equation}
\label{rim 8a}
\Vert {\mathbf D}^2 (B^0-\zeta \overline{Q_0})^{-1}\Vert _{L_2\rightarrow L_2}\leqslant {c}_*^{-1}
\Vert f\Vert _{L_\infty}\Vert f^{-1}\Vert _{L_\infty} c(\phi).
\end{equation}
Similarly, using \eqref{b^0[u,u]>=}, we see that
\begin{equation}
\label{rim 9}
\begin{split}
\Vert &\mathbf{D}(B^0-\zeta \overline{Q_0})^{-1}\Vert _{L_2 \rightarrow L_2}
\leqslant {c}_*^{-1/2}\Vert (B^0)^{1/2}(B^0-\zeta \overline{Q_0})^{-1}\Vert _{L_2\rightarrow L_2}\\
&={c}_*^{-1/2}\Vert (\widetilde{B}^0)^{1/2}(\widetilde{B}^0-\zeta I)^{-1}f_0\Vert _{L_2 \rightarrow L_2 }
\\
&\leqslant {c}_*^{-1/2}\Vert f\Vert _{L_\infty}\sup \limits _{\nu \geqslant 0}\frac{\nu^{1/2}}{\vert \nu -\zeta \vert}
\leqslant {c}_*^{-1/2} \Vert f\Vert _{L_\infty}c(\phi)\vert\zeta\vert ^{-1/2}.
\end{split}
\end{equation}
From \eqref{7.32}, \eqref{rim 8a}, \eqref{rim 9}, and the restriction $\vert \zeta \vert \geqslant 1$ it follows that
\begin{equation}
\label{rim 10}
\begin{split}
\biggl\Vert  \varepsilon \sum _{l=1}^d g^\varepsilon b_l\left([\Lambda ^\varepsilon]S_\varepsilon b(\mathbf{D})D_l +[\widetilde{\Lambda }^\varepsilon ]S_\varepsilon D_l \right) (B^0-\zeta \overline{Q_0})^{-1} \biggr\Vert _{L_2\rightarrow L_2}
\leqslant \mathfrak{C}_7\varepsilon c(\phi),
\end{split}
\end{equation}
where
\begin{equation*}
\begin{split}
\mathfrak{C}_7 &=\Vert g\Vert _{L_\infty}d^{1/2}\Vert f\Vert _{L_\infty}
\left( \alpha _1M_1 {c}_*^{-1} \Vert f^{-1}\Vert _{L_\infty}
+\alpha _1^{1/2}\widetilde{M}_1 {c}_*^{-1/2}\right).
\end{split}
\end{equation*}

By Proposition \ref{Proposition S_eps - I} and relations \eqref{<b^*b<}, \eqref{rim 8a}, we have
\begin{equation}
\label{rim 11}
\begin{split}
\Vert &g^\varepsilon b(\mathbf{D})(I-S_\varepsilon )(B^0-\zeta \overline{Q_0})^{-1}\Vert _{L_2(\mathbb{R}^d)\rightarrow L_2(\mathbb{R}^d)}\\
&\leqslant \Vert g\Vert _{L_\infty}\varepsilon r_1\Vert \mathbf{D}b(\mathbf{D})(B^0-\zeta \overline{Q_0})^{-1}\Vert _{L_2(\mathbb{R}^d)\rightarrow L_2(\mathbb{R}^d)}\\
&\leqslant \varepsilon r_1\alpha_1^{1/2}\Vert g\Vert _{L_\infty}\Vert \mathbf{D}^2(B^0-\zeta \overline{Q_0})^{-1}\Vert _{L_2(\mathbb{R}^d)\rightarrow L_2(\mathbb{R}^d)}
\leqslant \mathfrak{C}_{8}\varepsilon c(\phi),
\end{split}
\end{equation}
where $\mathfrak{C}_{8}:=r_1\alpha _1^{1/2}\Vert g\Vert _{L_\infty} {c}_*^{-1} \Vert f\Vert _{L_\infty}\Vert f^{-1}\Vert _{L_\infty}$.

Finally, relations \eqref{tilde g}, \eqref{7.30a}, \eqref{d-vu th 4/4 1}, \eqref{rim 10}, and \eqref{rim 11}
imply the required estimate \eqref{Th 4.4} with $C_7=\Vert g\Vert _{L_\infty}\alpha _1^{1/2}C_6+\mathfrak{C}_7 +\mathfrak{C}_{8}$.
This completes the proof of Theorem~\ref{Theorem main result fluxes}.

\qed

\section{Removal of the smoothing operator. \\ Special cases}
\label{section_special}

\subsection{Removal of $S_\varepsilon$ in the corrector}
It turns out that the smoothing operator $S _\varepsilon$ in the corrector can be removed
 under some additional assumptions on the matrix-valued functions $\Lambda (\mathbf{x})$ and $\widetilde{\Lambda}(\mathbf{x})$.

\begin{condition}
\label{Condition Lambda in L_infty}
Suppose that the $\Gamma$-periodic solution $\Lambda({\mathbf x})$
of problem \eqref{Lambda problem} is bounded, i.~e., $\Lambda \in L_\infty (\mathbb{R}^d)$.
\end{condition}

Some cases where Condition \ref{Condition Lambda in L_infty} is satisfied were distinguished in
\cite[Lemma~8.7]{BSu06}.

\begin{proposition}[\cite{BSu06}]
\label{Proposition Condition on Lambda holds}
Suppose that at least one of the following assumptions is satisfied{\rm :}

\noindent$1^\circ )$ $d\leqslant 2$\textnormal{;}

\noindent$2^\circ )$ $d\geqslant 1$, and the operator $\mathcal{A}_\varepsilon$ is of the form $\mathcal{A}_\varepsilon =\mathbf{D}^*g^\varepsilon (\mathbf{x})\mathbf{D}$, where $g(\mathbf{x})$ is symmetric matrix with real entries\textnormal{;}

\noindent$3^\circ )$ the dimension $d$ is arbitrary, and $g^0=\underline{g}$, i.~e., 
relations \eqref{underline-g} are satisfied.

\noindent
Then Condition \textnormal{\ref{Condition Lambda in L_infty}} is fulfilled.
\end{proposition}

In order to remove $S _\varepsilon$ in the term of the corrector containing $\widetilde{\Lambda}^\varepsilon $,
it suffices to impose the following condition.

\begin{condition}
\label{Condition on Lambda-tilda}
Suppose that the $\Gamma$-periodic solution $\widetilde{\Lambda}(\mathbf{x})$
of problem \eqref{tildeLambda_problem} is such that 
\begin{equation*}
\widetilde{\Lambda}\in L_p(\Omega),\quad p=2\mbox{ for } d=1,\quad p>2 \mbox{ for } d=2,\quad p=d \mbox{ for } d\geqslant 3.
\end{equation*}
\end{condition}

The following result was obtained in \cite[Proposition~8.11]{SuAA}.

\begin{proposition}[\cite{SuAA}]
\label{Proposition Condition on tilde Lambda holds}
Suppose that at least one of the following assumptions is satisfied{\rm :}

\noindent$1^\circ )$ $d\leqslant 4$\textnormal{;}

\noindent$2^\circ )$ the dimension $d$ is arbitrary and the operator $\mathcal{A}_\varepsilon$
is of the form $\mathbf{D}^*g^\varepsilon (\mathbf{x})\mathbf{D}$, where $g(\mathbf{x})$
is symmetric matrix with real entries.

\noindent
Then Condition \textnormal{\ref{Condition on Lambda-tilda}} is fulfilled.
\end{proposition}

\begin{remark}
\label{Remark A with d=1 and real g}
If $\mathcal{A}_\varepsilon=\mathbf{D}^* g^\varepsilon(\mathbf{x})\mathbf{D}$,
 where $g(\mathbf{x})$ is symmetric matrix with real entries,
 from Theorem {\rm 13.1} of \cite[Chapter {\rm III}]{LaU} it follows that
  the norm $\Vert \Lambda \Vert _{L_\infty}$ does not exceed a number depending only on
  $d$, $\Vert g\Vert _{L_\infty}$, $\Vert g^{-1}\Vert _{L_\infty}$, and $\Omega$,
  while the norm $\Vert \widetilde{\Lambda}\Vert _{L_\infty}$ is controlled in terms of
  $d$, $\rho$, $\Vert g\Vert _{L_\infty}$, $\Vert g^{-1}\Vert _{L_\infty}$, $\Vert a_j\Vert _{L_\rho (\Omega )}$, $j=1,\dots ,d$, and $\Omega$.
  Herewith, Conditions~\textnormal{\ref{Condition Lambda in L_infty}} and \textnormal{\ref{Condition on Lambda-tilda}}
  are satisfied.
\end{remark}

In this subsection, our goal is to prove the following theorem.

\begin{theorem}
\label{Theorem main result H1 no S_eps}
Suppose that the assumptions of Theorem~\textnormal{\ref{Theorem main result L2}} are satisfied.

\noindent
$1^\circ$. Suppose that Condition~\textnormal{\ref{Condition Lambda in L_infty}} is satisfied.
Then for $0<\varepsilon \leqslant 1$ and \hbox{$\zeta \in \mathbb{C}\setminus\mathbb{R}_+$}, $\vert \zeta \vert \geqslant 1$,
we have
\begin{align*}
\begin{split}
\Vert & (B_\varepsilon -\zeta Q_0^\varepsilon )^{-1}-(I+\varepsilon [\Lambda ^\varepsilon ]b(\mathbf{D})+\varepsilon [\widetilde{\Lambda }^\varepsilon ]S_\varepsilon )(B^0-\zeta \overline{Q_0})^{-1}\Vert _{L_2(\mathbb{R}^d)\rightarrow L_2(\mathbb{R}^d)}\\
&\leqslant C_{8}'c(\phi )^2\varepsilon \vert \zeta \vert ^{-1/2},
\end{split}\\
\begin{split}
\Vert &\mathbf{D}
\bigl(
(B_\varepsilon -\zeta Q_0^\varepsilon )^{-1}-(I+\varepsilon [\Lambda ^\varepsilon ]b(\mathbf{D})+\varepsilon [\widetilde{\Lambda }^\varepsilon ]S_\varepsilon )(B^0-\zeta \overline{Q_0})^{-1}
\bigr)
\Vert _{L_2(\mathbb{R}^d)\rightarrow L_2(\mathbb{R}^d)}\\
&\leqslant C_{9}'c(\phi )^2\varepsilon .
\end{split}
\end{align*}
The constants $C_{8}'$ and $C_{9}'$ depend only on the initial data \eqref{problem data} and the norm~$\Vert \Lambda \Vert _{L_\infty}$.

\noindent
$2^\circ$. Suppose that Condition~\textnormal{\ref{Condition on Lambda-tilda}} is satisfied.
Then for~$0<\varepsilon \leqslant 1$ and \hbox{$\zeta \in \mathbb{C}\setminus\mathbb{R}_+$}, $\vert \zeta \vert \geqslant 1$, we have
\begin{align*}
\begin{split}
\Vert & (B_\varepsilon -\zeta Q_0^\varepsilon )^{-1}-(I+\varepsilon [\Lambda ^\varepsilon ]b(\mathbf{D})S_\varepsilon +\varepsilon [\widetilde{\Lambda }^\varepsilon ] )(B^0-\zeta \overline{Q_0})^{-1}\Vert _{L_2(\mathbb{R}^d)\rightarrow L_2(\mathbb{R}^d)}\\
&\leqslant C_{8}''c(\phi )^2\varepsilon \vert \zeta \vert ^{-1/2},
\end{split}\\
\begin{split}
\Vert &\mathbf{D}
\bigl(
(B_\varepsilon -\zeta Q_0^\varepsilon )^{-1}-(I+\varepsilon [\Lambda ^\varepsilon ]b(\mathbf{D})S_\varepsilon +\varepsilon [\widetilde{\Lambda }^\varepsilon ])(B^0-\zeta \overline{Q_0})^{-1}
\bigr)
\Vert _{L_2(\mathbb{R}^d)\rightarrow L_2(\mathbb{R}^d)}\\
&\leqslant C_{9}''c(\phi )^2\varepsilon .
\end{split}
\end{align*}
The constants $C_{8}''$ and $C_{9}''$ are controlled in terms of the initial data \eqref{problem data}, $p$, and the norm
$\Vert \widetilde{\Lambda}\Vert _{L_p(\Omega)}$.

\noindent
$3^\circ$. Suppose that Conditions \textnormal{\ref{Condition Lambda in L_infty}} and
\textnormal{\ref{Condition on Lambda-tilda}} are satisfied.
Then for $0<\varepsilon \leqslant 1$ and $\zeta \in \mathbb{C}\setminus\mathbb{R}_+$, $\vert \zeta \vert \geqslant 1$,
we have
\begin{align}
\begin{split}
\Vert & (B_\varepsilon -\zeta Q_0^\varepsilon )^{-1}-(I+\varepsilon [\Lambda ^\varepsilon ]b(\mathbf{D}) +\varepsilon [\widetilde{\Lambda }^\varepsilon ] )(B^0-\zeta \overline{Q_0})^{-1}\Vert _{L_2(\mathbb{R}^d)\rightarrow L_2(\mathbb{R}^d)}\\
&\leqslant C_{8}c(\phi )^2\varepsilon \vert \zeta \vert ^{-1/2},
\end{split}
\nonumber
\\
\label{7*}
\begin{split}
\Vert &\mathbf{D}
\bigl(
(B_\varepsilon -\zeta Q_0^\varepsilon )^{-1}-(I+\varepsilon [\Lambda ^\varepsilon ]b(\mathbf{D}) +\varepsilon [\widetilde{\Lambda }^\varepsilon ])(B^0-\zeta \overline{Q_0})^{-1}
\bigr)
\Vert _{L_2(\mathbb{R}^d)\rightarrow L_2(\mathbb{R}^d)}\\
&\leqslant C_{9}c(\phi )^2\varepsilon .
\end{split}
\end{align}
The constants $C_{8}$ and $C_{9}$ depend only on the initial data \eqref{problem data}, $p$, and the norms
$\Vert \Lambda \Vert _{L_\infty}$,  $\Vert \widetilde{\Lambda}\Vert _{L_p(\Omega)}$.
\end{theorem}

From Corollary \ref{corollary_Lambda}, Lemma \ref{Lemma_Lambda_tilda3}, and Corollary~\ref{Lemma_Lambda_tilda4}
it follows that the operators under the norm sign in the estimates of Theorem~\ref{Theorem main result H1 no S_eps}
are bounded (under the corresponding assumptions).
The statements of Theorem \ref{Theorem main result H1 no S_eps} follow from \eqref{Th 3.2 L_2->L_2}, \eqref{Th 3.2 D L_2->L_2},
and Lemmas \ref{lemma_Lambda}, \ref{lemma_Lambda_tilde} proved below.

\begin{lemma}
\label{lemma_Lambda}
Suppose that the assumptions of Theorem \textnormal{\ref{Theorem main result L2}} and Condition~\textnormal{\ref{Condition Lambda in L_infty}}
are satisfied. Then for $\zeta\in\mathbb{C}\setminus\mathbb{R}_+$, $\vert \zeta \vert \geqslant 1$, and $0<\varepsilon\leqslant 1$
we have
\begin{align}
\label{Lm 7/2 a}
&\varepsilon \Vert [\Lambda ^\varepsilon ]b(\mathbf{D})(S_\varepsilon -I)(B^0-\zeta\overline{Q_0})^{-1}\Vert _{L_2(\mathbb{R}^d)\rightarrow L_2(\mathbb{R}^d)}\leqslant\mathfrak{C}^{(1)}_\Lambda \varepsilon c(\phi)\vert\zeta\vert ^{-1/2},\\
\label{Lm 7/2 b}
&\varepsilon\Vert \mathbf{D}[\Lambda ^\varepsilon ]b(\mathbf{D})(S_\varepsilon -I)(B^0-\zeta\overline{Q_0})^{-1}\Vert _{L_2(\mathbb{R}^d)\rightarrow L_2(\mathbb{R}^d)}\leqslant\mathfrak{C}^{(2)}_\Lambda \varepsilon c(\phi).
\end{align}
The constants $\mathfrak{C}^{(1)}_\Lambda$ and $\mathfrak{C}^{(2)}_\Lambda$ depend only on the initial data
\eqref{problem data} and $\Vert \Lambda\Vert _{L_\infty}$.
\end{lemma}

\begin{proof} To check \eqref{Lm 7/2 a}, we apply \eqref{<b^*b<} and \eqref{rim 9}:
\begin{equation*}
\begin{split}
\Vert &[\Lambda ^\varepsilon ]b(\mathbf{D})(S_\varepsilon -I)(B^0-\zeta\overline{Q_0})^{-1}\Vert _{L_2(\mathbb{R}^d)\rightarrow L_2(\mathbb{R}^d)}\\
&\leqslant 2\alpha _1^{1/2}\Vert \Lambda \Vert _{L_\infty}\Vert \mathbf{D}(B^0-\zeta\overline{Q_0})^{-1}\Vert _{L_2(\mathbb{R}^d)\rightarrow L_2(\mathbb{R}^d)}
\leqslant \mathfrak{C}^{(1)}_\Lambda  c(\phi)\vert\zeta\vert ^{-1/2},
\end{split}
\end{equation*}
where $\mathfrak{C}^{(1)}_\Lambda = 2\alpha _1^{1/2}c_*^{-1/2}\Vert f\Vert _{L_\infty}\Vert \Lambda\Vert _{L_\infty}$.

Now let us check \eqref{Lm 7/2 b}. We have
\begin{equation}
\label{Lm 7/2 derivatives}
\begin{split}
\varepsilon\partial _j [\Lambda ^\varepsilon ]b(\mathbf{D})(S_\varepsilon -I)(B^0-\zeta\overline{Q_0})^{-1}
&=[(\partial _j\Lambda )^\varepsilon ](S_\varepsilon -I)b(\mathbf{D})(B^0-\zeta\overline{Q_0})^{-1}\\
&+\varepsilon [\Lambda ^\varepsilon ](S_\varepsilon -I)b(\mathbf{D})\partial _j (B^0-\zeta\overline{Q_0})^{-1}.
\end{split}
\end{equation}
Hence,
\begin{equation*}
\begin{split}
\varepsilon ^2\Vert & \mathbf{D}[\Lambda ^\varepsilon ]b(\mathbf{D})(S_\varepsilon -I)(B^0-\zeta\overline{Q_0})^{-1}\Vert ^2_{L_2(\mathbb{R}^d)\rightarrow L_2(\mathbb{R}^d)}\\
&\leqslant 2\Vert [(\mathbf{D}\Lambda)^\varepsilon ](S_\varepsilon -I)b(\mathbf{D})(B^0-\zeta\overline{Q_0})^{-1}\Vert ^2_{L_2(\mathbb{R}^d)\rightarrow L_2(\mathbb{R}^d)}
\\
&+2\varepsilon ^2  \| \Lambda\|^2_{L_\infty} \sum _{j=1}^d \Vert (S_\varepsilon -I)
b(\mathbf{D})\partial _j (B^0-\zeta\overline{Q_0})^{-1}\Vert ^2_{L_2(\mathbb{R}^d)\rightarrow L_2(\mathbb{R}^d)}.
\end{split}
\end{equation*}

Combining this with Corollary  \ref{corollary_Lambda}, we obtain
\begin{equation*}
\begin{split}
\varepsilon ^2\Vert & \mathbf{D}[\Lambda ^\varepsilon ]b(\mathbf{D})(S_\varepsilon -I)(B^0-\zeta\overline{Q_0})^{-1}\Vert ^2_{L_2(\mathbb{R}^d)\rightarrow L_2(\mathbb{R}^d)}\\
&\leqslant 2 \beta _1\Vert (S_\varepsilon -I)b(\mathbf{D})(B^0-\zeta\overline{Q_0})^{-1}\Vert ^2 _{L_2(\mathbb{R}^d)\rightarrow L_2(\mathbb{R}^d)}\\
&+2 \varepsilon ^2\Vert \Lambda \Vert ^2 _{L_\infty}(\beta _2 +1)\sum _{j=1}^d \Vert (S_\varepsilon -I)b(\mathbf{D})\partial _j (B^0-\zeta\overline{Q_0})^{-1}\Vert _{L_2(\mathbb{R}^d)\rightarrow L_2(\mathbb{R}^d)}^2.
\end{split}
\end{equation*}
Applying Proposition \ref{Proposition S_eps - I} to estimate the first term on the right and using \eqref{<b^*b<},
we arrive at
\begin{equation*}
\begin{split}
\varepsilon ^2\Vert & \mathbf{D}[\Lambda ^\varepsilon ]b(\mathbf{D})(S_\varepsilon -I)(B^0-\zeta\overline{Q_0})^{-1}\Vert ^2_{L_2(\mathbb{R}^d)\rightarrow L_2(\mathbb{R}^d)}\\
&\leqslant \varepsilon ^2\alpha _1 \left( 2r_1^2\beta _1+8(\beta _2+1)\Vert \Lambda \Vert ^2_{L_\infty}\right)\Vert \mathbf{D}^2(B^0-\zeta\overline{Q_0})^{-1}\Vert ^2_{L_2(\mathbb{R}^d)\rightarrow L_2(\mathbb{R}^d)}.
\end{split}
\end{equation*}
Together with \eqref{rim 8a} this implies the required inequality~\eqref{Lm 7/2 b} with
\begin{equation*}
\mathfrak{C}_\Lambda ^{(2)}=\alpha _1^{1/2}c_*^{-1}\left(2r_1^2\beta _1+8(\beta _2+1)\Vert \Lambda \Vert ^2_{L_\infty}\right)^{1/2}\Vert f\Vert _{L_\infty}\Vert f^{-1}\Vert _{L_\infty}.
\end{equation*}
\end{proof}

\begin{lemma}
\label{lemma_Lambda_tilde}
Suppose that the assumptions of Theorem \textnormal{\ref{Theorem main result L2}} and Condition~\textnormal{\ref{Condition on Lambda-tilda}} are satisfied.
Then for $\zeta\in\mathbb{C}\setminus\mathbb{R}_+$, $\vert \zeta \vert \geqslant 1$, and $0<\varepsilon\leqslant 1$ we have
\begin{align}
\label{Lm 7/2 v}
&\varepsilon\Vert [\widetilde{\Lambda}^\varepsilon ](S_\varepsilon -I)(B^0-\zeta\overline{Q_0})^{-1}\Vert _{L_2(\mathbb{R}^d)\rightarrow L_2(\mathbb{R}^d)}\leqslant\mathfrak{C}^{(1)}_{\widetilde{\Lambda}}\varepsilon c(\phi)\vert\zeta\vert ^{-1/2},\\
\label{Lm 7/2 g}
&\varepsilon\Vert \mathbf{D}[\widetilde{\Lambda}^\varepsilon](S_\varepsilon -I)(B^0-\zeta\overline{Q_0})^{-1}\Vert _{L_2(\mathbb{R}^d)\rightarrow L_2(\mathbb{R}^d)}\leqslant\mathfrak{C}^{(2)}_{\widetilde{\Lambda}} \varepsilon c(\phi).
\end{align}
The constants $\mathfrak{C}^{(1)}_{\widetilde{\Lambda}}$ and $\mathfrak{C}^{(2)}_{\widetilde{\Lambda}}$
depend only on the initial data \eqref{problem data}, $p$, and $\Vert \widetilde{\Lambda}\Vert _{L_p(\Omega)}$.
\end{lemma}

\begin{proof}
By Lemma~\ref{Lemma_Lambda_tilda3} and Condition~\ref{Condition on Lambda-tilda}, we have
\begin{equation}
\label{Lm 7/2 d}
\begin{split}
\Vert & [\widetilde{\Lambda}^\varepsilon](S_\varepsilon -I)(B^0-\zeta\overline{Q_0})^{-1}\Vert _{L_2(\mathbb{R}^d)\rightarrow L_2(\mathbb{R}^d)}\\
&\leqslant 2 C_\Omega ^{(p)}\Vert \widetilde{\Lambda}\Vert _{L_p(\Omega)}\Vert (B^0-\zeta\overline{Q_0})^{-1}\Vert _{L_2(\mathbb{R}^d)\rightarrow H^1(\mathbb{R}^d)}.
\end{split}
\end{equation}
Using the relation between the operators $B^0=B^0(1)$ and $\widetilde{B}^0=\widetilde{B}^0(1)$ (see Subsection~\ref{Subsection tile-operators(theta)}), and also \eqref{f0_estimates}, we obtain
\begin{equation}
\label{Lm 7/2 e}
\begin{split}
\Vert (B^0-\zeta\overline{Q_0})^{-1}\Vert _{L_2(\mathbb{R}^d)\rightarrow L_2(\mathbb{R}^d)}
&\leqslant\vert f_0\vert ^2 \Vert (\widetilde{B}^0-\zeta I)^{-1}\Vert _{L_2(\mathbb{R}^d)\rightarrow L_2(\mathbb{R}^d)}\\
&\leqslant \Vert f\Vert ^2_{L_\infty}c(\phi)\vert\zeta\vert ^{-1}.
\end{split}
\end{equation}
Relations \eqref{rim 9}, \eqref{Lm 7/2 d}, and \eqref{Lm 7/2 e} imply  \eqref{Lm 7/2 v} with  $$\mathfrak{C}^{(1)}_{\widetilde{\Lambda}}=2C_\Omega^{(p)}\Vert \widetilde{\Lambda}\Vert _{L_p(\Omega)}(\Vert f\Vert ^2_{L_\infty}+c_*^{-1/2}\Vert f\Vert _{L_\infty}).$$

Now we prove  \eqref{Lm 7/2 g}. Similarly to \eqref{Lm 7/2 derivatives},
\begin{equation*}
\begin{split}
\varepsilon\partial _j [\widetilde{\Lambda} ^\varepsilon ]  (S_\varepsilon -I)(B^0-\zeta\overline{Q_0})^{-1}
&=[(\partial _j \widetilde{\Lambda} )^\varepsilon ](S_\varepsilon -I) (B^0-\zeta\overline{Q_0})^{-1}\\
&+\varepsilon [\widetilde{\Lambda} ^\varepsilon ](S_\varepsilon -I) \partial _j (B^0-\zeta\overline{Q_0})^{-1}.
\end{split}
\end{equation*}
Together with  Corollary \ref{Lemma_Lambda_tilda4} and Lemma \ref{Lemma_Lambda_tilda3} this implies
\begin{equation*}
\begin{split}
\varepsilon ^2 \Vert & \mathbf{D}[\widetilde{\Lambda}^\varepsilon](S_\varepsilon -I)(B^0-\zeta\overline{Q_0})^{-1}
\Vert ^2_{ L_2(\mathbb{R}^d)\to L_2(\mathbb{R}^d)}\\
&\leqslant 2\widetilde{\beta}_1\Vert (S_\varepsilon -I)(B^0-\zeta\overline{Q_0})^{-1}\Vert ^2_{L_2(\mathbb{R}^d) \to H^1(\mathbb{R}^d)}\\
&+2(\widetilde{\beta }_2+1)\varepsilon ^2  \|\widetilde{\Lambda} \|_{L_p(\Omega)}^2 (C_\Omega^{(p)})^2
\| (S_\varepsilon -I)\mathbf{D}(B^0-\zeta\overline{Q_0})^{-1}\|^2_{L_2(\mathbb{R}^d) \to H^1(\mathbb{R}^d)}.
\end{split}
\end{equation*}
Applying Proposition \ref{Proposition S_eps - I} to estimate the first term on the right, we arrive at
\begin{equation*}
\begin{split}
&\varepsilon\Vert  \mathbf{D}[\widetilde{\Lambda}^\varepsilon ](S_\varepsilon -I)(B^0-\zeta\overline{Q_0})^{-1}
\Vert  _{ L_2(\mathbb{R}^d) \to L_2(\mathbb{R}^d)}\\
&\leqslant \varepsilon \left(
2\widetilde{\beta}_1r_1^2+8(\widetilde{\beta }_2+1)\left(C_\Omega ^{(p)}\right)^2\Vert \widetilde{\Lambda}\Vert ^2_{L_p(\Omega)}
\right)^{1/2}
\Vert \mathbf{D}(B^0-\zeta\overline{Q_0})^{-1} \Vert _{ L_2(\mathbb{R}^d) \to  H^1(\mathbb{R}^d)}.
\end{split}
\end{equation*}
Combining this with \eqref{rim 8a} and \eqref{rim 9}, we obtain \eqref{Lm 7/2 g} with
\begin{equation*}
\begin{split}
\mathfrak{C}_{\widetilde{\Lambda}}^{(2)}&=\left(
2\widetilde{\beta}_1r_1^2+8(\widetilde{\beta }_2+1)\left(C_\Omega ^{(p)}\right)^2\Vert \widetilde{\Lambda}\Vert ^2_{L_p(\Omega)}
\right)^{1/2}\\
&\times
\left(c_*^{-1}\Vert f\Vert _{L_\infty}\Vert f^{-1}\Vert _{L_\infty}+c_*^{-1/2}\Vert f\Vert _{L_\infty}\right) .
\end{split}
\end{equation*}
\end{proof}

\subsection{Removal of $S_\varepsilon$ in approximation of the flux}

\begin{theorem}
\label{Theorem main result fluxes no S_eps}
Suppose that the assumptions of Theorem~\textnormal{\ref{Theorem main result L2}} are satisfied.
Let $\widetilde{g}(\mathbf{x})$ be given by \eqref{tilde g}.

\noindent
$1^\circ$. Suppose that Condition \textnormal{\ref{Condition Lambda in L_infty}} is satisfied. Denote
\begin{equation}
\label{G_1(eps,zeta)}
G_1(\varepsilon ;\zeta ):=\widetilde{g}^\varepsilon b(\mathbf{D})(B^0-\zeta \overline{Q_0})^{-1}
+g^\varepsilon \bigl(b(\mathbf{D})\widetilde{\Lambda }\bigr)^\varepsilon S_\varepsilon (B^0-\zeta \overline{Q_0})^{-1}.
\end{equation}
Then for $0<\varepsilon \leqslant 1$ and $\zeta \in \mathbb{C}\setminus\mathbb{R}_+$, $\vert \zeta \vert \geqslant 1$, we have
\begin{equation*}
\Vert g^\varepsilon b(\mathbf{D})(B_\varepsilon -\zeta Q_0^\varepsilon )^{-1}-G_1(\varepsilon ;\zeta )\Vert _{L_2(\mathbb{R}^d)\rightarrow L_2(\mathbb{R}^d)}\leqslant C_{10}'c(\phi)^2\varepsilon .
\end{equation*}
The constant $C_{10}'$ is controlled in terms of the problem data \eqref{problem data} and $\Vert \Lambda \Vert _{L_\infty}$.

\noindent
$2^\circ$. Suppose that Condition \textnormal{\ref{Condition on Lambda-tilda}} is satisfied.
Denote 
\begin{equation}
\label{G_2(eps,zeta)}
G_2(\varepsilon ;\zeta ):=\widetilde{g}^\varepsilon S_\varepsilon b(\mathbf{D})(B^0-\zeta \overline{Q_0})^{-1}
+g^\varepsilon \bigl(b(\mathbf{D})\widetilde{\Lambda }\bigr)^\varepsilon  (B^0-\zeta \overline{Q_0})^{-1}.
\end{equation}
Then for $0<\varepsilon \leqslant 1$ and $\zeta \in \mathbb{C}\setminus\mathbb{R}_+$, $\vert \zeta \vert \geqslant 1$, we have
\begin{equation*}
\Vert g^\varepsilon b(\mathbf{D})(B_\varepsilon -\zeta Q_0^\varepsilon )^{-1}-G_2(\varepsilon ;\zeta )\Vert _{L_2(\mathbb{R}^d)\rightarrow L_2(\mathbb{R}^d)}\leqslant C_{10}''c(\phi)^2\varepsilon .
\end{equation*}
The constant $C_{10}''$ depends only on the initial data \eqref{problem data}, $p$, and $\Vert \widetilde{\Lambda}\Vert _{L_p(\Omega)}$.

\noindent
$3^\circ$. Suppose that Conditions \textnormal{\ref{Condition Lambda in L_infty}} and \textnormal{\ref{Condition on Lambda-tilda}}
are satisfied. Denote
\begin{equation}
\label{G_3(eps,zeta)}
G_3(\varepsilon ;\zeta ):=\widetilde{g}^\varepsilon  b(\mathbf{D})(B^0-\zeta \overline{Q_0})^{-1}
+g^\varepsilon \bigl(b(\mathbf{D})\widetilde{\Lambda }\bigr)^\varepsilon  (B^0-\zeta \overline{Q_0})^{-1}.
\end{equation}
Then for $0<\varepsilon \leqslant 1$ and $\zeta \in \mathbb{C}\setminus\mathbb{R}_+$, $\vert \zeta \vert \geqslant 1$, we have
\begin{equation}
\label{7100}
\Vert g^\varepsilon b(\mathbf{D})(B_\varepsilon -\zeta Q_0^\varepsilon )^{-1}-G_3(\varepsilon ;\zeta )\Vert _{L_2(\mathbb{R}^d)\rightarrow L_2(\mathbb{R}^d)}\leqslant C_{10}c(\phi)^2\varepsilon .
\end{equation}
The constant $C_{10}$ depends only on the initial data \eqref{problem data}, $p$, and the norms
$\Vert \Lambda \Vert _{L_\infty}$, $\Vert \widetilde{\Lambda}\Vert _{L_p(\Omega)}$.
\end{theorem}

\begin{proof}
The statement of Theorem \ref{Theorem main result fluxes no S_eps} is deduced from Theorem \ref{Theorem main result H1 no S_eps}.
The proof is similar to that of Theorem \ref{Theorem main result fluxes} (see Subsection~\ref{Subsection proof fluxes}).
To be concrete, let us prove statement~$3^\circ$.
By analogy with \eqref{7.30a}, from \eqref{7*} it follows that
\begin{equation}
\label{710a}
\begin{split}
\Vert &g^\varepsilon b(\mathbf{D})(B_\varepsilon -\zeta Q_0^\varepsilon )^{-1}\\
&-g^\varepsilon b(\mathbf{D})\left(I+\varepsilon \left( [\Lambda ^\varepsilon]b(\mathbf{D})+[\widetilde{\Lambda}^\varepsilon ]\right) \right)
(B^0-\zeta \overline{Q_0})^{-1}\Vert _{L_2\rightarrow L_2}\\
&\leqslant \Vert g\Vert _{L_\infty}\alpha _1^{1/2}C_9 c(\phi )^2\varepsilon .
\end{split}
\end{equation}
Next, similarly to \eqref{d-vu th 4/4 1},
\begin{equation}
\label{710b}
\begin{split}
\varepsilon &
g^\varepsilon b(\mathbf{D})\left( [\Lambda ^\varepsilon ]b(\mathbf{D})+[\widetilde{\Lambda }^\varepsilon ]\right)(B^0-\zeta \overline{Q_0})^{-1}\\
&=g^\varepsilon (b(\mathbf{D})\Lambda )^\varepsilon  b(\mathbf{D})(B^0-\zeta \overline{Q_0})^{-1}
+g^\varepsilon (b(\mathbf{D})\widetilde{\Lambda})^\varepsilon  (B^0-\zeta \overline{Q_0})^{-1}\\
&+\varepsilon \sum _{l=1}^d g^\varepsilon b_l\left([\Lambda ^\varepsilon ] b(\mathbf{D})D_l +[\widetilde{\Lambda }^\varepsilon ] D_l\right) (B^0-\zeta \overline{Q_0})^{-1}.
\end{split}
\end{equation}
The difference with the proof of Theorem \ref{Theorem main result fluxes} is related to
estimation of the third term in the right-hand side of \eqref{710b}.
Condition \ref{Condition Lambda in L_infty} and relations  \eqref{<b^*b<}, \eqref{bl<=} imply that
\begin{equation}
\label{710c}
\begin{split}
\varepsilon &\sum _{j=1}^d\Vert g^\varepsilon b_l[\Lambda ^\varepsilon ]b(\mathbf{D})D_l(B^0-\zeta\overline{Q_0})^{-1}\Vert _{L_2(\mathbb{R}^d)\rightarrow L_2(\mathbb{R}^d)}
\\
&\leqslant\varepsilon\alpha _1d^{1/2}\Vert g\Vert _{L_\infty}\Vert \Lambda\Vert _{L_\infty}\Vert \mathbf{D}^2(B^0-\zeta\overline{Q_0})^{-1}\Vert _{L_2(\mathbb{R}^d)\rightarrow L_2(\mathbb{R}^d)}.
\end{split}
\end{equation}
Using Condition~\ref{Condition on Lambda-tilda}, \eqref{bl<=}, and Lemma~\ref{Lemma_Lambda_tilda3}, we obtain
\begin{equation}
\label{710d}
\begin{split}
\varepsilon &\sum _{j=1}^d\Vert g^\varepsilon b_l [\widetilde{\Lambda}^\varepsilon ]D_l (B^0-\zeta\overline{Q_0})^{-1}\Vert _{L_2(\mathbb{R}^d)\rightarrow L_2(\mathbb{R}^d)}\\
&\leqslant\varepsilon\alpha _1^{1/2}d^{1/2}C_\Omega ^{(p)}\|g\|_{L_\infty}\Vert \widetilde{\Lambda}\Vert _{L_p(\Omega)}\Vert \mathbf{D}(B^0-\zeta\overline{Q_0})^{-1}\Vert _{L_2(\mathbb{R}^d)\rightarrow H^1(\mathbb{R}^d)}.
\end{split}
\end{equation}
From \eqref{rim 8a},  \eqref{rim 9}, \eqref{710c},  and \eqref{710d} it follows that the~$(L_2 \to L_2)$-norm of the third term
in the right-hand side of \eqref{710b} does not exceed $\widehat{C}_{10}c(\phi)\varepsilon$, where
\begin{equation*}
\begin{split}
\widehat{C}_{10}&= \alpha _1 d^{1/2}\Vert g\Vert _{L_\infty}\Vert \Lambda\Vert _{L_\infty} c_*^{-1}\Vert f\Vert _{L_\infty}
\Vert f^{-1} \Vert_{L_\infty}
\cr
 &+
\alpha _1^{1/2}d^{1/2} C_\Omega ^{(p)}\|g\|_{L_\infty} \Vert \widetilde{\Lambda}\Vert _{L_p(\Omega)}
\left(c_*^{-1}\Vert f\Vert _{L_\infty}\Vert f^{-1}\Vert _{L_\infty}+c_*^{-1/2}\Vert f\Vert _{L_\infty}\right) .
\end{split}
\end{equation*}
Together with \eqref{710a} and \eqref{710b} this implies
\eqref{7100} with $C_{10}= \Vert g\Vert _{L_\infty}\alpha _1^{1/2}C_9 + \widehat{C}_{10}$.

Statements  $1^\circ$ and $2^\circ$ are proved in a similar fashion;
to check $2^\circ$, in addition we have to take \eqref{rim 11} into account.
\end{proof}

\subsection{The case of the zero corrector} Suppose that $g^0=\overline{g}$, which is equivalent to~\eqref{overline-g}. 
Then the $\Gamma$-periodic solution of problem~\eqref{Lambda problem}
is equal to zero: $\Lambda (\mathbf{x})=0$. Suppose also that
\begin{equation}
\label{Djaj=0}
\sum _{j=1}^d D_j a_j(\mathbf{x})^*=0.
\end{equation}
Then the $\Gamma$-periodic solution of problem \eqref{tildeLambda_problem} is equal to zero:
$\widetilde{\Lambda}(\mathbf{x})=0$. Under these assumptions, the operator \eqref{K(eps;zeta)} is equal to zero. Hence,
relation \eqref{Th 3.2 D L_2->L_2} simplifies, and Theorem~\ref{Theorem main result H1} implies the following result.

\begin{proposition}
\label{Proposition K=0}
Suppose that the assumptions of Theorem~\textnormal{\ref{Theorem main result L2}} are satisfied.
Suppose also that relations \eqref{overline-g} and \eqref{Djaj=0} are satisfied.
Then for $0<\varepsilon\leqslant 1$ and $\zeta \in \mathbb{C}\setminus\mathbb{R}_+$, $\vert \zeta \vert \geqslant 1$, we have
\begin{equation*}
\Vert \mathbf{D}
\left(
(B_\varepsilon -\zeta Q_0^\varepsilon )^{-1}-(B^0-\zeta \overline{Q_0})^{-1}
\right)
\Vert _{L_2(\mathbb{R}^d)\rightarrow L_2(\mathbb{R}^d)}\leqslant C_6c(\phi )^2\varepsilon .
\end{equation*}
\end{proposition}

\subsection{The special case} Suppose that $g^0=\underline{g}$, which is equivalent to~\eqref{underline-g}.
Then, by Proposition~\ref{Proposition Condition on Lambda holds}($3^\circ$), Condition~\ref{Condition Lambda in L_infty}
is satisfied. According to \cite[Remark~3.5]{BSu05}, in this case the matrix-valued function \eqref{tilde g} is constant
and coincides with $g^0$, i.~e., $\widetilde{g}(\mathbf{x})=g^0=\underline{g}$.
Thus, $\widetilde{g}^\varepsilon b(\mathbf{D})(B^0-\zeta \overline{Q_0})^{-1}=g^0b(\mathbf{D})(B^0-\zeta \overline{Q_0})^{-1}$.

In addition, suppose that relation \eqref{Djaj=0} is satisfied.
Then $\widetilde{\Lambda}(\mathbf{x})=0$, and Theorem~\ref{Theorem main result fluxes no S_eps}($3^\circ$) implies the
following result.

\begin{proposition}
\label{Proposition non-general case}
Suppose that the assumptions of Theorem~\textnormal{\ref{Theorem main result L2}} are satisfied.
Suppose also that relations \eqref{underline-g} and \eqref{Djaj=0} are satisfied.
Then for $0<\varepsilon \leqslant 1$ and $\zeta \in \mathbb{C}\setminus\mathbb{R}_+$, $\vert \zeta \vert \geqslant 1$, we have
\begin{equation*}
\Vert g^\varepsilon b(\mathbf{D})(B_\varepsilon -\zeta Q_0^\varepsilon )^{-1}-g^0b(\mathbf{D})(B^0-\zeta \overline{Q_0})^{-1}\Vert _{L_2(\mathbb{R}^d)\rightarrow L_2(\mathbb{R}^d)}\leqslant C_{10}c(\phi)^2\varepsilon .
\end{equation*}
\end{proposition}

\section{Another approximation \\ of the generalized resolvent $(B_\varepsilon -\zeta Q_0^\varepsilon )^{-1}$}
\label{Section another approximation}

\subsection{The result in the general case}
In Theorems of Sections \ref{Section main results} and \ref{section_special}, it was assumed that
$\zeta \in \mathbb{C}\setminus\mathbb{R}_+$ and $\vert \zeta \vert \geqslant 1$.
In the present subsection, we obtain the results in a wider domain of $\zeta$.

\begin{theorem}
\label{Theorem another approximation}
Suppose that the assumptions of Theorem~\textnormal{\ref{Theorem main result H1}} are satisfied.
Let $f(\mathbf{x})$ and $f_0$ be defined in Subsection \textnormal{\ref{Subsection tile-operators(theta)}}.
Let $\zeta \in \mathbb{C}\setminus [c_\flat ,\infty )$, where $c_\flat \geqslant 0$ is a common lower bound of the operators $\widetilde{B}_\varepsilon =(f^\varepsilon )^*B_\varepsilon f^\varepsilon$ and $\widetilde{B}^0=f_0B^0 f_0$.
We put $\zeta -c_\flat=\vert \zeta -c_\flat\vert e^{i\psi}$, $\psi \in (0,2\pi)$, and denote
\begin{equation}
\label{rho (zeta)}
\varrho (\zeta )=\begin{cases}
c(\psi )^2\vert \zeta -c_\flat \vert ^{-2}, &\vert \zeta -c_\flat \vert <1,\\
c(\psi )^2, &\vert \zeta -c_\flat \vert \geqslant 1.
\end{cases}
\end{equation}
Then for $0<\varepsilon \leqslant 1$ we have
\begin{align}
\label{Th 9.1 1}
\Vert &(B_\varepsilon -\zeta Q_0^\varepsilon )^{-1}-(B^0-\zeta \overline{Q_0})^{-1}\Vert _{L_2(\mathbb{R}^d)\rightarrow L_2(\mathbb{R}^d)}\leqslant C_{11}\varrho (\zeta )\varepsilon ,\\
\label{Th 9.1 2}
\Vert &(B_\varepsilon -\zeta Q_0^\varepsilon )^{-1}-(B^0-\zeta \overline{Q_0})^{-1}-\varepsilon K(\varepsilon ;\zeta)\Vert _{L_2(\mathbb{R}^d)\rightarrow L_2(\mathbb{R}^d)}\leqslant C_{12}\varrho (\zeta )\varepsilon ,\\
\label{Th 9.1 3}
\begin{split}
\Vert &\mathbf{D}
\bigl(
(B_\varepsilon -\zeta Q_0^\varepsilon )^{-1}-(B^0-\zeta \overline{Q_0})^{-1}-\varepsilon K(\varepsilon ;\zeta)
\bigr)\Vert _{L_2(\mathbb{R}^d)\rightarrow L_2(\mathbb{R}^d)}\\
&\leqslant \bigl( C_{13}+\vert \zeta +1\vert ^{1/2}C_{14}\bigr)\varrho (\zeta )\varepsilon .
\end{split}
\end{align}
Let $G(\varepsilon ;\zeta )$ be the operator defined by \eqref{G(eps,zeta)}. Then for $0<\varepsilon \leqslant 1$ we have
\begin{equation}
\label{Th 9.1 4}
\Vert g^\varepsilon b(\mathbf{D})(B_\varepsilon -\zeta Q_0^\varepsilon )^{-1}-G(\varepsilon ;\zeta )\Vert _{L_2(\mathbb{R}^d)\rightarrow L_2(\mathbb{R}^d)}
\leqslant ( C_{15}+\vert \zeta +1\vert ^{1/2}C_{16})\varrho (\zeta )\varepsilon .
\end{equation}
The constants $C_{11}$, $C_{12}$, $C_{13}$, $C_{14}$, $C_{15}$, and $C_{16}$ are controlled
in terms of the initial data \eqref{problem data} and  $c_\flat$.
\end{theorem}

\begin{corollary}
Under the assumptions of Theorem~\textnormal{\ref{Theorem another approximation}}, we have
\begin{equation}
\label{8.5a}
\begin{split}
\Vert &(B_\varepsilon -\zeta Q_0^\varepsilon )^{-1}-(B^0-\zeta \overline{Q_0})^{-1}-\varepsilon K(\varepsilon ;\zeta)\Vert _{L_2(\mathbb{R}^d)\rightarrow H^1(\mathbb{R}^d)}\\
&\leqslant \bigl(C_{12}+ C_{13}+\vert \zeta +1\vert ^{1/2}C_{14}\bigr)\varrho (\zeta )\varepsilon,\quad  0<\varepsilon\leqslant 1.
\end{split}
\end{equation}
\end{corollary}

\begin{remark}
{\rm 1)} We do not control the lower edges of the spectra of the operators $\widetilde{B}_\varepsilon \geqslant 0$ and
$\widetilde{B}^0\geqslant 0$ explicitly. However, we can always take $c_\flat =0$.
In this case, we have $\psi =\phi$, and for $\vert \zeta \vert =\vert \zeta -c_\flat \vert \geqslant 1$
estimates of Theorem {\rm \ref{Theorem another approximation}} are more rough than the results of
Theorems~\textnormal{\ref{Theorem main result L2}, \ref{Theorem main result H1}}, and
{\rm \ref{Theorem main result fluxes}}. {\rm 2)} For large $\vert \zeta \vert$ it is more convenient to apply
Theorems~\textnormal{\ref{Theorem main result L2}}, \textnormal{\ref{Theorem main result H1}}, and {\rm \ref{Theorem main result fluxes}},
while for bounded values of $\vert \zeta \vert$ Theorem~\textnormal{\ref{Theorem another approximation}} may be preferable.
\end{remark}

\begin{proof}
To check \eqref{Th 9.1 1}, we apply Theorem~\ref{Theorem main result L2} for $\zeta =-1$.
According to~\eqref{Th L2 principal part},
\begin{equation*}
\Vert (B_\varepsilon +Q_0^\varepsilon )^{-1}-(B^0+\overline{Q_0})^{-1}\Vert _{L_2(\mathbb{R}^d)\rightarrow L_2(\mathbb{R}^d)}\leqslant C_4\varepsilon ,\quad 0<\varepsilon \leqslant 1.
\end{equation*}
Similarly to \eqref{6.17a}--\eqref{tozd prav obveska},
using the analog of identity~\eqref{tozd generalized resolvent} for $\vartheta=1$
(with $\widehat{\zeta}$ and $\lambda _0$ replaced by $\zeta $ and $1$, respectively),
we obtain
\begin{equation}
\label{I}
\begin{split}
\Vert (B_\varepsilon -\zeta Q_0^\varepsilon )^{-1}-(B^0-\zeta \overline{Q_0})^{-1}\Vert _{L_2\rightarrow L_2}
\leqslant C_4\varepsilon \Vert f\Vert ^2_{L_\infty}\Vert f^{-1}\Vert ^2_{L_\infty}\sup \limits _{\nu \geqslant c_\flat} \frac{(\nu +1)^2}{\vert \nu -\zeta \vert ^2}\\
+\vert 1+\zeta \vert \Vert (B_\varepsilon -\zeta Q_0^\varepsilon )^{-1}\Vert _{H^{-1}\rightarrow L_2}\Vert [Q_0^\varepsilon -\overline{Q_0}]\Vert _{H^1\rightarrow H^{-1}}
\Vert (B^0-\zeta \overline{Q_0})^{-1}\Vert _{L_2\rightarrow H^1}.
\end{split}
\end{equation}
A calculation shows that
\begin{equation}
\label{II}
\sup \limits _{\nu \geqslant c_\flat} \frac{(\nu +1)^2}{\vert \nu -\zeta \vert ^2}\leqslant (c_\flat +2)^2\varrho (\zeta ),\quad \zeta \in \mathbb{C}\setminus [c_\flat ,\infty).
\end{equation}
Next, by the duality argument,
\begin{equation}
\label{III}
\Vert (B_\varepsilon -\zeta Q_0^\varepsilon )^{-1}\Vert _{H^{-1}\rightarrow L_2}
=\Vert (B_\varepsilon -\zeta ^* Q_0^\varepsilon )^{-1}\Vert _{L_2\rightarrow H^1}.
\end{equation}
By \eqref{Res B eps and tilde B eps}, we have
\begin{equation}
\label{IV}
\begin{split}
\Vert (B_\varepsilon -\zeta ^*Q_0^\varepsilon )^{-1}\Vert _{L_2\rightarrow L_2}\leqslant \Vert f\Vert ^2_{L_\infty}\sup\limits _{\nu\geqslant c_\flat}\frac{1}{\vert \nu -\zeta ^*\vert}
=\Vert f\Vert ^2_{L_\infty}c(\psi)\vert \zeta -c_\flat\vert ^{-1}.
\end{split}
\end{equation}
Note that
\begin{align}
\label{V}
&\vert \zeta +1\vert ^{1/2}\leqslant (2+c_\flat)^{1/2}\;\mbox{for}\;\vert \zeta -c_\flat \vert <1,\\
\label{VI}
&\vert \zeta +1\vert ^{1/2}\vert \zeta -c_\flat \vert ^{-1}\leqslant (2+c_\flat)^{1/2}\;\mbox{for}\;\vert \zeta -c_\flat \vert \geqslant 1.
\end{align}
Therefore,
\begin{equation}
\label{VII}
\vert \zeta +1\vert ^{1/2}\Vert (B_\varepsilon -\zeta ^*Q_0^\varepsilon )^{-1}\Vert _{L_2\rightarrow L_2}\leqslant \Vert f\Vert ^2_{L_\infty}(2+c _\flat )^{1/2}\varrho (\zeta )^{1/2}.
\end{equation}
By analogy with \eqref{6.22a}, we obtain
\begin{equation}
\label{9.10a}
\Vert \mathbf{D}(B_\varepsilon -\zeta ^*Q_0^\varepsilon )^{-1}\Vert _{L_2\rightarrow L_2}\leqslant
c_*^{-1/2} \Vert f\Vert _{L_\infty}
\sup \limits _{\nu \geqslant c_\flat}\frac{\nu ^{1/2}}{\vert \nu-\zeta ^*\vert}.
\end{equation}
A calculation shows that 
\begin{equation}
\label{VIIa}
\sup \limits _{\nu \geqslant c_\flat}\frac{\nu}{\vert \nu -\zeta ^* \vert^2}
\leqslant \begin{cases}
(c_\flat +1)c(\psi)^2\vert \zeta -c_\flat \vert ^{-2}, &\vert \zeta -c_\flat \vert <1,\\
(c_\flat +1)c(\psi)^2\vert \zeta -c_\flat \vert ^{-1},&\vert \zeta -c_\flat\vert \geqslant 1.
\end{cases}
\end{equation}
Using \eqref{V} and the estimate
$\vert \zeta +1\vert \vert \zeta -c_\flat \vert ^{-1}\leqslant 2+c_\flat$ for $\vert \zeta -c_\flat \vert \geqslant 1$,
we see that
\begin{equation}
\label{VIII}
\vert \zeta +1\vert \sup\limits _{\nu \geqslant c_\flat }\frac{\nu}{\vert \nu -\zeta ^*\vert ^2}\leqslant (c_\flat +2)(c_\flat +1)\varrho (\zeta ).
\end{equation}
By \eqref{9.10a} and \eqref{VIII},
\begin{equation}
\label{IX}
\begin{split}
\vert &\zeta +1\vert ^{1/2}\Vert \mathbf{D}(B_\varepsilon -\zeta ^*Q_0^\varepsilon )^{-1}\Vert _{L_2\rightarrow L_2}\\
&\leqslant c_*^{-1/2} \Vert f\Vert _{L_\infty} (c_\flat +2)^{1/2}(c_\flat +1)^{1/2}\varrho (\zeta )^{1/2}.
\end{split}
\end{equation}
Relations \eqref{III}, \eqref{VII}, and \eqref{IX} imply that
\begin{equation}
\label{X}
\vert \zeta +1\vert ^{1/2}\Vert (B_\varepsilon -\zeta Q_0^\varepsilon )^{-1}\Vert _{H^{-1}\rightarrow L_2}
\leqslant \mathfrak{C}_{9}\varrho (\zeta )^{1/2},
\end{equation}
where
\begin{equation*}
\mathfrak{C}_{9}=\Vert f\Vert ^2_{L_\infty}(2+c_\flat )^{1/2}
+c_*^{-1/2} \Vert f\Vert _{L_\infty}(c_\flat +2)^{1/2}(c_\flat +1)^{1/2}.
\end{equation*}

Similarly to \eqref{IV} and \eqref{9.10a}, using \eqref{b^0[u,u]>=} and \eqref{f0_estimates}, we obtain
\begin{equation}
\label{8.18}
\Vert (B^0 -\zeta \overline{Q_0} )^{-1}\Vert _{L_2\rightarrow L_2}\leqslant
\Vert f\Vert ^2_{L_\infty}c(\psi)\vert \zeta -c_\flat\vert ^{-1},
\end{equation}
\begin{equation}
\label{8.18a}
\Vert {\mathbf D} (B^0 -\zeta \overline{Q_0} )^{-1}\Vert _{L_2\rightarrow L_2}\leqslant
c_*^{-1/2} \Vert f\Vert _{L_\infty}
\sup \limits _{\nu \geqslant c_\flat}\frac{\nu ^{1/2}}{\vert \nu-\zeta ^*\vert}.
\end{equation}
Together with \eqref{V}, \eqref{VI}, and \eqref{VIII} this yields 
\begin{equation}
\label{8.18b}
|\zeta+1|^{1/2} \Vert  (B^0 -\zeta \overline{Q_0} )^{-1}\Vert _{L_2\rightarrow H^1}\leqslant
{\mathfrak C}_9 \rho(\zeta)^{1/2}.
\end{equation}

Combining \eqref{lemma Q 0 eps -Q}, \eqref{I}, \eqref{II}, \eqref{X}, and \eqref{8.18b},
we arrive at \eqref{Th 9.1 1} with
$C_{11}=C_4(c_\flat+2)^2\Vert f\Vert ^2_{L_\infty}\Vert f^{-1}\Vert ^2 _{L_\infty}+C_{Q_0}\mathfrak{C}_{9}^2$.

To prove \eqref{Th 9.1 2}, we rely on the already proved estimate \eqref{Th 9.1 1}:
\begin{equation}
\label{XIII}
\begin{split}
\Vert &(B_\varepsilon -\zeta Q_0^\varepsilon )^{-1}-(B^0-\zeta \overline{Q_0})^{-1}-\varepsilon K(\varepsilon ;\zeta )\Vert _{L_2\rightarrow L_2}\\
&\leqslant C_{11}\varrho(\zeta )\varepsilon +\varepsilon \Vert K(\varepsilon ;\zeta )\Vert _{L_2\rightarrow L_2}.
\end{split}
\end{equation}
By \eqref{<b^*b<}, \eqref{tildeLambda S_eps}, and \eqref{Lambda S_eps <=}, the operator \eqref{K(eps;zeta)} satisfies
\begin{equation}
\label{XIV}
\begin{split}
\Vert K(\varepsilon ;\zeta )\Vert _{L_2\rightarrow L_2}&\leqslant M_1\alpha _1^{1/2}
\Vert \mathbf{D}(B^0-\zeta \overline{Q_0})^{-1}\Vert _{L_2\rightarrow L_2}\\
&+\widetilde{M}_1\Vert (B^0-\zeta \overline{Q_0})^{-1}\Vert _{L_2\rightarrow L_2}.
\end{split}
\end{equation}
Taking \eqref{VIIa} and \eqref{8.18a} into account, we have
\begin{equation}
\label{8.22a}
\Vert \mathbf{D}(B^0-\zeta \overline{Q_0})^{-1}\Vert _{L_2\rightarrow L_2}
\leqslant {c}_*^{-1/2}  (c_\flat +1)^{1/2} \Vert f\Vert _{L_\infty} \rho(\zeta)^{1/2}.
\end{equation}
By \eqref{8.18},  \eqref{XIV}, and \eqref{8.22a},
\begin{equation}
\label{XV}
\Vert K(\varepsilon ;\zeta )\Vert _{L_2\rightarrow L_2}
\leqslant \mathfrak{C}_{10}\varrho (\zeta )^{1/2},
\end{equation}
where $\mathfrak{C}_{10}= M_1 \alpha_1^{1/2} {c}_*^{-1/2}  (c_\flat +1)^{1/2} \Vert f\Vert _{L_\infty}+ \widetilde{M}_1 \Vert f\Vert _{L_\infty}^2$.
Combining \eqref{XIII} and \eqref{XV}, and noting that $\varrho (\zeta )^{1/2}\leqslant \varrho (\zeta )$,
we arrive at estimate \eqref{Th 9.1 2} with $C_{12}=C_{11}+\mathfrak{C}_{10}$.

In order to prove \eqref{Th 9.1 3}, we write down the analog of \eqref{tozd corrector} for $\vartheta =1$
(at the points $\zeta $ and $1$) and apply the operator $B_\varepsilon ^{1/2}$
to both sides of the corresponding identity:
\begin{equation}
\label{tozd corrector_S9}
\begin{split}
&B_\varepsilon^{1/2} \bigl((B_\varepsilon - {\zeta} Q_0^\varepsilon )^{-1}-(B^0 -{\zeta} \overline{Q_0})^{-1}-
\varepsilon K(\varepsilon  ;{\zeta})\bigr)\\
&= B_\varepsilon^{1/2} (B_\varepsilon - {\zeta} Q_0^\varepsilon )^{-1}(B_\varepsilon+ Q_0^\varepsilon )
 \left( (B_\varepsilon +Q_0^\varepsilon )^{-1}-(B^0+ \overline{Q_0})^{-1}-\varepsilon K(\varepsilon  ;-1 )\right)\\
&\times (B^0 + \overline{Q_0})(B^0 - {\zeta} \overline{Q_0})^{-1}
+\varepsilon (1 +{\zeta} ) B_\varepsilon^{1/2} (B_\varepsilon - {\zeta} Q_0^\varepsilon )^{-1}Q_0^\varepsilon K(\varepsilon ;{\zeta} )\\
&+(1 +{\zeta})B_\varepsilon^{1/2} (B_\varepsilon - {\zeta} Q_0^\varepsilon )^{-1}(Q_0^\varepsilon -\overline{Q_0})
(B^0-{\zeta} \overline{Q_0})^{-1}.
\end{split}
\end{equation}
Denote the consecutive summands in the right-hand side by ${\mathcal I}_j(\varepsilon;\zeta)$,\break $j=1,2,3$.

Using the analogs of \eqref{Res B eps and tilde B eps}, \eqref{rim 1}, and \eqref{rim 2} with $\vartheta =1$,
we see that
\begin{equation}
\label{9.20b}
\begin{split}
\Vert &B_\varepsilon ^{1/2}(B_\varepsilon -\zeta Q_0^\varepsilon )^{-1}(B_\varepsilon +Q_0^\varepsilon )\mathbf{w}\Vert _{L_2(\mathbb{R}^d)}\\
&=\Vert \widetilde{B}_\varepsilon ^{1/2}(\widetilde{B}_\varepsilon -\zeta I)^{-1}(\widetilde{B}_\varepsilon +I)(f^\varepsilon )^{-1}\mathbf{w}\Vert _{L_2(\mathbb{R}^d)}\\
&\leqslant \Vert (\widetilde{B}_\varepsilon -\zeta I)^{-1}(\widetilde{B}_\varepsilon +I)\Vert _{L_2(\mathbb{R}^d)\rightarrow L_2(\mathbb{R}^d)}\Vert B_\varepsilon ^{1/2}\mathbf{w}\Vert _{L_2(\mathbb{R}^d)}
\end{split}
\end{equation}
for any $\mathbf{w}\in H^1(\mathbb{R}^d;\mathbb{C}^n)$.
The first term in the right-hand side of~\eqref{tozd corrector_S9}
is estimated with the help of \eqref{9.20b} and the upper estimate~\eqref{3/10a b_eps(theta)>=}
with $\vartheta=1$:
\begin{equation}
\label{9.20_star}
\begin{split}
\| {\mathcal I}_1&(\varepsilon;\zeta)\|_{L_2 \to L_2}
\leqslant c_6^{1/2}\Vert (\widetilde{B}_\varepsilon -\zeta I)^{-1}(\widetilde{B}_\varepsilon +I)\Vert _{L_2\rightarrow L_2}
\\
&\times
\Vert (B_\varepsilon +Q_0^\varepsilon )^{-1}-(B^0+\overline{Q_0})^{-1}-\varepsilon K(\varepsilon ;-1)\Vert _{L_2\rightarrow H^1}\\
&\times \Vert (B^0+\overline{Q_0})(B^0-\zeta \overline{Q_0})^{-1}\Vert _{L_2\rightarrow L_2}
\leqslant {\mathfrak C}_{11} \varrho (\zeta)\varepsilon,
\end{split}
\end{equation}
where ${\mathfrak C}_{11} = c_6^{1/2}(C_5+C_6)(c_\flat +2)^2\Vert f\Vert _{L_\infty}\Vert f^{-1}\Vert _{L_\infty}$.
In the last passage, we have used Corollary~\ref{Corollary main result L2->H1} for $\zeta =-1$,
the analog of \eqref{7.39_star}, and \eqref{II}.

Now we estimate the second term in the right-hand side of \eqref{tozd corrector_S9}:
\begin{equation}
\label{I2}
\| {\mathcal I}_2(\varepsilon;\zeta)\|_{L_2\to L_2}
\leqslant \varepsilon \vert \zeta +1\vert \Vert B_\varepsilon ^{1/2}(B_\varepsilon -\zeta Q_0^\varepsilon )^{-1}
\Vert _{L_2\rightarrow L_2}\Vert f^{-1}\Vert ^2_{L_\infty}\Vert K(\varepsilon ;\zeta )\Vert _{L_2\rightarrow L_2}.
\end{equation}
According to \eqref{6.10a} and \eqref{Res B eps and tilde B eps} (with $\vartheta =1$), we have
\begin{equation*}
\Vert B_\varepsilon ^{1/2}(B_\varepsilon -\zeta Q_0^\varepsilon )^{-1}\Vert _{L_2\rightarrow L_2}\leqslant \Vert f\Vert _{L_\infty}\sup\limits _{\nu \geqslant c_\flat}\frac{\nu ^{1/2}}{\vert \nu -\zeta\vert }.
\end{equation*}
By \eqref{VIII}, this implies that
\begin{equation}
\label{XVII}
\vert \zeta +1\vert ^{1/2}\Vert B_\varepsilon ^{1/2}(B_\varepsilon -\zeta Q_0^\varepsilon )^{-1}\Vert _{L_2\rightarrow L_2}
\leqslant (c_\flat +2)^{1/2}(c_\flat +1)^{1/2}\Vert f\Vert _{L_\infty}\varrho (\zeta )^{1/2}.
\end{equation}
From \eqref{8.18b} and \eqref{XIV} it follows that
\begin{equation}
\label{XVIII}
\vert \zeta +1\vert ^{1/2}\Vert K(\varepsilon ;\zeta )\Vert _{L_2\rightarrow L_2}\leqslant
( \alpha _1^{1/2}M_1 + \widetilde{M}_1) \mathfrak{C}_{9}\varrho(\zeta)^{1/2}.
\end{equation}
Relations \eqref{I2}--\eqref{XVIII} imply that
\begin{equation}
\label{I22}
\|{\mathcal I}_2(\varepsilon;\zeta)\|_{L_2 \to L_2} \leqslant {\mathfrak C}_{12} \rho(\zeta) \varepsilon,
\end{equation}
where ${\mathfrak C}_{12} = (c_\flat +2)^{1/2}(c_\flat +1)^{1/2}\Vert f\Vert _{L_\infty}\Vert f^{-1}\Vert^2 _{L_\infty}
(\alpha _1^{1/2}M_1 +\widetilde{M}_1) \mathfrak{C}_{9}$.

For the third term in the right-hand side of \eqref{tozd corrector_S9}, we have
\begin{equation}
\label{I3}
\begin{split}
\|{\mathcal I}_3(\varepsilon;\zeta)\|_{L_2 \to L_2}
\leqslant&
\vert \zeta +1\vert \Vert B_\varepsilon ^{1/2}(B_\varepsilon -\zeta Q_0^\varepsilon )^{-1}\Vert _{H^{-1}\rightarrow L_2}
\\
&\times \Vert [Q_0^\varepsilon -\overline{Q_0}]\Vert _{H^1\rightarrow H^{-1}}
 \Vert (B^0-\zeta \overline{Q_0})^{-1}\Vert _{L_2\rightarrow H^1}.
\end{split}
\end{equation}
By the duality argument (cf. \eqref{rim 5}),
\begin{equation}
\label{XIX}
\Vert B_\varepsilon ^{1/2}(B_\varepsilon -\zeta Q_0^\varepsilon )^{-1}\Vert _{H^{-1}\rightarrow L_2}
=\Vert f^\varepsilon \widetilde{B}_\varepsilon ^{1/2}(\widetilde{B}_\varepsilon -\zeta ^* I)^{-1}\Vert _{L_2\rightarrow H^1}.
\end{equation}
Similarly to \eqref{rim 6}, using \eqref{II}, we obtain
\begin{equation}
\label{XX}
\begin{split}
\Vert &\mathbf{D}[f^\varepsilon ]\widetilde{B}_\varepsilon ^{1/2}(\widetilde{B}_\varepsilon -\zeta ^* I)^{-1}\Vert _{L_2\rightarrow L_2}\\
&\leqslant
c_*^{-1/2}\Vert \widetilde{B}_\varepsilon (\widetilde{B}_\varepsilon-\zeta ^* I)^{-1}\Vert _{L_2\rightarrow L_2}
\leqslant c_*^{-1/2}(c_\flat +2)\varrho (\zeta )^{1/2}.
\end{split}
\end{equation}
Next, taking \eqref{VIIa} into account, we have
\begin{equation}
\label{XXI}
\Vert [f^\varepsilon ]\widetilde{B}_\varepsilon ^{1/2}(\widetilde{B}_\varepsilon -\zeta ^* I)^{-1}\Vert _{L_2\rightarrow L_2}
\leqslant \Vert f \Vert _{L_\infty}(c_\flat +1)^{1/2}\varrho (\zeta )^{1/2}.
\end{equation}
From \eqref{XIX}--\eqref{XXI} it follows that
\begin{equation}
\label{XXII}
\Vert B_\varepsilon ^{1/2}(B_\varepsilon -\zeta Q_0^\varepsilon )^{-1}\Vert _{H^{-1}\rightarrow L_2}
\leqslant \mathfrak{C}_{13}\varrho (\zeta )^{1/2},
\end{equation}
where $\mathfrak{C}_{13}=\Vert f\Vert _{L_\infty}(c_\flat+1)^{1/2} +c_*^{-1/2}(c_\flat +2)$.
Relations \eqref{lemma Q 0 eps -Q}, \eqref{8.18b}, \eqref{I3}, and \eqref{XXII} yield
\begin{equation}
\label{I32}
\| {\mathcal I}_3(\varepsilon;\zeta) \|_{L_2 \to L_2} \leqslant C_{Q_0}{\mathfrak C}_{9}{\mathfrak C}_{13} |\zeta+1|^{1/2} \rho(\zeta) \varepsilon.
\end{equation}

Combining \eqref{tozd corrector_S9}, \eqref{9.20_star}, \eqref{I22}, \eqref{I32}, and the lower estimate \eqref{3/10a b_eps(theta)>=}
(with $\vartheta =1$), we arrive at \eqref{Th 9.1 3} with
$C_{13}= c_*^{-1/2} ({\mathfrak C}_{11} +{\mathfrak C}_{12})$,
$C_{14}=c_*^{-1/2} C_{Q_0}\mathfrak{C}_{9}\mathfrak{C}_{13}$.

It remains to check \eqref{Th 9.1 4}. From \eqref{Th 9.1 3} and \eqref{<b^*b<} it follows that
\begin{equation}
\label{9.28}
\begin{split}
\Vert &g^\varepsilon b(\mathbf{D})(B_\varepsilon -\zeta Q_0^\varepsilon )^{-1}\\
&-g^\varepsilon b(\mathbf{D})
\bigl(
I+\varepsilon [\Lambda ^\varepsilon ] S_\varepsilon b(\mathbf{D})+\varepsilon [\widetilde{\Lambda}^\varepsilon ]S_\varepsilon
\bigr)
(B^0-\zeta\overline{Q_0})^{-1}\Vert _{L_2\rightarrow L_2}\\
&\leqslant \alpha _1^{1/2}\Vert g\Vert _{L_\infty}(C_{13}+\vert \zeta +1\vert ^{1/2}C_{14})\varrho (\zeta )\varepsilon .
\end{split}
\end{equation}
By analogy with \eqref{d-vu th 4/4 1}, \eqref{7.32}, and \eqref{rim 11}, we obtain
\begin{equation}
\label{8.38}
\begin{split}
\Vert &g^\varepsilon b(\mathbf{D})
\bigl(
I+\varepsilon [\Lambda ^\varepsilon ]S_\varepsilon b(\mathbf{D})+\varepsilon [\widetilde{\Lambda}^\varepsilon ]S_\varepsilon
\bigr)
(B^0-\zeta\overline{Q_0})^{-1}
-G(\varepsilon ;\zeta )\Vert _{L_2\rightarrow L_2}\\
& \leqslant \varepsilon \|g\|_{L_\infty} \alpha_1^{1/2}((\alpha_1 d)^{1/2} M_1+r_1)
\Vert {\mathbf D}^2 (B^0-\zeta \overline{Q_0})^{-1}\Vert _{L_2\rightarrow L_2}
\\
&+ \varepsilon \|g\|_{L_\infty} (\alpha_1 d)^{1/2} \widetilde{M}_1
\Vert {\mathbf D} (B^0-\zeta \overline{Q_0})^{-1}\Vert _{L_2\rightarrow L_2}.
\end{split}
\end{equation}
Similarly to \eqref{rim 8} and \eqref{7.39_star}, taking \eqref{II} into account, we have
\begin{equation}
\label{8.38a}
\begin{split}
\Vert {\mathbf D}^2 (B^0-\zeta \overline{Q_0})^{-1}\Vert _{L_2\rightarrow L_2}
&\leqslant c_*^{-1} \|f\|_{L_\infty}  \|f^{-1}\|_{L_\infty} \sup \limits _{\nu \geqslant c_\flat} \frac{\nu}{\vert \nu -\zeta \vert }
\\
&\leqslant c_*^{-1} \|f\|_{L_\infty}  \|f^{-1}\|_{L_\infty} ( c_\flat +2) \rho(\zeta)^{1/2}.
\end{split}
\end{equation}
Together with \eqref{8.22a} and \eqref{8.38} this implies
\begin{equation}
\label{8.38b}
\begin{split}
\Vert &g^\varepsilon b(\mathbf{D})
\bigl(
I+\varepsilon [\Lambda ^\varepsilon ]S_\varepsilon b(\mathbf{D})+\varepsilon [\widetilde{\Lambda}^\varepsilon ]S_\varepsilon
\bigr)
(B^0-\zeta\overline{Q_0})^{-1}
-G(\varepsilon ;\zeta )\Vert _{L_2\rightarrow L_2}\\
& \leqslant {\mathfrak C}_{14}\varepsilon \rho(\zeta)^{1/2},
\end{split}
\end{equation}
where ${\mathfrak C}_{14}=\Vert g\Vert _{L_\infty}\Vert f\Vert _{L_\infty} \alpha_1^{1/2}
\bigl( ( (\alpha _1 d)^{1/2}M_1+r_1) c_*^{-1}  \Vert f^{-1}\Vert _{L_\infty}(c_\flat +2)
+ d^{1/2} \widetilde{M}_1   c_*^{-1/2}  (c_\flat +1)^{1/2} \bigr)$.
Now, relations \eqref{9.28} and \eqref{8.38b} imply estimate \eqref{Th 9.1 4}
with
$C_{15}=\alpha _1^{1/2}\Vert g\Vert _{L_\infty}C_{13} + {\mathfrak C}_{14} $ and
$C_{16}=\alpha _1^{1/2}\Vert g\Vert _{L_\infty}C_{14}$.
\end{proof}

Note that, if $Q_0$ is a constant matrix, then $Q_0^\varepsilon = \overline{Q_0}$
and the third summand in the right-hand side of \eqref{tozd corrector_S9} is equal to zero:
$\mathcal{I}_3(\varepsilon ;\zeta )=0$. Therefore, from the proof of Theorem~\ref{Theorem another approximation}
we derive the following.

\begin{remark}
If the assumptions of Theorem \textnormal{\ref{Theorem another approximation}} are satisfied and
$Q_0$ is a constant matrix, then estimates~\eqref{Th 9.1 3}--\eqref{8.5a} are true with~$C_{14}=C_{16}=0$.
It means that in estimates \eqref{Th 9.1 3}--\eqref{8.5a} there are no terms
containing $\vert \zeta +1\vert ^{1/2}$.
\end{remark}

\subsection{Removal of $S_\varepsilon$}
Now we distinguish the cases where the smoothing operator $S_\varepsilon$ in the corrector
can be removed.

\begin{theorem}
\label{Theorem no S_eps another approximation}
Suppose that the assumptions of Theorem~\textnormal{\ref{Theorem another approximation}} are satisfied.

\noindent
$1^\circ $. Suppose that Condition~\textnormal{\ref{Condition Lambda in L_infty}} is satisfied.
Let $G_1(\varepsilon ;\zeta )$ be the operator defined by \eqref{G_1(eps,zeta)}.
Then for $0<\varepsilon \leqslant 1$ we have
\begin{align*}
\begin{split}
\Vert &(B_\varepsilon -\zeta Q_0^\varepsilon )^{-1}-
\bigl(
I+\varepsilon [\Lambda ^\varepsilon ]b(\mathbf{D})+\varepsilon [\widetilde{\Lambda}^\varepsilon ]S_\varepsilon
\bigr)
(B^0-\zeta \overline{Q_0})^{-1}\Vert _{L_2(\mathbb{R}^d)\rightarrow L_2(\mathbb{R}^d)}\\
&\leqslant C_{17}'\varrho (\zeta )\varepsilon ,
\end{split}\\
\begin{split}
\Vert & \mathbf{D}\left((B_\varepsilon -\zeta Q_0^\varepsilon )^{-1}-
\bigl(
I+\varepsilon [\Lambda ^\varepsilon ]b(\mathbf{D})+\varepsilon [\widetilde{\Lambda}^\varepsilon ]S_\varepsilon
\bigr)
(B^0-\zeta \overline{Q_0})^{-1}\right)\Vert _{L_2(\mathbb{R}^d)\rightarrow L_2(\mathbb{R}^d)}\\
&\leqslant (C_{18}'+\vert \zeta +1\vert ^{1/2}C_{19}')\varrho (\zeta )\varepsilon ,
\end{split}\\
\Vert & g^\varepsilon b(\mathbf{D})(B_\varepsilon -\zeta Q_0^\varepsilon )^{-1}-G_1(\varepsilon;\zeta )\Vert _{L_2(\mathbb{R}^d)\rightarrow L_2(\mathbb{R}^d)}
\leqslant (C_{20}'+\vert \zeta +1\vert ^{1/2}C_{21}')\varrho (\zeta )\varepsilon .
\end{align*}
The constants $C_{17}'$, $C_{18}'$, $C_{19}'$, $C_{20}'$, and $C_{21}'$ are controlled in terms of the problem data
\eqref{problem data}, $c_\flat$, and $\Vert \Lambda \Vert _{L_\infty}$.

\noindent
$2^\circ $. Suppose that Condition~\textnormal{\ref{Condition on Lambda-tilda}} is satisfied.
Let $G_2(\varepsilon ;\zeta )$ be the operator given by~\eqref{G_2(eps,zeta)}.
Then for $0<\varepsilon \leqslant 1$ we have
\begin{align*}
\begin{split}
\Vert &(B_\varepsilon -\zeta Q_0^\varepsilon )^{-1}-
\bigl(
I+\varepsilon [\Lambda ^\varepsilon ] b(\mathbf{D}) S_\varepsilon +\varepsilon [\widetilde{\Lambda}^\varepsilon ]
\bigr)
(B^0-\zeta \overline{Q_0})^{-1}\Vert _{L_2(\mathbb{R}^d)\rightarrow L_2(\mathbb{R}^d)}\\
&\leqslant C_{17}''\varrho (\zeta )\varepsilon ,
\end{split}\\
\begin{split}
\Vert & \mathbf{D}\left((B_\varepsilon -\zeta Q_0^\varepsilon )^{-1}-
\bigl(
I+\varepsilon [\Lambda ^\varepsilon ] b(\mathbf{D}) S_\varepsilon +\varepsilon [\widetilde{\Lambda}^\varepsilon ]
\bigr)
(B^0-\zeta \overline{Q_0})^{-1}\right)\Vert _{L_2(\mathbb{R}^d)\rightarrow L_2(\mathbb{R}^d)}\\
&\leqslant (C_{18}''+\vert \zeta +1\vert ^{1/2}C_{19}'')\varrho (\zeta )\varepsilon ,
\end{split}\\
\Vert & g^\varepsilon b(\mathbf{D})(B_\varepsilon -\zeta Q_0^\varepsilon )^{-1}-G_2(\varepsilon;\zeta )\Vert _{L_2(\mathbb{R}^d)\rightarrow L_2(\mathbb{R}^d)}
\leqslant (C_{20}''+\vert \zeta +1\vert ^{1/2}C_{21}'')\varrho (\zeta )\varepsilon .
\end{align*}
The constants $C_{17}''$, $C_{18}''$, $C_{19}''$, $C_{20}''$, and $C_{21}''$ depend only on the initial data
\eqref{problem data}, $c_\flat$, $p$, and $\Vert \widetilde{ \Lambda }\Vert _{L_p(\Omega)}$.

\noindent
$3^\circ $. Suppose that Conditions \textnormal{\ref{Condition Lambda in L_infty}} and \textnormal{\ref{Condition on Lambda-tilda}}
are satisfied. Let $G_3(\varepsilon ;\zeta )$ be the operator given by \eqref{G_3(eps,zeta)}.
Then for $0<\varepsilon \leqslant 1$ we have
\begin{align*}
\begin{split}
\Vert &(B_\varepsilon -\zeta Q_0^\varepsilon )^{-1}-
\bigl(
I+\varepsilon [\Lambda ^\varepsilon ]b(\mathbf{D})+\varepsilon [\widetilde{\Lambda}^\varepsilon ]
\bigr)
(B^0-\zeta \overline{Q_0})^{-1}\Vert _{L_2(\mathbb{R}^d)\rightarrow L_2(\mathbb{R}^d)}\\
&\leqslant C_{17}\varrho (\zeta )\varepsilon ,
\end{split}\\
\begin{split}
\Vert & \mathbf{D}\left((B_\varepsilon -\zeta Q_0^\varepsilon )^{-1}-
\bigl(
I+\varepsilon [\Lambda ^\varepsilon ]b(\mathbf{D})+\varepsilon [\widetilde{\Lambda}^\varepsilon ]
\bigr)
(B^0-\zeta \overline{Q_0})^{-1}\right)\Vert _{L_2(\mathbb{R}^d)\rightarrow L_2(\mathbb{R}^d)}\\
&\leqslant (C_{18}+\vert \zeta +1\vert ^{1/2}C_{19})\varrho (\zeta )\varepsilon ,
\end{split}\\
\Vert & g^\varepsilon b(\mathbf{D})(B_\varepsilon -\zeta Q_0^\varepsilon )^{-1}-G_3(\varepsilon;\zeta )\Vert _{L_2(\mathbb{R}^d)\rightarrow L_2(\mathbb{R}^d)}
\leqslant (C_{20}+\vert \zeta +1\vert ^{1/2}C_{21})\varrho (\zeta )\varepsilon .
\end{align*}
The constants $C_{17}$, $C_{18}$, $C_{19}$, $C_{20}$, and $C_{21}$ are controlled in terms of the initial data~\eqref{problem data},
$c_\flat$, $p$, $\Vert \Lambda \Vert _{L_\infty}$, and $\Vert \widetilde{\Lambda}\Vert _{L_p(\Omega)}$.
\end{theorem}

\begin{remark}
If the assumptions of Theorem~\textnormal{\ref{Theorem no S_eps another approximation}} are satisfied and $Q_0$ is a~constant matrix,
then estimates of Theorem~\textnormal{\ref{Theorem no S_eps another approximation}} are valid with  $C_{19}'=C_{21}'=C_{19}''=C_{21}''=C_{19}=C_{21}=0$.
\end{remark}

\begin{proof}
Approximations with corrector for the operator 
\hbox{$(B_\varepsilon -\zeta Q_0^\varepsilon )^{-1}$} are deduced from Theorem~\ref{Theorem another approximation},
by analogy with the proof of Theorem~\ref{Theorem main result H1 no S_eps}.
The difference is that, instead of \eqref{rim 8a}, \eqref{rim 9}, and \eqref{Lm 7/2 e}
we use \eqref{8.38a}, \eqref{8.22a}, and \eqref{8.18}, respectively.

The statements concerning the flux are derived from approximations for the operator
$(B_\varepsilon -\zeta Q_0^\varepsilon )^{-1}$, by analogy with the proof of
Theorem~\ref{Theorem main result fluxes no S_eps}. The difference is that, instead of
\eqref{rim 8a} and \eqref{rim 9} we apply \eqref{8.38a} and \eqref{8.22a}, respectively.
\end{proof}

\subsection{Special cases} Similarly to Proposition~\ref{Proposition K=0},
with the help of Theorem~\ref{Theorem another approximation} we distinguish the case where the corrector is equal to zero.

\begin{proposition}
Suppose that the assumptions of Theorem \textnormal{\ref{Theorem another approximation}} are satisfied.
Suppose also that relations \eqref{overline-g} and \eqref{Djaj=0} are satisfied.
Then for $0<\varepsilon \leqslant 1$ we have
\begin{equation*}
\begin{split}
\Vert &\mathbf{D}
\left(
(B_\varepsilon -\zeta Q_0^\varepsilon )^{-1}-(B^0-\zeta \overline{Q_0})^{-1}
\right)\Vert _{L_2(\mathbb{R}^d)\rightarrow L_2(\mathbb{R}^d)}\\
&\leqslant (C_{13}+\vert \zeta +1\vert ^{1/2}C_{14})\varrho (\zeta )\varepsilon .
\end{split}
\end{equation*}
\end{proposition}

Similarly to Proposition \ref{Proposition non-general case}, we deduce the following statement from
Theorem~\ref{Theorem no S_eps another approximation}($3^\circ$).

\begin{proposition}
Suppose that the assumptions of Theorem \textnormal{\ref{Theorem another approximation}} are satisfied.
Suppose also that relations \eqref{underline-g} and \eqref{Djaj=0} are satisfied. Then for
$0<\varepsilon \leqslant 1$ we have
\begin{equation*}
\begin{split}
\Vert & g^\varepsilon b(\mathbf{D})(B_\varepsilon -\zeta Q_0^\varepsilon )^{-1}-g^0b(\mathbf{D})(B^0-\zeta \overline{Q_0})^{-1}\Vert _{L_2(\mathbb{R}^d)\rightarrow L_2(\mathbb{R}^d)}\\
&\leqslant (C_{20}+\vert \zeta +1\vert ^{1/2}C_{21})\varrho (\zeta )\varepsilon .
\end{split}
\end{equation*}
\end{proposition}

\section{Application of the general results}
\label{Section Applications}

In this section, we consider examples that were studied before in \cite{SuAA} and \cite{SuAA14}.

\subsection{The scalar elliptic operator}
\label{Subsection applications 1 scalar}
We consider the case where $n=1$, $m=d$, $b(\mathbf{D})=\mathbf{D}$, and $g(\mathbf{x})$ is a $\Gamma$-periodic symmetric
$(d\times d)$-matrix-valued function with {\it real entries} such that $g(\mathbf{x})>0$ and $g, g^{-1}\in L_\infty$.
Then, obviously, $\alpha _0=\alpha _1=1$ (see \eqref{<b^*b<}). The operator $\mathcal{A}_\varepsilon$ takes the form
$\mathcal{A}_\varepsilon =-\mathrm{div}\,g^\varepsilon (\mathbf{x})\nabla $.

Next, let $\mathbf{A}(\mathbf{x})=\mathrm{col}\,\lbrace A_1(\mathbf{x}),\dots ,A_d(\mathbf{x})\rbrace$,
where $A_j(\mathbf{x})$, $j=1,\dots,d$, are $\Gamma$-periodic real-valued functions such that
\begin{equation}
\label{A_j cond}
\begin{split}
A_j\in L_\rho (\Omega),\quad\rho=2\ \text{for}\ d=1,\quad \rho >d\ \text{for}\ d\geqslant 2;\quad j=1,\dots,d.
\end{split}
\end{equation}
Let $ v(\mathbf{x})$ and $\mathcal{V} (\mathbf{x})$ be real-valued $\Gamma$-periodic functions such that
\begin{equation}
\label{v, V cond}
\begin{split}
v,\mathcal{V}\in L_s(\Omega),\quad \int _{\Omega}v(\mathbf{x})\,d\mathbf{x}=0,\quad
s=1\ \text{for}\  d=1,\quad s>d/2\ \text{for}\ d\geqslant 2.
\end{split}
\end{equation}

In $L_2(\mathbb{R}^d)$, we consider the operator $\mathcal{B}_\varepsilon$ given formally  by
the differential expression
\begin{equation}
\label{B_eps from sec. 9.1}
\mathcal{B}_\varepsilon =(\mathbf{D}-\mathbf{A}^\varepsilon (\mathbf{x}))^*g^\varepsilon (\mathbf{x})(\mathbf{D}-\mathbf{A}^\varepsilon (\mathbf{x}))+\varepsilon ^{-1}v^\varepsilon (\mathbf{x})+\mathcal{V}^\varepsilon (\mathbf{x}).
\end{equation}
Precisely, $\mathcal{B}_\varepsilon$ is the operator generated by the quadratic form
\begin{equation*}
\begin{split}
\mathfrak{b}_\varepsilon [u,u]=\int _{\mathbb{R}^d}\left( \langle g^\varepsilon (\mathbf{D}-\mathbf{A}^\varepsilon )u, (\mathbf{D}-\mathbf{A}^\varepsilon )u\rangle
+(\varepsilon ^{-1}v^\varepsilon +\mathcal{V}^\varepsilon)\vert u\vert ^2\right)\,d\mathbf{x},\\
u\in H^1(\mathbb{R}^d).
\end{split}
\end{equation*}
The operator \eqref{B_eps from sec. 9.1} can be treated as the periodic Schr\"odinger operator with the metric $g^\varepsilon$, the
magnetic potential~$\mathbf{A}^\varepsilon$, and the electric potential \hbox{$\varepsilon ^{-1}v^\varepsilon +\mathcal{V}^\varepsilon$}
containing the singular summand~$\varepsilon ^{-1}v^\varepsilon$.

It is easy to check (see \cite[Subsection~13.1]{SuAA}) that the operator \eqref{B_eps from sec. 9.1}
can be written in the required form~\eqref{B_eps}:
\begin{equation*}
\mathcal{B}_\varepsilon = \mathbf{D}^*g^\varepsilon (\mathbf{x})\mathbf{D}+\sum _{j=1}^d\left( a_j^\varepsilon (\mathbf{x})D_j+D_j(a_j^\varepsilon (\mathbf{x}))^*\right)
+Q^\varepsilon (\mathbf{x}).
\end{equation*}
Here the real-valued function $Q(\mathbf{x})$ is given by
\begin{equation}
\label{Q from sec. 9.1}
Q(\mathbf{x})=\mathcal{V}(\mathbf{x})+\langle g(\mathbf{x})\mathbf{A}(\mathbf{x}),\mathbf{A}(\mathbf{x})\rangle .
\end{equation}
The complex-valued functions $a_j(\mathbf{x})$ are given by
\begin{equation}
\label{a_j from sec. 9.1}
a_j(\mathbf{x})=-\eta _j(\mathbf{x})+i\gamma _j(\mathbf{x}),\quad j=1,\dots ,d,
\end{equation}
where $\eta _j(\mathbf{x})$ are the components of the vector-valued function
$\boldsymbol{\eta}(\mathbf{x})=g(\mathbf{x})\mathbf{A}(\mathbf{x})$, and
$\gamma _j(\mathbf{x})=-\partial _j\Phi (\mathbf{x})$. Here $\Phi(\mathbf{x})$ is the $\Gamma$-periodic solution of the problem
$\bigtriangleup \Phi (\mathbf{x})=v(\mathbf{x})$, $\int _\Omega\Phi (\mathbf{x})\,d\mathbf{x}=0$. Then 
\begin{equation}
\label{v= sec. 9.1}
v(\mathbf{x})=-\sum _{j=1}^d\partial _j \gamma _j(\mathbf{x}).
\end{equation}
It is easy to check that the functions \eqref{a_j from sec. 9.1} satisfy condition \eqref{a_j}
with sui\-tab\-le $\rho '$ depending on $\rho$ and $s$; the norms $\Vert a_j\Vert _{L_{\rho '}(\Omega)}$
are controlled in terms of $\Vert g\Vert _{L_\infty}$, $\Vert \mathbf{A}\Vert _{L_\rho (\Omega )}$, $\Vert v\Vert _{L_s(\Omega )}$, and
 the parameters of the lattice~$\Gamma$. (See \cite[Subsection~13.1]{SuAA} for details).
 The function \eqref{Q from sec. 9.1} satisfies condition \eqref{Q in L_s example} with suitable
 $s'=\min\lbrace s;\rho/2\rbrace$. Thus, now the form $q$ is as in Example \ref{Example where Q in L_s}.

Suppose that $Q_0(\mathbf{x})$ is a positive definite and bounded $\Gamma$-periodic function.
According to \eqref{Beps}, we consider the nonnegative operator
$B_\varepsilon :=\mathcal{B}_\varepsilon +c_5Q_0^\varepsilon$.
Here $c_5=(c_0+c_4)\Vert Q_0^{-1}\Vert _{L_\infty}$,
and the constants $c_0$ and $c_4$ correspond to the operator \eqref{B_eps from sec. 9.1}.
We are interested in the behavior of the operator \hbox{$(B_\varepsilon -\zeta Q_0^\varepsilon )^{-1}$}, where $\zeta \in \mathbb{C}\setminus\mathbb{R}_+$. In the case under consideration, the initial data (see~\eqref{problem data}) reduces to the following set
\begin{equation}
\label{problem data from sec. 9.1}
\begin{split}
&d,\rho ,s ; \Vert g\Vert _{L_\infty}, \Vert g^{-1}\Vert _{L_\infty}, \Vert \mathbf{A}\Vert _{L_\rho (\Omega )},\Vert v\Vert _{L_s(\Omega)},
\Vert \mathcal{V}\Vert _{L_s(\Omega )},\\
&\Vert Q_0\Vert _{L_\infty}, \Vert Q_0^{-1}\Vert _{L_\infty};\  \text{the parameters of the lattice}\, \Gamma .
\end{split}
\end{equation}

Let us find the effective operator. In the case under consideration, the $\Gamma$-periodic solution of problem \eqref{Lambda problem}
is the row
\begin{equation*}
\Lambda(\mathbf{x})=i\Psi (\mathbf{x}),\quad \Psi (\mathbf{x})=(\psi _1 (\mathbf{x}),\dots ,\psi _d (\mathbf{x})),
\end{equation*}
where $\psi _j\in \widetilde{H}^1(\Omega)$ is the solution of the problem
\begin{equation*}
\mathrm{div}\,g(\mathbf{x})\left( \nabla \psi _j(\mathbf{x})+\mathbf{e}_j\right)=0,\quad \int _\Omega \psi _j(\mathbf{x})\,d\mathbf{x}=0.
\end{equation*}
Here $\mathbf{e}_j$, $j=1,\dots,d$, is the standard orthonormal basis in $\mathbb{R}^d$.
Clearly, the functions $\psi _j(\mathbf{x})$ are real-valued, and $\Lambda (\mathbf{x})$ has purely imaginary entries.
By \eqref{tilde g}, $\widetilde{g}({\mathbf x})$ is the $(d\times d)$-matrix-valued function with the columns
   $g(\mathbf{x})(\nabla \psi _j(\mathbf{x})+\mathbf{e}_j)$, $j=1,\dots ,d$.
  According to \eqref{g^0}, the effective matrix is given by $g^0=\vert \Omega \vert ^{-1}\int _\Omega \widetilde{g}(\mathbf{x})\,d\mathbf{x}$. Clearly, the matrices $\widetilde{g}(\mathbf{x})$ and $g^0$ have real entries.

By \eqref{a_j from sec. 9.1} and \eqref{v= sec. 9.1}, the periodic solution of problem \eqref{tildeLambda_problem}
can be represented as $\widetilde{\Lambda}(\mathbf{x})=\widetilde{\Lambda}_1(\mathbf{x})+i\widetilde{\Lambda}_2(\mathbf{x})$,
where the real-valued $\Gamma$-periodic functions $\widetilde{\Lambda}_1(\mathbf{x})$ and $\widetilde{\Lambda}_2(\mathbf{x})$
are the solutions of the following problems:
\begin{align*}
&-\mathrm{div}\,g(\mathbf{x})\nabla \widetilde{\Lambda}_1(\mathbf{x})+v(\mathbf{x})=0,\quad \int _\Omega \widetilde{\Lambda}_1(\mathbf{x})\,d\mathbf{x}=0,\\
&-\mathrm{div}\,g(\mathbf{x})\nabla \widetilde{\Lambda}_2(\mathbf{x})+\mathrm{div}\,g(\mathbf{x})\mathbf{A}(\mathbf{x})=0,\quad \int _\Omega\widetilde{\Lambda}_2(\mathbf{x})\,d\mathbf{x}=0.
\end{align*}
The column  $V$ (see~\eqref{V matr}) can be written as~$V=V_1+iV_2$, where the columns $V_1$ and $V_2$ 
with real entries are given by
\begin{equation}
\label{V1 and V2 for scalar Schrodinger}
\begin{split}
&V_1=\vert \Omega\vert ^{-1}\int  _\Omega (\nabla \Psi (\mathbf{x}))^t g(\mathbf{x})\nabla \widetilde{\Lambda}_2(\mathbf{x})\,d\mathbf{x},\\
&V_2=-\vert \Omega\vert ^{-1}\int _\Omega (\nabla \Psi (\mathbf{x}))^t g(\mathbf{x})\nabla \widetilde{\Lambda}_1(\mathbf{x})\,d\mathbf{x}.
\end{split}
\end{equation}
According to \eqref{W}, the constant $W$ is given by
\begin{equation}
\label{W for scalar Schodinger}
W=\vert \Omega \vert ^{-1}\int_\Omega \left(
 \langle g(\mathbf{x})\nabla \widetilde{\Lambda }_1(\mathbf{x}),\nabla \widetilde{\Lambda }_1(\mathbf{x})\rangle  + \langle g(\mathbf{x})\nabla \widetilde{\Lambda}_2(\mathbf{x}),\nabla \widetilde{\Lambda }_2(\mathbf{x})\rangle
\right) \,d\mathbf{x}.
\end{equation}
The effective operator for $B_\varepsilon$ takes the form
\begin{equation*}
B^0u=-\mathrm{div}\,g^0\nabla u +2i\langle \nabla u, V_1+\overline{\boldsymbol{\eta}}\rangle
+(-W+\overline{Q}+c_5\overline{Q_0})u,\quad u \in H^2(\mathbb{R}^d).
\end{equation*}
In other words,
\begin{equation}
\label{B^0 from sec. 9.1}
B^0=(\mathbf{D}-\mathbf{A}^0)^*g^0(\mathbf{D}-\mathbf{A}^0)+\mathcal{V}^0+c_5\overline{Q_0},
\end{equation}
where
\begin{equation}
\label{A0 and V0 for scalar Scrodinger}
\mathbf{A}^0=(g^0)^{-1}(V_1+\overline{g\mathbf{A}}),\quad
\mathcal{V}^0=\overline{\mathcal{V}}+\overline{\langle g\mathbf{A},\mathbf{A}\rangle}-\langle g^0\mathbf{A}^0,\mathbf{A}^0\rangle -W .
\end{equation}

According to Remark \ref{Remark A with d=1 and real g},
 in the case under consideration  Conditions~\ref{Condition Lambda in L_infty} and \ref{Condition on Lambda-tilda}
 are satisfied, and the norms
$\Vert \Lambda \Vert _{L_\infty}$ and $\Vert \widetilde{\Lambda}\Vert _{L_\infty}$ are controlled
in terms of the initial data~\eqref{problem data from sec. 9.1}. Therefore, it is possible to use
the simpler corrector given by
\begin{equation}
\label{K^0(eps;zeta)}
K^0(\varepsilon ;\zeta):=([\Lambda ^\varepsilon ] \mathbf{D}+[\widetilde{\Lambda}^\varepsilon ])(B^0-\zeta \overline{Q_0})^{-1}
=([\Psi ^\varepsilon ]\nabla +[\widetilde{\Lambda}^\varepsilon ])(B^0-\zeta \overline{Q_0})^{-1}.
\end{equation}
The operator \eqref{G_3(eps,zeta)} takes the form
\begin{equation*}
G_3(\varepsilon;\zeta) =\widetilde{g}^\varepsilon \mathbf{D}(B^0-\zeta \overline{Q_0})^{-1}
+g^\varepsilon(\mathbf{D}\widetilde{\Lambda})^\varepsilon (B^0-\zeta \overline{Q_0})^{-1}.
\end{equation*}
Applying Theorems \ref{Theorem main result L2}, \ref{Theorem main result H1 no S_eps}($3^\circ$), and \ref{Theorem main result fluxes no S_eps}($3^\circ$), we deduce the following result.

\begin{proposition}
\label{Proposition scalar Schr?dinger}
Let $\mathcal{B}_\varepsilon$ be the operator \eqref{B_eps from sec. 9.1}
whose coefficients satisfy the assumptions of Subsection~\textnormal{\ref{Subsection applications 1 scalar}}.
Let $Q_0(\mathbf{x})$ be a $\Gamma$-periodic positive definite and bounded function, and let
$c_5=(c_0+c_4)\Vert Q_0^{-1}\Vert _{L_\infty}$, where the constants $c_0$ and $c_4$ correspond to the operator
\eqref{B_eps from sec. 9.1}. Let $B_\varepsilon =\mathcal{B}_\varepsilon +c_5Q_0^\varepsilon$, and let $B^0$ be the effective operator \eqref{B^0 from sec. 9.1}, whose coefficients are defined by
\eqref{V1 and V2 for scalar Schrodinger}, \eqref{W for scalar Schodinger}, and \eqref{A0 and V0 for scalar Scrodinger}.
Let $\zeta\in\mathbb{C}\setminus \mathbb{R}_+$, $\zeta =\vert \zeta \vert e^{i\phi}$, $0<\phi <2\pi$, and $\vert \zeta \vert \geqslant 1$.
Suppose that $K^0(\varepsilon ;\zeta)$ is defined by \eqref{K^0(eps;zeta)}.
Then for $0<\varepsilon\leqslant 1$ we have
\begin{align}
\label{10.13a}
\Vert &(B_\varepsilon -\zeta Q_0^\varepsilon )^{-1}-(B^0-\zeta \overline{Q_0})^{-1}\Vert _{L_2(\mathbb{R}^d)\rightarrow L_2(\mathbb{R}^d)}\leqslant C_4\varepsilon c(\phi)^2\vert \zeta \vert ^{-1/2},\\
\label{10.13b}
\begin{split}
\Vert &(B_\varepsilon -\zeta Q_0^\varepsilon )^{-1}-(B^0-\zeta \overline{Q_0})^{-1}-\varepsilon K^0(\varepsilon ;\zeta )
\Vert _{L_2(\mathbb{R}^d)\rightarrow H^1(\mathbb{R}^d)}\\
&\leqslant (C_{8}+C_{9}) c(\phi)^2\varepsilon ,
\end{split}
\\
\label{10.13c}
\begin{split}
\Vert & g^\varepsilon \nabla(B_\varepsilon -\zeta Q_0^\varepsilon)^{-1} -\left(\widetilde{g}^\varepsilon\nabla +g^\varepsilon (\nabla \widetilde{\Lambda})^\varepsilon\right) (B^0-\zeta \overline{Q_0})^{-1}\Vert _{L_2(\mathbb{R}^d)\rightarrow L_2(\mathbb{R}^d)}\\
&\leqslant C_{10}c(\phi)^2\varepsilon .
\end{split}
\end{align}
Here $c(\phi)$ is given by~\eqref{c(phi)}.
The constants $C_4$, $C_{8}$, $C_{9}$, and $C_{10}$ depend only on the initial data~\eqref{problem data from sec. 9.1}.
\end{proposition}

In order to obtain ,,another'' approximation of the operator \hbox{$(B_\varepsilon -\zeta Q_0^\varepsilon )^{-1}$},
we use Theorem~\ref{Theorem another approximation} (for the principal term of approximation)
and Theorem~\ref{Theorem no S_eps another approximation}($3^\circ$).

\begin{proposition}
\label{scalar}
 Suppose that the assumptions of Proposition \textnormal{\ref{Proposition scalar Schr?dinger}} are sa\-tis\-fied.
Denote $f(\mathbf{x}) = Q_0(\mathbf{x})^{-1/2}$ and $f_0 = (\overline{Q_0})^{-1/2}$.
Let $\zeta \in \mathbb{C}\setminus [c_\flat ,\infty)$, where $c_\flat \geqslant 0$ is the common lower bound
 of the operators $\widetilde{B}_\varepsilon :=f^\varepsilon B_\varepsilon f^\varepsilon $ and $\widetilde{B}^0:=f_0B^0f_0$.
 Let $\varrho (\zeta)$ be given by \eqref{rho (zeta)}. Then for $0<\varepsilon\leqslant 1$
we have
\begin{align}
\Vert &(B_\varepsilon -\zeta Q_0^\varepsilon )^{-1}-(B^0-\zeta \overline{Q_0})^{-1}\Vert _{L_2(\mathbb{R}^d)\rightarrow L_2(\mathbb{R}^d)}\leqslant C_{11}\varrho (\zeta)\varepsilon ,\nonumber
\\
\label{9.16a}
\begin{split}
\Vert &
(B_\varepsilon -\zeta Q_0^\varepsilon )^{-1}-(B^0-\zeta \overline{Q_0})^{-1}-\varepsilon K^0(\varepsilon ;\zeta )
\Vert _{L_2(\mathbb{R}^d)\rightarrow H^1(\mathbb{R}^d)}\\
&\leqslant (C_{17}+ C_{18}+\vert \zeta +1\vert ^{1/2}C_{19})\varrho (\zeta )\varepsilon ,
\end{split}\\
\label{9.16b}
\begin{split}
\Vert & g^\varepsilon \nabla (B_\varepsilon -\zeta Q_0^\varepsilon )^{-1}-\left(\widetilde{g}^\varepsilon\nabla +g^\varepsilon (\nabla \widetilde{\Lambda})^\varepsilon\right) (B^0-\zeta \overline{Q_0})^{-1}\Vert _{L_2(\mathbb{R}^d)\rightarrow L_2(\mathbb{R}^d)}\\
&\leqslant (C_{20} +\vert \zeta +1\vert ^{1/2}C_{21})\varrho (\zeta )\varepsilon .
\end{split}
\end{align}
The constants $C_{11}$, $C_{17}$, $C_{18}$, $C_{19}$,  $C_{20}$, and $C_{21}$ depend only on the initial data
\eqref{problem data from sec. 9.1} and $c_\flat$. In the case where $Q_0$ is constant,
estimates {\rm \eqref{9.16a} and \eqref{9.16b}} are true with $C_{19}=C_{21}=0$.
\end{proposition}

\subsection{The periodic Schr\"odinger operator}
\label{Subsection periodic Schr?dinger operator}
 In $L_2(\mathbb{R}^d)$, we consider the operator $\check{\mathcal{A}}=\mathbf{D}^*\check{g}(\mathbf{x})\mathbf{D}+\check{v}(\mathbf{x})$,
 where $\check{g}(\mathbf{x})$ is a $\Gamma$-periodic symmetric \hbox{$(d\times d)$}-matrix-valued function with real entries such that
 $\check{g}(\mathbf{x})>0$, $\check{g}$, $\check{g}^{-1}\in L_\infty (\mathbb{R}^d)$; and $\check{v}(\mathbf{x})$ is a
 real-valued  $\Gamma $-periodic function such that
\begin{equation*}
\check{v}\in L_s(\Omega ),\quad s=1\ \text{for}\ d=1,\quad s>d/2\ \text{for}\ d\geqslant 2.
\end{equation*}
As usual, the precise definition of the operator $\check{\mathcal{A}}$ is given in terms of the quadratic form
\begin{equation}
\label{check a}
\check{\mathfrak{a}}[u,u]=\int _{\mathbb{R}^d}\left(
\langle \check{g}(\mathbf{x})\mathbf{D}u,\mathbf{D}u\rangle + \check{v}(\mathbf{x})\vert u\vert ^2
\right)
\,d\mathbf{x},\quad u\in  H^1(\mathbb{R}^d).
\end{equation}
Adding an appropriate constant to $\check{v}(\mathbf{x})$, {\it we may assume that the bottom of the spectrum
 of the operator $\check{\mathcal{A}}$ is the point zero.}
 Under this condition, the operator $\check{\mathcal{A}}$ admits a convenient factorization
 (see, e.~g., \cite[Chapter~6, Subsection~1.1]{BSu}).
 To describe this factorization, we consider the equation
 \begin{equation}
\label{omega equation}
\mathbf{D}^*\check{g}(\mathbf{x})\mathbf{D}\omega (\mathbf{x})+\check{v}(\mathbf{x})\omega (\mathbf{x})=0.
\end{equation}
This equation has a $\Gamma$-periodic solution $\omega\in\widetilde{H}^1(\Omega)$
defined up to a constant factor. This factor can be fixed so that $\omega (\mathbf{x})>0$ and
\begin{equation}
\label{omega condition}
\int _\Omega \omega ^2(\mathbf{x})\,d\mathbf{x}=\vert \Omega \vert .\end{equation}
Moreover, this solution is positive definite and bounded:
$0<\omega _0\leqslant \omega (\mathbf{x})\leqslant \omega _1 <\infty$,
and the norms $\Vert \omega \Vert _{L_\infty}$, $\Vert \omega ^{-1}\Vert _{L_\infty}$ are controlled in terms of
$\Vert \check{g}\Vert _{L_\infty}$, $\Vert \check{g}^{-1}\Vert _{L_\infty}$, and $\Vert \check{v}\Vert _{L_s(\Omega)}$.
The function $\omega$ is a multiplier in $H^1(\mathbb{R}^d)$, as well as in $\widetilde{H}^1(\Omega)$.
 After the substitution $u = \omega z$, the form \eqref{check a}  turns into
\begin{equation*}
\check{\mathfrak{a}}[u,u]=\int _{\mathbb{R}^d}\omega ^2(\mathbf{x})\langle \check{g}(\mathbf{x})\mathbf{D}z ,\mathbf{D}z \rangle \,d\mathbf{x},\quad u=\omega z ,\quad z \in H^1(\mathbb{R}^d).
\end{equation*}
Thus, the operator $\check{\mathcal{A}}$ admits the following factorization:
\begin{equation}
\label{check A =}
\check{\mathcal{A}}=\omega ^{-1}\mathbf{D}^*g\mathbf{D}\omega ^{-1},\quad g =\omega ^2 \check{g}.
\end{equation}

Now we consider the operator
\begin{equation}
\label{check A eps}
\check{\mathcal{A}}_\varepsilon =(\omega ^\varepsilon )^{-1}\mathbf{D}^* g^\varepsilon \mathbf{D}(\omega ^\varepsilon )^{-1}.
\end{equation}
In the initial terms, \eqref{check A eps} can be written as
\begin{equation}
\label{check Aeps=}
\check{\mathcal{A}}_\varepsilon =\mathbf{D}^*\check{g}^\varepsilon \mathbf{D}+\varepsilon ^{-2}\check{v}^\varepsilon.
\end{equation}
We emphasize that the expression \eqref{check Aeps=} involves the large factor $\varepsilon ^{-2}$ 
standing at the rapidly oscillating potential
$\check{v}^\varepsilon$. The operator $\check{\mathcal{A}}_\varepsilon $ can be viewed as the Schr\"odinger operator with
 the rapidly oscillating metric $\check{g}^\varepsilon$ and the strongly singular potential $\varepsilon ^{-2}\check{v}^\varepsilon$.

Next, as above, we assume that $\mathbf{A}=\mathrm{col}\,\lbrace A_1(\mathbf{x}),\dots ,A_d(\mathbf{x})\rbrace $, where
$A_j(\mathbf{x})$, $j=1,\dots,d$, are $\Gamma$-periodic real-valued functions satisfying \eqref{A_j cond}. Let $\widehat{v}(\mathbf{x})$ and $\check{\mathcal{V}}(\mathbf{x})$ be $\Gamma$-periodic real-valued functions such that
\begin{equation}
\label{hat v, check V}
\begin{split}
&\widehat{v}, \check{\mathcal{V}}\in L_s(\Omega),\quad s=1\ \text{for}\ d=1,\quad s>d/2\ \text{for}\ d\geqslant 2,\\
&\int _\Omega \widehat{v}(\mathbf{x})\omega ^2(\mathbf{x})\,d\mathbf{x}=0.
\end{split}
\end{equation}
Consider the operator $\check{\mathcal{B}}_\varepsilon$ given formally by the expression 
\begin{equation}
\label{check mathcal B_eps}
\check{\mathcal{B}}_\varepsilon =(\mathbf{D}-\mathbf{A}^\varepsilon )^*\check{g}^\varepsilon (\mathbf{D}-\mathbf{A}^\varepsilon )+\varepsilon ^{-2}\check{v}^\varepsilon +\varepsilon ^{-1}\widehat{v}^\varepsilon +\check{\mathcal{V}}^\varepsilon .
\end{equation}
(The precise definition is given in terms of the corresponding quadratic form.)
The operator $\check{\mathcal{B}}_\varepsilon$ can be treated as the Schr\"odinger operator with the metric $\check{g}^\varepsilon$,
the magnetic potential $\mathbf{A}^\varepsilon$, and the electric potential
$\varepsilon ^{-2}\check{v}^\varepsilon +\varepsilon ^{-1}\widehat{v}^\varepsilon +\check{\mathcal{V}}^\varepsilon $
containing the singular summands $\varepsilon ^{-2}\check{v}^\varepsilon$ and $\varepsilon ^{-1}\widehat{v}^\varepsilon $.

We put
\begin{equation}
\label{v =, V =}
v(\mathbf{x}):=\widehat{v}(\mathbf{x})\omega ^2(\mathbf{x}),\quad \mathcal{V}(\mathbf{x}):=\check{\mathcal{V}}(\mathbf{x})\omega ^2(\mathbf{x}).
\end{equation}
Using \eqref{check A eps} and \eqref{check Aeps=}, we see that
$\check{\mathcal{B}}_\varepsilon =(\omega ^\varepsilon )^{-1}\mathcal{B}_\varepsilon (\omega ^\varepsilon )^{-1}$,
where the operator $\mathcal{B}_\varepsilon$ is given by \eqref{B_eps from sec. 9.1}
with $g$ as in \eqref{check A =}, and $v$, $\mathcal{V}$ as in \eqref{v =, V =}.
From \eqref{hat v, check V} and the properties of $\omega$ it follows that the coefficients $v$ and $\mathcal{V}$
satisfy \eqref{v, V cond}.

Let $\check{Q}_0(\mathbf{x})$ be a $\Gamma$-periodic real-valued function;
we assume that $\check{Q}_0(\mathbf{x})$ is positive definite and bounded. We put
$Q_0(\mathbf{x}):=\check{Q}_0(\mathbf{x})\omega ^2(\mathbf{x})$. Denote $c_5=(c_0+c_4)\Vert Q_0^{-1}\Vert _{L_\infty}$,
where the constants $c_0$ and $c_4$ correspond to the operator $\mathcal{B}_\varepsilon$ described above.
The operators $\check{B}_\varepsilon :=\check{\mathcal{B}}_\varepsilon +c_5\check{Q}_0^\varepsilon$ and
$B_\varepsilon =\mathcal{B}_\varepsilon +c_5Q_0^\varepsilon$ satisfy the following relation:
$
\check{B}_\varepsilon =(\omega ^\varepsilon )^{-1}B_\varepsilon (\omega ^\varepsilon )^{-1}.
$
Obviously,
\begin{equation}
\label{10.22a}
(\check{B}_\varepsilon - \zeta \check{Q}_0^\varepsilon)^{-1}=\omega ^\varepsilon
(B_\varepsilon - \zeta Q_0^\varepsilon)^{-1} \omega ^\varepsilon.
\end{equation}
Now the initial data reduces to the following set
\begin{equation}
\label{problem data for periodic Schr?dinger}
\begin{split}
&d,\rho ,s ;\Vert \check{g}\Vert _{L_\infty}, \Vert \check{g}^{-1}\Vert _{L_\infty}, \Vert \mathbf{A}\Vert _{L_\rho (\Omega )}, \Vert \check{v}\Vert _{L_s(\Omega)}, \Vert \widehat{v}\Vert _{L_s(\Omega )}, \Vert \check{\mathcal{V}}\Vert _{L_s(\Omega )},\\
&\Vert \check{Q}_0\Vert _{L_\infty}, \Vert \check{Q}_0^{-1}\Vert _{L_\infty}; \text{the parameters of the lattice}\,\Gamma .
\end{split}
\end{equation}

Using \eqref{10.22a} and Proposition \ref{Proposition scalar Schr?dinger}, we obtain the following result.

\begin{proposition}
\label{Proposition periodic Schr?dinger}
Let $\check{\mathcal{B}}_\varepsilon$ be the operator \eqref{check mathcal B_eps} whose coefficients $\check{g}^\varepsilon$, $\mathbf{A}^\varepsilon$, $\widehat{v}^\varepsilon $, and $\check{\mathcal{V}}^\varepsilon $ satisfy
  the assumptions of Subsection~\textnormal{\ref{Subsection periodic Schr?dinger operator}}.
Let $\omega (\mathbf{x})$ be the $\Gamma$-periodic positive solution of equation \eqref{omega equation}
satisfying condition \eqref{omega condition}. Let $\mathcal{B}_\varepsilon$ be the operator \eqref{B_eps from sec. 9.1}
with the coefficients $g^\varepsilon =\check{g}^\varepsilon (\omega ^\varepsilon )^2$, $\mathbf{A}^\varepsilon$, $v^\varepsilon =\widehat{v}^\varepsilon (\omega ^\varepsilon )^2$, and $\mathcal{V}^\varepsilon =\check{\mathcal{V}}^\varepsilon (\omega ^\varepsilon )^2$.
Let $\check{Q}_0$ be a $\Gamma$-periodic bounded and positive definite function, and let
$Q_0:=\check{Q}_0\omega ^2$. Denote $c_5=(c_0+c_4)\Vert Q_0^{-1}\Vert _{L_\infty}$, where the constants $c_0$ and $c_4$
correspond to the operator $\mathcal{B}_\varepsilon$. Let $B_\varepsilon =\mathcal{B}_\varepsilon +c_5Q_0^\varepsilon$
 and $\check{B}_\varepsilon =\check{\mathcal{B}}_\varepsilon +c_5\check{Q}_0^\varepsilon$. Let $B^0$ be the effective operator for $B_\varepsilon$
defined by \eqref{B^0 from sec. 9.1}.
Let $\zeta \in \mathbb{C}\setminus \mathbb{R}_+$, $\zeta =\vert \zeta \vert e^{i\phi}$, $\phi \in (0,2\pi)$, and $\vert \zeta \vert \geqslant 1$.
Let $K^0(\varepsilon ;\zeta)$ be the corrector \eqref{K^0(eps;zeta)} for the operator~$B_\varepsilon$.
Then for $0<\varepsilon \leqslant 1$ we have
\begin{align}
\label{10.23a}
\Vert &(\check{B}_\varepsilon -\zeta \check{Q}_0^\varepsilon )^{-1}-\omega ^\varepsilon (B^0-\zeta\overline{Q_0})^{-1}\omega ^\varepsilon \Vert _{L_2(\mathbb{R}^d)\rightarrow L_2(\mathbb{R}^d)}
\leqslant C_4 \Vert \omega \Vert ^2_{L_\infty}c(\phi)^2\varepsilon \vert \zeta \vert ^{-1/2},\\
\label{10.23b}
\begin{split}
\Vert &(\omega ^\varepsilon )^{-1}(\check{B}_\varepsilon -\zeta \check{Q}_0^\varepsilon )^{-1}-(B^0-\zeta \overline{Q_0})^{-1}
\omega ^\varepsilon  -\varepsilon K^0(\varepsilon ;\zeta ) \omega ^\varepsilon
\Vert _{L_2(\mathbb{R}^d)\rightarrow H^1(\mathbb{R}^d)}\\
&\leqslant (C_{8}+ C_{9})\Vert \omega \Vert _{L_\infty}c(\phi)^2 \varepsilon ,
\end{split}
\\
\label{10.23c}
\begin{split}
\Vert &g^\varepsilon \nabla (\omega ^\varepsilon )^{-1}(\check{B}_\varepsilon -\zeta \check{Q}_0^\varepsilon )^{-1}-\left(
\widetilde{g}^\varepsilon \nabla +g^\varepsilon (\nabla \widetilde{\Lambda})^\varepsilon
\right)
(B^0-\zeta \overline{Q_0})^{-1}\omega ^\varepsilon \Vert _{L_2(\mathbb{R}^d)\rightarrow L_2(\mathbb{R}^d)}\\
&\leqslant C_{10}\Vert \omega \Vert _{L_\infty}c(\phi )^2\varepsilon .
\end{split}
\end{align}
Here $c(\phi)$ is given by \eqref{c(phi)}. The constants $C_4\Vert\omega \Vert ^2_{L_\infty}$, $(C_{8}+ C_{9})\Vert \omega \Vert _{L_\infty}$, and $C_{10}\Vert \omega \Vert _{L_\infty}$ depend only on the initial data \eqref{problem data for periodic Schr?dinger}.
\end{proposition}

\begin{proof}
Multiplying the operators under the norm sign in \eqref{10.13a} by $\omega^\varepsilon$ from both sides
and using \eqref{10.22a}, we arrive at \eqref{10.23a}.

By \eqref{10.22a}, $(\omega^\varepsilon)^{-1}(\check{B}_\varepsilon - \zeta \check{Q}_0^\varepsilon)^{-1}=
(B_\varepsilon - \zeta Q_0^\varepsilon)^{-1} \omega ^\varepsilon.$ Multiplying the operators under the norm sign in
\eqref{10.13b} by $\omega^\varepsilon$ from the right, we obtain \eqref{10.23b}.
Similarly, \eqref{10.13c} implies \eqref{10.23c}.
\end{proof}

In order to obtain approximation of the operator $(\check{B}_\varepsilon -\zeta \check{Q}_0^\varepsilon )^{-1}$
in a wider domain of the parameter $\zeta$, we use Proposition~\ref{scalar}.
By analogy with the proof of Proposition~\ref{Proposition periodic Schr?dinger}, it is easy to
check the following statement.

\begin{proposition}
\label{Proposition periodic Schr?dinger2}
Suppose that the assumptions of Proposition~\textnormal{\ref{Proposition periodic Schr?dinger}}
are satisfied.
Denote $f(\mathbf{x}) = Q_0(\mathbf{x})^{-1/2}$ and $f_0 = (\overline{Q_0})^{-1/2}$. Let
$\zeta \in \mathbb{C}\setminus [c_\flat ,\infty)$, where $c_\flat \geqslant 0$ is a common lower bound of the operators
$\widetilde{B}_\varepsilon :=f^\varepsilon B_\varepsilon f^\varepsilon $ and $\widetilde{B}^0:=f_0B^0f_0$.
Then for $0<\varepsilon \leqslant 1$ we have
\begin{align}
\Vert &(\check{B}_\varepsilon -\zeta \check{Q}_0^\varepsilon )^{-1}-\omega ^\varepsilon (B^0-\zeta\overline{Q_0})^{-1}\omega ^\varepsilon \Vert _{L_2(\mathbb{R}^d)\rightarrow L_2(\mathbb{R}^d)}
\leqslant C_{11} \Vert \omega \Vert ^2_{L_\infty}\varrho (\zeta )\varepsilon,
\nonumber
\\
\label{9.30a}
\begin{split}
\Vert &
(\omega ^\varepsilon )^{-1}(\check{B}_\varepsilon -\zeta \check{Q}_0^\varepsilon )^{-1}-(B^0-\zeta \overline{Q_0})^{-1} \omega ^\varepsilon  -\varepsilon K^0(\varepsilon ;\zeta )\omega ^\varepsilon
\Vert _{L_2(\mathbb{R}^d)\rightarrow H^1(\mathbb{R}^d)}\\
&\leqslant (C_{17}+ C_{18} + \vert \zeta +1\vert ^{1/2}C_{19})\Vert \omega \Vert _{L_\infty}\varrho (\zeta )\varepsilon,
\end{split}
\\
\label{9.30b}
\begin{split}
\Vert &g^\varepsilon \nabla (\omega ^\varepsilon )^{-1}(\check{B}_\varepsilon -\zeta \check{Q}_0^\varepsilon )^{-1}-\left(
\widetilde{g}^\varepsilon \nabla +g^\varepsilon (\nabla \widetilde{\Lambda})^\varepsilon
\right)
(B^0-\zeta \overline{Q_0})^{-1} \omega ^\varepsilon \Vert _{L_2(\mathbb{R}^d)\rightarrow L_2(\mathbb{R}^d)}\\
&\leqslant (C_{20}+\vert \zeta +1\vert ^{1/2}C_{21})\Vert \omega \Vert _{L_\infty}\varrho (\zeta)\varepsilon .
\end{split}
\end{align}
Here $\varrho (\zeta )$ is given by \eqref{rho (zeta)}. The constants $C_{11}\Vert \omega \Vert ^2_{L_\infty}$,
$(C_{17}+ C_{18})\Vert \omega \Vert _{L_\infty}$,  $C_{19}\Vert \omega \Vert _{L_\infty}$, $C_{20}\Vert \omega \Vert _{L_\infty}$,
 and $C_{21}\Vert \omega \Vert _{L_\infty}$ are controlled in terms of the initial data \eqref{problem data for periodic Schr?dinger} and $c_\flat$.
In the case where $Q_0$ is constant, estimates \eqref{9.30a} and \eqref{9.30b} are valid with $C_{19}= C_{21}=0$.
\end{proposition}

\begin{remark}
Propositions {\rm \ref{Proposition periodic Schr?dinger}} and {\rm \ref{Proposition periodic Schr?dinger2}} demonstrate that for the operator
 \eqref{check mathcal B_eps} the nature of the results is different from the results for the operator \eqref{B_eps from sec. 9.1}.
Because of the presence of the strongly singular potential $\varepsilon^{-2} \check{v}^\varepsilon$,
the generalized resolvent  \hbox{$(\check{B}_\varepsilon -\zeta \check{Q}_0^\varepsilon )^{-1}$}
has no limit in the $L_2({\mathbb R}^d)$-operator norm{\rm ;} it is approximated by the
generalized resolvent $(B^0 - \zeta \overline{Q_0})^{-1}$ sandwiched between the rapidly oscillating factors
$\omega^\varepsilon$.
\end{remark}

\subsection{The two-dimensional Pauli operator} Note that it is also possible to apply
the general results to the two-dimensional periodic Pauli operator with a singular magnetic potential perturbed by
a singular electric potential. This operator was considered in \cite[\S 14]{SuAA} by using a convenient
factorization for the Pauli operator.
The nature of the results for this operator is similar to Propositions {\rm \ref{Proposition periodic Schr?dinger}} and
{\rm \ref{Proposition periodic Schr?dinger2}}. We will not dwell on this.

\end{document}